\documentclass{article}

\usepackage{amsmath,amsthm,amsfonts,amsbsy,amssymb}
\usepackage[initials]{amsrefs}
\newtheorem{theorem}{Theorem}
\newtheorem{proposition}{Proposition}
\newtheorem{corollary}{Corollary}
\newtheorem{lemma}{Lemma}

\newcommand{\R}{{\mathbb R}}
\newcommand{\C}{{\mathbb C}}
\newcommand{\Z}{{\mathbb Z}}
\newcommand{\set}[2]{ \left\{ #1 \ \left| \ #2 \right. \right\}}
\newcommand{\ang}[1]{\left< #1 \right>}

\newcommand{\dH}{d {\mathcal H}}

\newcommand{\li}[1]{{}^{#1}\Sigma}
\newcommand{\ri}[1]{\Sigma^{#1}}
\newcommand{\lx}[1]{{}^{#1}}

\newcommand{\gen}{\operatorname{gen}}
\newcommand{\nr}{{\mathcal N}_{\mathcal R}}
\newcommand{\jac}[1]{\varphi^{#1}}

\newcommand{\en}{n_1}
\newcommand{\de}{d_1}

\newcommand{\wt}{W}

\newcommand{\tup}[2]{{\underline{#1}}_{|#2}}

\usepackage[dvipsnames]{xcolor}

\title{Local curvature of maximally nondegenerate Radon-like transforms}
\author{Philip T. Gressman\footnote{Partially supported by NSF grant DMS-2054602.}}
\date{\today}

\begin{document}
\maketitle

\begin{abstract}
This paper gives a complete geometric characterization in all dimensions and codimensions of those Radon-like transforms which, up to endpoints, satisfy the largest possible range of local $L^p \rightarrow L^q$ inequalities permitted by quadratic-type scaling. The necessary and sufficient curvature-type criterion is phrased in terms of an associated Newton-like diagram. In the case of averages over families of polynomial graphs, the curvature condition implies sharp endpoint estimates as well. The proof relies on the recently-developed multilinear Radon-Brascamp-Lieb testing criterion \cite{testingcond} and a refined version of differential inequalities for polynomials first appearing in work on the Oberlin affine curvature condition \cite{gressman2019}.
\end{abstract}

\tableofcontents

\section{Introduction}
\subsection{Formulation and main theorem}

The problem of establishing optimal or near-optimal $L^p$-improving inequalities for Radon-like transforms is one with a rich history going back to the early 1970s with work of Strichartz \cite{strichartz1970} and Littman \cite{littman1971}. Until the 1990s, most work on this problem relied heavily on Fourier and oscillatory integral methods; a key achievement of this approach was the identification of the Phong-Stein rotational curvature condition \cite{ps1986I} (which was itself informed by and equivalent to a nondegeneracy condition for Fourier Integral Operators formulated by Guillemin and Sternberg \cite{gs1977}*{Chapter VI, Section 6}). By considering the special case of convolution operators, it is clear that nonvanishing of the Phong-Stein rotational curvature sharply characterizes those Radon-like transforms which exhibit a maximal degree of smoothing in the scale of $L^2$ Sobolev spaces even in higher codimension (see Seeger and Wainger \cite{sw2003}). In this sense, it has been understood for essentially four decades what makes a Radon-like transform a ``best possible'' one in the sense of $L^2$ Sobolev improvement. Such methods have in some important situations (most notably, for averages over hypersurfaces) also yielded optimal $L^p$-improving inequalities for Radon-like operators, but the approach tends not to generalize, especially when averaging over submanifolds of intermediate dimension and codimension, which are geometrically unlike the endpoint cases of curves and hypersurfaces.

This paper provides an answer to the basic question of what geometric property is shared by those Radon-like transforms which exhibit ``best possible'' regularity improvement in the scale of $L^p$ spaces. While there have been several explorations of this class of operators over the past two decades, most notably by Ricci \cite{ricci1997} and D.~Oberlin \cite{oberlin2008} (see also \cite{gressman2019II}), there has until now been no coherent way to describe the local geometry of such objects. Previous approaches have also generally been limited to special combinations of dimension and codimension. Theorem \ref{characterthm}, by contrast, provides a sharp characterization of the so-called model operator geometry in all relevant dimensions and codimensions in terms of a local quantity computed using only second-order derivatives of the family of submanifolds. The condition involves computing a nontraditional sort of Newton diagram, which makes its study somewhat challenging. For this reason, there are also a number of results (e.g., Lemma \ref{newtonlemma}) dedicated to understanding basic properties like stability of nondegeneracy.

The formulation to be used here is as follows.
Let $\de ,n$, and $k$ be positive integers such that $n > k$. Suppose that $U$ is an open subset of $\R^n \times \R^{\de}$, that $\phi(x,t)$ is a smooth function of $(x,t) \in U$ with values in $\R^k$ and that $\gamma_t (x) := (t, \phi(x,t))$. Let $\en  := \de  + k$ and $d := n - k$ and consider the Radon-like transform
\begin{equation}
T f(x) := \int_{\R^{\de }} f(\gamma_t(x)) \eta(x,\gamma_t(x)) dt, \label{theop}
\end{equation}
which is well-defined \textit{a priori} for all nonnegative Borel-measurable functions $f$ on $\R^{\en}$. Here $\eta(x,y)$ serves as a continuous cutoff function which will be restricted so that the support of $\eta(x,\gamma_t(x))$ is compact and contains some distinguished point $(x_*,t_*) \in U$. It will also be assumed that the Jacobian matrix $D_x \phi$ (i.e., the matrix of first partial derivatives of $\phi$ with respect to the $x$ variables) has full rank $k$ for every $(x,t)$ belonging to the support of $\eta(x,\gamma_t(x))$. Although the particular form of $\gamma_t(x)$ as the graph of $\phi(x,t)$ is not one commonly adopted elsewhere in the literature, it is no serious limitation; any smooth mapping $\tilde \gamma_t(x)$ defined locally on some neighborhood of $(x_*,t_*) \in \R^n \times \R^{\de}$ and having values in $\R^{\en}$ can locally be written as $(t,\tilde \phi(x,t))$ after suitable changes of coordinates and a change of variables in $t$ so long as $t \mapsto \tilde \gamma_t(x)$ parametrizes a $\de$-dimensional submanifold of $\R^{\en}$ for fixed $x$. 

When studying the $L^p \rightarrow L^q$ mapping properties of such Radon-like transforms, a particular pair of exponents arise via Knapp-type examples as the best-possible $p$ and $q$ for any specific values of $n,k$, and $\de $, namely:
\begin{equation}
p_b := \frac{k d}{n \de } + 1 \text{ and } q_b := \frac{\en  d }{k \de } + 1. \label{bestexponents}
\end{equation}
For reference, the H\"{o}lder dual exponents are exactly
\begin{equation} p'_b = \frac{n \de  }{k d} + 1 \text{ and } q'_b = \frac{k \de }{\en d} +1. \label{bestdual} \end{equation}
In the special case when $n = \en$ and $k = 1$, $p_b = (n+1)/n$ and $q_b = n+1$; it has long been understood that the operator \eqref{theop} maps $L^{(n+1)/n}(\R^n)$ to $L^{n+1}(\R^n)$ precisely when the family of submanifolds indexed by $x \in \R^n$ and parametrized by $\gamma_t(x)$ for $t \in \R^{\de}$ for each fixed $x$ exhibits nonzero rotational curvature in the sense of Phong and Stein. However, when $k > 1$, nonvanishing rotational curvature is sufficient but not generally necessary.

The key local nondegeneracy condition which governs boundedness of \eqref{theop} at (and also near) the exponents \eqref{bestexponents} can be described in terms of a trilinear curvature functional.
Since it is assumed that Jacobian matrix $D_x \phi$ (arranged so that rows correspond to the coordinates of $\phi$ and columns to the coordinates of $x$)
 is rank $k$ at $(x,t)$, its kernel must always be $d$-dimensional; at any point $(x,t)$, let $z_1,\ldots,z_d$ be any orthonormal vectors in $\R^n$ which span the kernel of $D_x \phi$ there. For $i \in \{1,\ldots,\de\}$, $i' \in \{1,\ldots,k\}$, and $i'' \in \{1,\ldots,d\}$, let
 \[ Q_{ii'i''} := \sum_{\ell=1}^{n} z_{i''}^\ell \frac{\partial^2 \phi^{i'}}{\partial t^i \partial x^\ell } (x,t) \]
(upper indices as in $p^{i'}$ and $z_{i''}^\ell$ will be used represent the coordinates in the standard bases, e.g, $\phi := (\phi^1,\ldots,\phi^k)$ and $z_{i'} := (z^1_{i'},\ldots,z^{n}_{i'})$); to these coefficients we associate a trilinear functional $Q : \R^{\de } \times \R^k \times \R^d \rightarrow \R$ by means of the formula
\begin{equation} Q(u,v,w) := \sum_{i=1}^{\de } \sum_{i'=1}^k \sum_{i''=1}^d Q_{ii'i''} u^i v^{i'} w^{i''} \label{qdef}
\end{equation}
for all $u \in \R^{\de }$, $v \in \R^k$, and $w \in \R^d$. 
There is a key geometric object taking the form of a Newton-type diagram associated to $Q$ which, as it turns out, contains the essential information about whether \eqref{theop} satisfies a best-possible $L^{p_b} \rightarrow L^{q_b}$ inequality. (Note that the broader observation that $Q$ plays \textit{some} role in quantifying boundedness is not new; indeed, nonvanishing rotational curvature corresponds exactly to the situation when $Q(\cdot,v,\cdot)$ is a nondegenerate bilinear form on $\R^d$ for all $v \neq 0$.) To define it, a simple auxiliary definition is required.
Given a multiindex $\beta \in \Z_{\geq 0}^k$ and a sequence $\mathcal I := (i_1,\ldots,i_s)$ of integers belonging to $\{1,\ldots,k\}$, it will be said that $\beta$ counts $\mathcal I$ when for each $\ell \in \{1,\ldots,k\}$, there are exactly $\beta_\ell$ values of the index $j \in \{1,\ldots,s\}$ such that $i_j = \ell$; in other words, for each $\ell$, $\beta_\ell$ is simply the number of elements of the sequence $\mathcal I$ which equal $\ell$.

Given $Q$ as in \eqref{qdef},  let $N(Q)$ denote the convex hull in $[0,\infty)^{\de  + k + d}$ of the collection of all triples $(\alpha,\beta,\gamma) \in \Z_{\geq 0}^{\de } \times \Z_{\geq 0}^{k} \times \Z_{\geq 0}^d$ with $|\alpha| = |\beta| = |\gamma| \leq \min\{d,k\}$ (where $\alpha,\beta$, and $\gamma$ are regarded as multiindices) for which either $(\alpha,\beta,\gamma) = (0,0,0)$ or for which there exist $\mathcal I := (i_1,\ldots,i_s) \subset \{1,\ldots,k\}^s$ and $\mathcal J := (j_1,\ldots,j_s) \subset \{1,\ldots,d\}^s$ such that $\beta$ counts $\mathcal I$, $\gamma$ counts $\mathcal J$, and
\begin{equation} \left. \partial_\tau^{\alpha} \right|_{\tau=0} \det \begin{bmatrix} Q(\tau,e_{i_1}, e_{j_1}') & \cdots & Q(\tau,e_{i_1}, e_{j_s}') \\ \vdots & \ddots & \vdots \\ Q(\tau,e_{i_s}, e_{j_1}') & \cdots & Q(\tau,e_{i_s},e_{j_s'}) \end{bmatrix} \neq 0, \label{qdetdef} \end{equation}
where $\{e_i\}_{i=1}^k$ is the standard basis of $\R^k$,$\{e_j'\}_{j=1}^d$ is the standard basis of $\R^d$, and $\tau \in \R^{\de}$.
Then let
\begin{equation}
\begin{split}
\nr(Q) := \bigcap \Big\{ & N(Q') \ | \  Q'(u,v,q) = Q(O_1 u, O_2 v, O_3 w)  \text{ for orthogonal} \\ & \text{matrices } O_1, O_2, O_3 \text{ and all } u \in \R^{\de }, v \in \R^k, w \in \R^{d} \Big\}
\end{split} \label{nr0def}
\end{equation}
(i.e., $\nr(Q)$ is the intersection of all such $N(Q')$).
The functional $Q$ will be called nondegenerate when the point
\begin{equation} \Big(\overbrace{\frac{dk}{\de n},\ldots,\frac{dk}{\de n}}^{\de  \text{ copies}},\overbrace{\frac{d}{n},\ldots,\frac{d}{n}}^{k \text{ copies}},\overbrace{\frac{k}{n},\ldots,\frac{k}{n}}^{d \text{ copies}} \Big) \label{diagonal} \end{equation}
belongs to $\nr(Q)$. Any $Q$ for which \eqref{diagonal} does not belong to $\nr(Q)$ is called degenerate. The main result of this paper is as follows.
\begin{theorem}
Consider the transform $T$ given by \eqref{theop}. Let $(x_*,y_*) \in \R^n \times \R^{\en }$ have the property that $y_* = \gamma_{t_*}(x_*)$ for some $(x_*,t_*) \in U$, and suppose that the Jacobian matrix $D_x \phi$ is rank $k$ at $(x_*,t_*)$. Let $Q$ be the trilinear functional given by \eqref{qdef} at the point $(x_*,t_*)$.  Let $\Delta \subset [0,1]^2$ be the closed triangle with vertices $(0,0), (1,1)$ and $(1/p_b,1/q_b)$. \label{characterthm} 
\begin{enumerate}
\item If $Q$ is nondegenerate, and $\phi$ is a polynomial in $x$ and $t$, then there exists an $\eta$ of compact support which is nonvanishing at $(x_*,y_*)$ such that \eqref{theop} maps $L^{p_b}$ to $L^{q_b}$. By interpolation, $T$ maps $L^p(\R^{\en})$ to $L^q(\R^n)$ for all points $(1/p,1/q)$ belonging to the triangle $\Delta$.
\item If $Q$ is nondegenerate and $\phi$ is merely a smooth function of $x$ and $t$, then there exists an $\eta$ of compact support which is nonvanishing at $(x_*,y_*)$ such that \eqref{theop} maps $L^p(\R^{\en})$ to $L^q(\R^{n})$ for all pairs $(1/p,1/q)$ belonging to the interior of the triangle $\Delta$ (note that $T$ also trivially maps $L^p$ to itself for all $p \in [1,\infty]$).
\item If $Q$ is degenerate at and $\eta(x_*,y_*) \neq 0$, then \eqref{theop} fails to be bounded from $L^{p}(\R^{\en})$ to $L^{q}(\R^n)$ for all pairs $(1/p,1/q)$ belonging to some neighborhood of $(1/p_b,1/q_b)$. This neighborhood may be taken to depend only on $Q$.
\end{enumerate}
In short, for smooth $\phi$, nondegeneracy of $Q$ is necessary and sufficient for $L^{p} \rightarrow L^{q}$ boundedness for some set of pairs $(1/p,1/q)$ having $(1/p_b,1/q_b)$ in its closure. 
\end{theorem}
While it seems likely that nondegeneracy of $Q$ is both necessary and sufficient for full $L^{p_b} \rightarrow L^{q_b}$ boundedness of \eqref{theop} even in the smooth case, the robust characterization of boundedness of \eqref{theop} that appears in \cite{testingcond} relies very heavily on the algebraic properties of polynomials. Consequently, resolving the endpoint question in the smooth case of Theorem \ref{characterthm} will likely require a substantially different approach than the one used here.

\subsection{Summary and outline}

The proof of Theorem \ref{characterthm} includes three major ingredients. The first is the recent characterization, appearing in \cite{testingcond}, of boundedness of certain multilinear Radon-like transforms on an important scaling line which, crucially, passes through $(1/p_b,1/q_b)$.  The key result from \cite{testingcond}, which will be used as a black box here, is recorded for convenience in Section \ref{importsec}. This is one of two points in the argument at which there are important distinctions between polynomial and nonpolynomial mappings.

Although the current paper does directly not rely on any external results from the field of Geometric Invariant Theory, the overall strategy of the proof of Theorem \ref{characterthm} was heavily influenced by ideas originating from that area and there are a number of existing results and tools which could have been used here. Readers familiar with GIT will recognize the connection of \eqref{hilbertmumford} to the Hilbert-Mumford criterion. Most of the GIT-inspired work is contained in Section \ref{linalgsec} and the connections are strongest in Section \ref{nondegensec}, the centerpiece of which is Lemma \ref{newtonlemma} concerning quantitative characterizations of nondegeneracy.

The second major component of the proof of Theorem \ref{characterthm} is a modification of certain tools first appearing in \cite{gressman2019}, developed to study D.~Oberlin's affine Hausdorff measure and related curvature condition \cite{oberlin2003}. 
These tools are the subject of Theorem \ref{mainineq}, which appears and is proved in Section \ref{gdineqsec}. The resulting inequalities can, in some sense, be understood as a geometric generalization of the more common notion of polynomial type functions (see \cite{ps1994II}) which bound derivatives of ``nice'' functions in terms of simple scale factors and the supremum of the functions themselves.
Theorem \ref{mainineq} is an update of the more general differential inequalities from \cite{gressman2019} and features two main innovations: the dependence of various constants on degrees of various polynomials involved is made explicit (a necessary component of the passage to smooth functions), and it is posed in the category of Nash functions, which are in some sense the most general functions for which a quantitative version of Theorem \ref{mainineq} can hold. These results and their proofs are recorded in Section \ref{gdineqsec}. The proof that nondegeneracy implies boundedness for polynomial mappings is then completed in Section \ref{suffsec}.

Necessity of nondegeneracy is established in Section \ref{nesssec}. The strategy there is essentially a very careful Knapp-type analysis of testing on suitable characteristic functions. In particular, it is established in Section \ref{tomodelsec} that every Radon-like transform which satisfies some $L^p \rightarrow L^q$ which is nearly best-possible is closely related to a simpler operator for which its $\phi$ is a bilinear function of $x$ and $t$. Section \ref{modelsec} then establishes that degeneracy in the model case precludes boundedness near the best-possible exponents.

The third major component of the proof of Theorem \ref{characterthm} is contained in Section \ref{npasec} and concerns the passage from polynomial $\phi$ to smooth $\phi$, which is accomplished via Jackson's theorem and approximation. The same sort of idea appears in earlier work of Bourgain and Guth \cite{bg2011} and Zahl \cite{zahl2012}, though the precise contexts are rather different. In the case of the present paper, the approximation is accomplished via a simple Littlewood-Paley decomposition of $T$ which allows one to make precise the rough idea that when $T$ acts on functions of scale $2^{-j}$, one is essentially free to perturb the mapping $\gamma_t(x)$ within an error of $2^{-j}$ as well. Summing the Littlewood-Paley pieces at the endpoint $L^{p_b} \rightarrow L^{q_b}$ is essentially impossible because polynomial approximation is accompanied by operator norms which diverge to infinity as the scales become increasingly fine. However, it is possible to sum the pieces with only infinitesimal loss at nearby pairs of exponents by virtue of a general $L^2$ Sobolev inequality appearing in the seminal work of Christ, Nagel, Stein, and Wainger \cite{cnsw1999}.

\section{Notation and results from \cite{testingcond}}
\subsection{Notational conventions}
As mentioned in the introduction, superscripts are generally reserved for coordinates of vectors in the standard bases, e.g., $x := (x^1,\ldots,x^n)$ when $x \in \R^n$. Throughout the proofs that follow, standard multiindex notation is also used extensively: given $\alpha := (\alpha^1,\ldots,\alpha^n) \in \Z_{\geq 0}^n$, one defines $|\alpha| := \alpha^1 + \cdots + \alpha^n$, $\alpha! := \alpha^1! \cdots \alpha^n!$, $x^\alpha := (x^1)^{\alpha^1} \cdots (x^n)^{\alpha^n}$ for any $x := (x^1,\ldots,x^n) \in \R^n$, and \[\partial^\alpha f(x) := \frac{\partial^{|\alpha|}f}{ \partial (x^1)^{\alpha_1} \cdots (x^n)^{\alpha^n}}(x) \]
when $x \in \R^n$.

A significant portion of the labor to come also involves working with highly multilinear objects. To vastly simplify the notation, the following nonstandard summation convention will be used. When $F$ is any real-valued quantity that depends on a $k$-tuple of objects, we make the definition that
\begin{equation} \sum_{i = 1}^{n} F(\tup{\omega_i}{k}) := \sum_{i_1=1}^n \cdots \sum_{i_k=1}^n F(\omega_{i_1},\ldots,\omega_{i_k}). \label{tupdef} \end{equation}
In other words, when there is an underlined expression involving a variable of summation, the summation variable should be replaced by a $k$-tuple of summations (or whatever length of tuple is indicated to the right of the underline) and the underlined expression should be understood as a stand-in for a sum over a $k$-tuple where each entry of the tuple depends on its own summation index.

The notation $A \lesssim B$ will used to indicate the existence of a finite positive constant $C$ such that $A \leq C B$ uniformly over the parameters of $A$ and $B$ except as noted; likewise $A \approx B$ will indicate that $A \lesssim B$ and $B \lesssim A$.

\subsection{Testing conditions for Radon-like operators}
\label{importsec}
This section reviews key results from \cite{testingcond} which will be used frequently.

Suppose $\Omega \subset \R^{n} \times \R^{\en }$ is an open set whose points have the form $(x,y)$ for $x \in \R^n$ and $y \in \R^{\en }$. Let $\pi : \Omega \rightarrow \R^{k}$ for some $k \leq \min\{n,\en \}$ and suppose that $\pi$ is smooth.  The symbol $D_x \pi$ will denote the Jacobian matrix of $\pi$ with respect to $x$ only, i.e., $D_x \pi |_{(x,y)} := {\partial \pi}/{\partial x} |_{(x,y)}$ with rows of $D_x \pi$ corresponding to entries of $\pi$ and columns corresponding to the directions of differentiation, and similarly for $D_y \pi$. The object $d_x \pi$ is an alternating $k$-linear functional on $\R^n$ given by
\begin{equation}
d_x \pi |_{(x,y)} (v_1,\ldots,v_k) := \det \left[ D_x \pi |_{(x,y)} v_1, \ldots, D_x \pi |_{(x,y)} v_k \right],
\end{equation}
and analogously for $d_y \pi$, which is an alternating $k$-linear functional on $\R^{\en }$. Given any collection $\omega$ of vectors $\omega_1,\ldots,\omega_n \in \R^n$, 
\begin{equation} \begin{split} ||d_x \pi(x,y)||_\omega & := \left[ \frac{1}{k!} \sum_{i=1}^n \left| d_x \pi (\tup{\omega_i}{k}) \right|^2 \right]^\frac{1}{2} \\ & = \left[ \frac{1}{k!} \sum_{i_1=1}^n \cdots \sum_{i_k=1}^n \left| d_x \pi|_{(x,y)} ( \omega_{i_1},\ldots,\omega_{i_k}) \right|^2 \right]^\frac{1}{2} \end{split} \label{typicalsum} \end{equation}
and likewise set
\[ ||d_y \pi(x,y)||_\omega := \left[ \frac{1}{k!} \sum_{i=1}^{\en} \left|  d_y \pi |_{(x,y)} (\tup{\omega_i}{k}) \right|^2 \right]^\frac{1}{2} \]
for any $\en$-tuple of vectors $\omega_1,\ldots,\omega_{\en}$ in $\R^{\en}$. The notation $||d_x \pi(x,y)||$ and $||d_y \pi(x,y)||$ indicates that $\omega$ should be taken to be the tuple of standard basis vectors.

Any triple $(\Omega,\pi,\Sigma)$ is called a smooth incidence relation on $\R^{n} \times \R^{\en}$ of codimension $k$ when $\Omega \subset \R^{n} \times \R^{\en }$ is open, $\pi : \Omega \rightarrow \R^{k}$ is smooth, and
\[ \Sigma = \set{ (x,y) \in \Omega}{ \pi(x,y) = 0, ||d_x \pi(x,y)||, ||d_y \pi(x,y)|| > 0}. \]
The notation $\li{x}$ and $\ri{y}$ indicates slices of $\Sigma$ with fixed $x$ and $y$, respectively:
\begin{align*}
\li{x} & := \set{y \in \R^{\en }}{ (x,y) \in \Sigma} \text{ and }
\ri{y}  := \set{x \in \R^{n\vphantom{'}}}{(x,y) \in \Sigma},
\end{align*}
and on each slice $\li{x}$ and $\ri{y}$, $\sigma$ denotes what is called coarea measure, given by the formulas
\[ \int_{\li{x}} \! f d \sigma := \int_{\li{x}} \! \! f(y) \frac{\dH^{\de }(y)}{||d_y \pi (x,y)||} \text{ and } \int_{\ri{y}} \! f d \sigma := \int_{\ri{y}} \! \! f(x) \frac{\dH^{d}(x)}{||d_x \pi (x,y)||}, \]
where $\de  = \en  -k$, $d = n-k$, and $\dH^{s}$ is the usual $s$-dimensional Hausdorff measure.  

The main result needed from \cite{testingcond} is the following, which is the linear case of \cite{testingcond}*{Theorem 4}.
\begin{theorem}[cf. Theorem 4 of \cite{testingcond}] \label{fromtest}
Let $(\Omega,\pi,\Sigma)$ be a smooth incidence relation on $\R^n \times \R^{\en }$ and suppose that $\pi : \R^n \times \R^{\en } \rightarrow \R^k$ is polynomial. The Radon-like transform
\[ T f(x) := \int_{\li{x}} f(y) \eta(x,y) d \sigma(y) \]
for continuous $\eta$ satisfies an $L^p \rightarrow L^{pn/k}$ inequality whenever there exists a finite constant $C$ such that
\begin{equation} \int_{\li{x}} \frac{|\eta(x,y)|^{p'}d \sigma(y)}{||d_x \pi(x,y)||_\omega^{p'-1}}   \leq C \label{boundthisnotweak} \end{equation}
for all $x \in \R^n$, all $n$-tuples $\{\omega_i\}_{i=1}^n$ with $|\det (\omega_1,\ldots,\omega_n)| = 1$, and all $\epsilon > 0$.
The norm bound for $T$ is at most some constant depending on $n,\en$, and $k$ times some powers (depending on $n,\en,$ and $k$) of $C$ and the product of degrees of the polynomials $\pi^i$, $i \in \{1,\ldots,k\}.$
\end{theorem}

A few small but useful reductions are in order when working specifically within the context of \eqref{theop}. The first is that one can fix the defining function $\pi$ once and for all in terms of $\phi$ so that $\pi := (\pi^1,\ldots,\pi^k)$ for
\begin{equation}
\pi^j(x,y) := -y^{\de +j} + \phi^j (x,(y^1,\ldots,y^{\de })), \qquad j \in \{1,\ldots,k\}, \label{definingfn}
\end{equation}
and any pair $(x,y) \in \R^n \times \R^{\en}$ such that $(x,(y^1,\ldots,y^{\de})) \in U$. The second is that the coarea measure is, for this definition of $\pi$, simply equal to Lebesgue measure in $y^1,\ldots,y^{\de}$:
\begin{proposition}
Let $\pi$ be the defining function \eqref{definingfn}. \label{coareaflat}
The coarea measure $d\sigma$ on the submanifolds $\li{x}$ is exactly Lebesgue measure $dt$ when $\li{x}$ is parametrized by $\gamma_t(x)$ (as defined before \eqref{theop})
for $t \in \R^{\de }$, i.e.,
\begin{equation} \int_{\li{x}} f d \sigma = \int_{\lx{x}U} f(\gamma_t(x)) \, dt  = \label{intdef} \int_{\lx{x}U} f(t,\phi(x,t)) \, dt\end{equation}
for all nonnegative Borel-measurable functions $f$ on $\R^{\en }$,
where $\lx{x}{U}$ is the open subset of those $t \in \R^{\de }$ on which $D_x \phi(x,t)$ is full rank.
\end{proposition}
\begin{proof}
This is a minor variation of Proposition 5 from \cite{gressman2021}. 
By definition of the coarea measure and the incidence relation $\Sigma$, it suffices to show that $dt = \dH^{\de }(y) / ||d_y \pi(x,y)||$ at every point $y \in \li{x}$ for every $x$. This is because the map $t \mapsto \gamma_t(x)$ clearly parametrizes the level set $\pi(x,\cdot) = 0$ for each $x$. Hausdorff measure $\dH^{\de }$ on the graph of $\gamma_t$ is always equal to
\[ \sqrt{  \det \left( D_t \gamma_t(x) \right)^T \left(D_t \gamma_t(x) \right) } dt \]
when $D_t \gamma_t(x)$ is regarded as a $(\de  + k) \times \de $ matrix. Since $||d_y \pi(x,y)|| =  \sqrt{\det (D_y \pi(x,y)) (D_y \pi(x,y))^T}$ (by virtue of Proposition 1 of \cite{testingcond}), one only needs to show that
\begin{equation}  \det \left( D_t \gamma_t(x) \right)^T  \left(D_t \gamma_t(x) \right)  = \det (D_y \pi(x,y)) (D_y \pi(x,y))^T \label{matcheddet} \end{equation}
when $(y^1,\ldots,y^{\de}) = (t^1,\ldots,t^{\de})$.
Let $P$ be the $k \times \de $ matrix given by 
\[ P := \begin{bmatrix} \frac{\partial \phi^1}{\partial t^1}(x,t) & \cdots & \frac{\partial \phi^1}{\partial t^{\de }}(x,t) \\
\vdots & \ddots & \vdots \\
\frac{\partial \phi^k}{\partial t^1}(x,t) & \cdots & \frac{\partial \phi^k}{\partial t^{\de }}(x,t)
\end{bmatrix}.
\]
This matrix $P$ appears as a submatrix (on the left) of $D_y \pi(x,y)$ and a submatrix (on the bottom) of $D_t \gamma_t(x)$. The remaining entries are the negative of a $k \times k$ identity block (for $D_y \pi(x,y)$) and a $\de  \times \de $ identity block (for $D_t \gamma_t(x)$). Consequently
\[ \det (D_y \pi) (D_y \pi)^T = \det ( P P^T + I_{k \times k}) \]
and
\[\det \left( D_t \gamma_t(x) \right)^T \left(D_t \gamma_t(x) \right) = \det (I_{\de  \times \de } + P^T P). \]
By the Singular Value Decomposition, there exists a $k \times \de $ diagonal matrix $D$ and orthogonal matrices $O_1, O_2$ of size $k \times k$ and $\de  \times \de $, respectively, such that $P = O_1 D O_2$. Then
\[ P P^T + I_{k \times k} = O_1 \left( D D^T + I_{k \times k} \right) O_1 \]
and
\[ I_{\de  \times \de } + P^T P = O_2 \left( D^T D + I_{\de  \times \de } \right) O_2, \]
so
\[ \det (D_y \pi ) (D_y \pi)^T = \det \left( D D^T + I_{k \times k} \right) \]
and
\[ \det \left( D_t \gamma_t(x) \right)^T  \left(D_t \gamma_t(x) \right)  = \det (I_{\de  \times \de } + D^T D). \]
Now the nonzero diagonal entries of $D D^T$ and $D^T D$ must simply equal the square of the corresponding nonzero diagonal entries of $D$, so both $\det (I_{\de  \times \de } + D^T D)$ and $\det (D D^T + I_{k \times k})$ equal the product $\prod_j (1 + \sigma_j^2)$, where $\sigma_j$ is the $j$-th diagonal entry of $D$. Thus \eqref{matcheddet} holds.

To conclude, a remark on the set $\lx{x}{U}$ is needed. Observe that the definition of $\pi$ guarantees that
\[ D_x \pi(x,y) = D_x \phi(x,t) \]
when  $(y^1,\ldots,y^{\de}) = (t^1,\ldots,t^{\de})$. The definition of $\pi$ also trivially guarantees that $D_y \pi(x,y)$ is always full rank, so the point $(x,y)$ will belong to $\li{x}$ when $D_x \phi(x,t)$ is full rank. Hence the set of $y \in \li{x}$ corresponds exactly to those $t \in \lx{x}{U}$.
\end{proof}
As it will assumed that $D_x \phi(x,t)$ is full rank at all points of interest, Proposition \ref{coareaflat} guarantees that the definition \eqref{theop} agrees with the definition of Radon-like transforms as it appeared in Theorem \ref{fromtest}.

\section{Linear and multilinear algebra}
\label{linalgsec}
\subsection{Basis replacement results}
\label{basissec}

A fundamental requirement of the condition \eqref{boundthisnotweak} from Theorem \ref{fromtest} is that one must prove a uniform bound over all bases $\{\omega_i\}_{i=1}^n$ of $\R^n$ with volume $1$. It is perhaps not too difficult to imagine that proving a uniform bound over such a large class of parameters is tedious and can easily obscure the important features of the problem. To that end, it is useful to observe that the expression \eqref{typicalsum}, which is a sum over basis elements of some squared quantity, exhibits a large symmetry group which allows one to change the basis $\{\omega_i\}_{i=1}^n$ without changing the value of $||d_x \pi(x,y)||_\omega$ itself. This section deals with a number of what can ultimately be understood as ``basis replacement results'' once the proof of Theorem \ref{characterthm} begins in full force. These basis replacement propositions all deal in one form or another with the question of writing a sum of squares of linear or multilinear functionals (like $||d_x \pi(x,y)||_\omega^2$) in terms of some more desirable collection of functionals or over a more desirable basis. The first such result is as follows; it allows one to realign basis vectors within some scale of subspaces and is effectively a version of the classical Cholesky decomposition.
\begin{proposition}
Let $\{V_j\}_{j=1}^\ell$ be a decreasing sequence of nontrivial subspaces of a real Hilbert space $H$. Suppose $\{v_i\}_{i=1}^m$ is a basis of $V_1$. Then there exists a basis $\{u_i\}_{i=1}^{m}$ of $V_1$ such that the span of $u_{m-\dim V_j + 1},\ldots,u_m$ equals $V_j$ for each $j \in \{1,\ldots,\ell\}$ and
\label{sbfprop}
\begin{equation}
\sum_{i=1}^m \ang{x,v_i} \ang{v_i,y}= \sum_{i=1}^m \ang{x,u_i} \ang{u_i,y} \qquad \forall x,y \in H. \label{sbf}
\end{equation}
In other words, the vectors $u_1,\ldots,u_m$ can be chosen so that $u_{m-k} \in V_j$ for each $k \in \{0,\ldots,\dim V_j -1\}$ and $j \in \{1,\ldots,\ell\}$.
\end{proposition}
\begin{proof}
Note that the diagonal case $y = x$ of \eqref{sbf} is the case of real interest, but polarization identities show that the bilinear case is entirely equivalent (see Proposition \ref{linprop} below). 

Let $\{e_i\}_{i=1}^m$ be an orthonormal basis of $V_1$ such that for each $j \in \{1,\ldots,\ell\}$, the span of $\{e_{m-i+1}\}_{i=1}^{\dim V_j}$ is $V_j$ (i.e., the final $\dim V_j$ vectors in the basis belong to $V_j$). Such a basis can always be constructed in reverse by finding an orthonormal basis of $V_\ell$ and augmenting it with a maximal collection of linearly independent vectors in $V_{j-1} \setminus V_{j}$ for each $j$ from $\ell$ down to $2$, and then applying the Gram-Schmidt process in the order of construction (which is the opposite of the ordering by index). Now consider the bilinear form
\[ B(x,y) := \sum_{i=1}^m \ang{x,v_i} \ang{v_i,y} \qquad \forall x,y \in H. \]
This $B$ is symmetric and and positive definite on $V_1$ (because no nonzero vector $x$ can belong to the span of the $v_i$ while being orthogonal to each $v_i$), so $(x,y) \mapsto B(x,y)$ is an inner product on $V_1$. Applying the Gram-Schmidt process in the new inner product to the sequence of vectors $e_1,\ldots,e_n$, it follows that there exist $\tilde e_1,\ldots,\tilde e_n$ which are orthonormal with respect to $B(\cdot,\cdot)$ and have the property that $\tilde e_j$ belongs to the span of $e_1,\ldots,e_j$ for each $j$. Orthonormality in $B( \cdot , \cdot)$ implies that
\begin{equation} B(x,y) = \sum_{i=1}^n B(x,\tilde e_i) B(\tilde e_i, y) \qquad \forall x,y \in V_1. \label{sbf0} \end{equation}
By the Riesz Representation Theorem in the original inner product $\ang{\cdot,\cdot}$, there exist $u_1,\ldots,u_n \in V_1$ such that $\ang{u_i,x} = B(\tilde e_i,x)$ for all $x \in V_1$. By \eqref{sbf0} and the definition of the $u_i$'s, the identity \eqref{sbf} must hold for all $x,y \in V_1$. Additionally, if either $x$ or $y$ belongs to $(V_1)^{\perp}$, both sides of \eqref{sbf} are identically zero. Thus \eqref{sbf} holds for all $x,y \in H$ by simply writing $x=x_0+x_1$ and $y = y_0 + y_1$ for $x_0,y_0 \in V_1$ and $x_1,y_1 \in (V_1)^\perp$ and using bilinearity. Orthonormality of the $\tilde e_i$ in the $B$ inner product means that $\delta_{ii'} = B(\tilde e_i, \tilde e_{i'}) = \ang{u_i,\tilde e_{i'}}$ (where $\delta_{ii'}$ is the Kronecker delta), implying that $u_i$ is orthogonal (in the usual sense) to the span of $\tilde e_1,\ldots,\tilde e_{i-1}$, which is also equal to the span of $e_{1},\ldots,e_{i-1}$. Therefore $u_i$ must itself belong to the span of $e_i,\ldots,e_n$ for each $i \in \{1,\ldots,n\}$. This means that the final $\dim V_j$ vectors in $\{u_i\}_{i=1}^m$ belong to $V_j$ for each $j \in \{1,\ldots,\ell\}$. But $u_1,\ldots,u_m$ must be linearly independent by virtue of \eqref{sbf} (as otherwise there would be a nonzero $x \in V_1$ orthogonal to the span of $u_1,\ldots,u_m$ for which the right-hand side must vanish on the diagonal even though the left-hand side is known not to vanish on the diagonal). Thus $u_{m-\dim V_j+1},\ldots,u_m$ must span $V_j$ for each $j \in \{1,\ldots,\ell\}$.
\end{proof}
The next proposition demonstrates that it is also always possible to replace a given basis with one whose vectors are mutually orthogonal (in analogy with the Singular Value Decomposition).
\begin{proposition}
Suppose $\{v_i\}_{i=1}^m$ are linearly independent vectors in a real Hilbert space $H$. There exist $\{v'_i\}_{i=1}^m$ which are mutually orthogonal, have the same span as $\{v_i\}_{i=1}^m$, and satisfy \label{sumsqprop}
\begin{equation} \sum_{i=1}^m  \ang{x,v_i} \ang{v_i,y} = \sum_{i=1}^m \ang{x,v'_i} \ang{v'_i,y} \ \forall x,y \in H. \label{sumsq} \end{equation}
\end{proposition}
\begin{proof}
Let $A$ be the $m \times m$ matrix such that $A_{ii'} := \ang{v_i,v_{i'}}$ for each $i,i' \in \{1,\ldots,m\}$. There exists an orthogonal matrix $O$ of eigenvectors of $A$ such that $\sum_{j,j'} O_{ij} A_{jj'} O_{i'j'} = D_{ii'}$ for some diagonal matrix $D$. Let $v'_i := \sum_{j} O_{ij} v_{j}$ for this matrix $O$. The vectors $v'_i$ are mutually orthogonal because $\ang{v'_i,v'_{i'}} = \sum_{jj'} O_{ij} \ang{v_j,v_{j'}} O_{i'j'} = 0$ when $i \neq i'$. Now
\[ \sum_{i=1}^m \ang{x,v'_i} \ang{v'_i,y} = \sum_{i,j,j'=1}^m O_{ij} \ang{x,v_{j}} O_{ij'} \ang{v_{j'},y} = \sum_{j=1}^m \ang{x,v_j} \ang{v_j,y} \]
for any $x,y \in H$, which one sees from the middle expression by first carrying out the sum over $i$ and using that $O$ is orthogonal.
\end{proof}

The usefulness of Propositions \ref{sbfprop} and \ref{sumsqprop} as basis change results in later sections is a consequence of the identity \eqref{normsqsum} of the following proposition.
\begin{proposition}
Let $\{v_i\}_{i=1}^m$ and $\{w_i\}_{i=1}^m$ be sets of linearly-independent vectors in some finite-dimensional real Hilbert space $H$. The following are equivalent: \label{linprop}
\begin{enumerate}
\item The identity 
\begin{equation} \sum_{i=1}^m \ang{x,v_i} \ang{v_i,y} = \sum_{i=1}^m \ang{x,w_i} \ang{w_i,x} \label{bilin} \end{equation}
holds for all $x, y \in H$.
\item For every integer $k \geq 1$, every $k$-linear functional $L$ satisfies
\begin{equation} \sum_{i=1}^m |L(\tup{v_i}{k})|^2 = \sum_{i=1}^m |L(\tup{w_i}{k})|^2. \label{normsqsum} \end{equation}
\item There exists an $m \times m$ orthogonal matrix $O$ such that 
\begin{equation} v_{i} = \sum_{j} O_{ij} w_j  \text{ for all } i \in \{1,\ldots,m\}. \label{bssize} \end{equation} 
\end{enumerate}
\end{proposition}
\begin{proof}
The proof of \eqref{normsqsum} from \eqref{bilin} follows by induction on $k$ with $k=1$ being the restriction of \eqref{bilin} to the diagonal.  Then for each fixed $i_1,\ldots,i_{k-1}$, one has
\[ \sum_{i_k=1}^m |L(v_{i_1},\ldots,v_{i_k})|^2 = \sum_{i_k=1}^m |L(v_{i_1},\ldots,v_{i_{k-1}},w_{i_k})|^2 \]
simply because each map $u \mapsto L(v_{i_1},\ldots,v_{i_{k-1}},u)$ is a linear functional of $u$, so there is some $x \in H$ such that $L(v_{i_1},\ldots,v_{i_{k-1}},z) = \ang{x,z}$ for all $z \in H$. But also by induction
\[ \sum_{i = 1}^m |L(\tup{v_i}{k-1},w_{i_k})|^2 = \sum_{i=1}^m |L(\tup{w_i}{k-1} , w_{i_k})|^2 \]
for each $i_k$ because $(u_1,\ldots,u_{k-1}) \mapsto L(u_1,\ldots,u_{k-1},w_{i_k})$ is a $(k-1)$-linear functional. Summing over $i_k \in \{1,\ldots,m\}$ completes the proof.

To see why \eqref{bssize} follows from \eqref{normsqsum}, let $\{u_i\}_{i=1}^m$ be vectors in the span of the $w_i$'s such that $\ang{u_i,w_{i'}} = \delta_{ii'}$ for each $i,i' \in \{1,\ldots,m\}$ (which are possible to construct because the $w_i$ are linearly independent). Now let $L(u) := \sum_{i=1}^m t^i \ang{u,u_i}$ for arbitrary $t^1,\ldots,t^m \in \R$. By \eqref{normsqsum},
\[ |t|^2 = \sum_{i=1}^m |L(w_i)|^2 = \sum_{i''=1}^m |L(v_{i''})|^2 = \sum_{i,i',i''=1}^m t^i t^{i'} \ang{v_{i''},u_{i'}} \ang{v_{i''},  u_{i}} \]
for all $t = (t^1,\ldots,t^m) \in \R^m$. Taking partial derivatives of both sides implies that the matrix $O$ such that $O_{ij} := \ang{v_i,u_j}$ satisfies $O^T O = I$, which means that $O$ must be orthogonal.  Similar to earlier observations, \eqref{normsqsum} forces the span of the $v_i$ to equal the span of the $w_i$, since otherwise there would be a vector $u$ orthogonal to all $v_i$ but not all $w_i$ or vice-versa, which violates \eqref{normsqsum} by taking $L$ to be the inner product with this $u$. Therefore, one may always write $v_i = \sum_{j} c_{ij} w_j$ for constants $c_{ij}$; taking inner products of both sides with $u_{j'}$ gives that $O_{ij'} = \ang{v_i,u_{j'}} = c_{ij'}$ for each $i,j'$.

Finally, \eqref{bssize} implies \eqref{bilin} by the same argument that was used to establish \eqref{sumsq} in Proposition \ref{sumsqprop}.
\end{proof}

\subsection{Nondegeneracy examined}
\label{nondegensec}

This section contains proofs of a number of quantitative consequences of the nondegeneracy condition for the curvature functional $Q$ as defined by \eqref{qdef}. The first result, however, is of a qualitative nature and demonstrates the relationship between nondegeneracy and the H\"{o}rmander condition, which will be a small but important ingredient of the passage from polynomial to $C^\infty$ mappings in Section \ref{npasec}.
It should not be a surprise that nondegeneracy of $Q$ corresponds to a strong local curvature criterion.
\begin{proposition}
Let $U \subset \R^n \times \R^{\de}$ be an open set and suppose that $\pi_1 : U \rightarrow \R^{\en}$ and $\pi_2 : U \rightarrow \R^n$ are given by $\pi_1(x,t) := (t, \phi(x,t))$ and $\pi_2(x,t) := x$, where $\phi : U \rightarrow \R^k$ is some smooth function on $U$ such that the Jacobian matrix $D_x \phi$ is everywhere full rank (and here, as in the introduction, let $\de := \en - k$ and $d := n - k$). Suppose that $\{X_i\}_{i=1}^d$ are smooth vector fields on $U$ which span the kernel of $D_{x,t} \pi_1$ (i.e., the Jacobian of $\pi_1$ with respect to both sets of variables $x$ and $t$) at every point and that $\{T_j\}_{j=1}^{\de}$ are smooth vector fields that span the kernel of $D_{x,t} \pi_2$ at every point. If $Q$ as defined by \eqref{qdef} is nondegenerate, then the vector fields \label{hormander}
\begin{equation} X_1,\ldots,X_d,T_1,\ldots,T_{\de},[X_1,T_1],\ldots,[X_1,T_{\de}],\ldots,[X_d,T_1],\ldots,[X_d,T_{\de}] \label{vflist} \end{equation}
span the tangent space of $U$ at $(x,t)$.  Conversely, if the vector fields \eqref{vflist} span the tangent space of $U$ at $(x,t)$, this does not imply nondegeneracy of $Q$.
\end{proposition}
\begin{proof}
By virtue of the definitions of $\pi_1$ and $\pi_2$, it must be the case that the span of $T_1,\ldots,T_{\de}$ is everywhere equal to the span of $\{\partial/\partial t^j\}_{j=1}^{\de}$. Similarly, each $X_i$ must belong to the span of $\{\partial/ \partial x^i\}_{i=1}^n$ and must satisfy
\begin{equation} X_i \phi(x,t) \equiv 0. \label{vanishcond} \end{equation}
Because $X_i$ annihilates $t$, one can regard $X_i$ as a vector in $\R^n$ (i.e., it has no nonzero components in the $t$ directions of $\R^n \times \R^{\de}$) and \eqref{vanishcond} becomes equivalent to the assertion that
\[ D_x \phi|_{(x,t)} X_i|_{(x,t)} \equiv 0, \]
where $D_x \phi$ is the usual Jacobian matrix of $\phi$ with respect to $x$. Differentiating this expression with respect to $T_j$ implies that
\[ T_j D_x \phi|_{(x,t)} X_i|_{(x,t)} \equiv 0 \]
for any $i \in \{1,\ldots,d\}$ and $j \in \{1,\ldots,\de\}$. Writing \textit{both} $X_i$ and $T_j$ in coordinates on $U$ allows one to apply the differential operator $T_j$ directly to the $\R^n$-valued function $X_j$; by definition of the commutator combined with the fact that $T_1,\ldots,T_{\de}$ span all directions in the $t$ variables, $T_j X_i = [T_j, X_i] + \sum_{j'=1}^{\de} c^{j'}_j(x,t) T_{j'}$ for some smooth functions $c^{j'}_j$. Therefore, by the product rule,
\begin{equation} (T_j D_x \phi|_{(x,t)} ) X_i |_{(x,t)} + D_x \phi |_{(x,t)} \left[ [T_j, X_i] + \sum_{j'=1}^{\de} c^{j'}_j(x,t) T_{j'} \right] \equiv 0 \label{twoderiv} \end{equation}
for all $i \in \{1,\ldots,d\}$ and all $j \in \{1,\ldots,\de\}$, where $T_j D_x \phi$ is understood to be the $T_j$ derivative of the matrix $D_x \phi$ expressed in standard coordinates.

Suppose the vector fields \eqref{vflist} fail to span the tangent space of $U$ at $(x,t)$. It still must be the case that the collection  $\{X_i\}_{i=1}^d \cup \{T_j\}_{j=1}^{\de}$ is linearly independent simply because the spaces spanned by $\{X_i\}_{j=1}^d$ and $\{T_j\}_{j=1}^{\de}$ are known to be transverse (and neither collection is itself linearly dependent). Moreover, as the $\{T_j\}_{j=1}^{\de}$ span the same space as $\{\partial/\partial t^j\}_{j=1}^{\de}$, failure of the collection \eqref{vflist} to span means that there must be some direction $X_*$ in the span of $\{\partial/\partial x^i\}_{i=1}^n$ not expressible as a linear combination of $X_1,\ldots,X_d$ and the vector fields $[T_j, X_i] + \sum_{j'=1}^{\de} c^{j'}_j(x,t) T_{j'}$ (which also belong to the span of the $\partial/\partial x^i$ because they equal $T_j X_i$). The missing direction $X_*$ cannot belong to the kernel of $D_x \phi$ because $X_1,\ldots,X_d$ span the kernel. Thus $D_x \phi|_{(x,t)} X_* \neq 0$. If any linear combination $X'$ of the vector fields $X_1,\ldots,X_d$ and the vector fields $[T_j, X_i] + \sum_{j'=1}^{\de} c^{j'}_j(x,t) T_{j'}$ had the property that $D_x \phi|_{(x,t)} X' = D_x \phi|_{(x,t)} X_*$, then $X' - X_*$ would belong to the kernel of $D_x \phi|_{(x,t)}$, which would then imply that $X_*$ itself must belong to the span of the vector fields \eqref{vflist} at the point $(x,t)$. Therefore, dimension of the image via $D_x \phi|_{(x,t)}$ of the span of $X_1,\ldots,X_d$ and the $[T_j, X_i] + \sum_{j'=1}^{\de} c^{j'}_j(x,t) T_{j'}$ cannot be full and therefore the span of just the vectors $D_x \phi|_{(x,t)} \left( [T_j, X_i] + \sum_{j'=1}^{\de} c^{j'}_j(x,t) T_{j'} \right)$ at $(x,t)$ cannot have full dimension. This implies the existence of some nonzero $v \in \R^k$ such that $v \cdot D_x \phi|_{(x,t)} \left( [T_j, X_i] + \sum_{j'=1}^{\de} c^{j'}_j(x,t) T_{j'} \right) = 0$ for all $i$ and $j$. By \eqref{twoderiv}, it must be the case that
\[ v \cdot (T_j D_x \phi|_{(x,t)})  X_i|_{(x,t)} = 0 \]
for all $i$ and $j$. In terms of the definition \eqref{qdef}, this means that $Q$ as defined at the point $(x,t)$ admits a nonzero vector $v \in \R^k$ such that $Q(u,v,w) = 0$ for all $u$ and $w$.
Now choose a coordinate system on $\R^k$ for which the vector $v$ points in the first coordinate direction. In these coordinates, any triple $(\alpha,\beta,\gamma)$ of multiindices satisfying \eqref{qdetdef} would necessarily satisfy $\beta_1 = 0$ (since if $\beta_1 > 0$, at least one of the rows of the matrix in \eqref{qdetdef} would be identically zero as a function of $\tau$). This implies that the convex hull of all such $(\alpha,\beta,\gamma)$ would also belong to the hyperplane $\beta_1 = 0$, and hence that \eqref{diagonal} would not be contained in that convex hull. Thus $Q$ would necessarily be degenerate.

To see that the spanning condition on \eqref{vflist} is insufficient to imply nondegeneracy of $Q$, simply consider the case $\phi((x^1,x^2,x^3),t) := x^2 + x^1 t$ in a neighborhood of the origin. Clearly $D_x \phi$ is never zero and therefore always full rank. In this example, $d = 3-1 = 2$ and $\de = 1$. The vector fields $X_1$ and $X_2$ can be taken to equal $\partial/\partial x^1 - t \partial/\partial x^2$ and $\partial/\partial x^3$, respectively, and $T_1$ can simply equal $\partial/\partial t$.
The vector fields \eqref{vflist} span at the origin because $[X_1,T_1]$ points in the direction $\partial/\partial x^2$ there (which does not lie in the span of $X_1, X_2$, and $T_1$), but at $(x,t) = (0,0)$, there is a nonzero $w$ such that $Q(u,v,w) = 0$ for all $u, v$. This $w$ corresponds to differentiation of $\phi$ in the $x^3$-direction (i.e., $\partial^2 \phi / \partial x^3 \partial t \equiv 0$). As above, this forces $Q$ to be degenerate.
\end{proof}

It is now time to turn to the main goal of this section, which is to examine the property of nondegeneracy in a general way. For convenience,
suppose $Q : \R^{\de } \times \R^k \times \R^d \rightarrow \R$ is any trilinear form (i.e., not necessarily derived from \eqref{qdef} and, at this point, not depending on any smooth parameters like $x$).  A key component of the proof of Theorem \ref{characterthm} is to develop appropriate qualitative and quantitative measures of nondegeneracy of $Q$. As indicated in the introduction, this involves studying determinants like \eqref{qdetdef}.
For any integer $s \geq 1$, let 
\begin{equation} Q_s(t,v_1,\ldots,v_s,w_1,\ldots,w_s) := \det \begin{bmatrix} Q(t , v_1 , w_1) & \cdots & Q(t,v_1, w_{s}) \\  \vdots & \ddots & \vdots \\ Q(t , v_{s} , w_{1}) & \cdots & Q(t,v_{s}, w_{s}) \end{bmatrix} \label{qsdef} \end{equation}
for any vectors $t \in \R^{\de }$, $v_1,\ldots,v_s \in \R^k$ and $w_1,\ldots,w_s \in \R^d$.
Given any bases $\{u_i\}_{i=1}^{\de }$, $\{v_i\}_{i=1}^{k}$ and $\{w_i\}_{i=1}^d$, let 
\begin{equation} \begin{split} \left( \mathcal{Q} [\{u_i\}_{i=1}^{\de }, \{v_i\}_{i=1}^{k},\{w_i\}_{i=1}^d] \right)^2 & := \\  
1 + \sum_{s=1}^{\min\{d,k\}} \sum_{i=1}^{\de }& \sum_{i'=1}^k \sum_{i''=1}^d \left| ( \tup{u_i}{s} \cdot \nabla_t^s) Q_s(t, \tup{v_{i'}}{s}, \tup{w_{i''}}{s}) \right|^2. \end{split} \label{qsumdef}
\end{equation}
Here the notation $\tup{u_i}{s} \cdot \nabla_t^s$ is meant to expand to become the order $s$ differential operator $(u_{i_1} \cdot \nabla_t) \cdots (u_{i_s} \cdot \nabla_t)$, where $\nabla_t$ is the usual gradient in $t$.
Note that there is a simple reason why the the sum over $s$ ends at $\max \{d,k\}$: any larger values of $s$ would necessarily involve determinants of matrices with repeated rows or columns and therefore would be identically zero. Also note that the point $t$ at which the derivative in $t$ is evaluated does not matter because it is an order $s$ derivative of a homogeneous polynomial of degree $s$ in the variable $t$. In fact, this means that for each fixed $s$ in the sum on the right-hand side of \eqref{qsumdef}, the expanded sum over $i_1,\ldots,i_s$ is simply a norm (squared) of the polynomial $Q_s(t,v_{i'_1},\ldots,v_{i'_s},w_{i''_1},\ldots,w_{i''_s})$, which means that it could be replaced by any other comparable norm without fundamentally changing the magnitude of the quantity. For example, if one defines
\begin{equation}
B_u := \set{ t \in \R^{\de }}{ t = \sum_{i=1}^{\de } c_i u_i \text{ for some } \sum_{i=1}^{\de } c_i^2 \leq 1} \label{balldef}
\end{equation}
and 
\begin{equation} \begin{split} \left( \mathcal{Q}_{\sup} [\{u_i\}_{i=1}^{\de }, \{v_i\}_{i=1}^{k},\{w_i\}_{i=1}^d] \right)^2 & := \\  
1 + \sum_{s=1}^{\min\{d,k\}} & \sum_{i'=1}^k \sum_{i''=1}^d \sup_{t \in B_u} \left| Q_s(t, \tup{v_{i'}}{s}, \tup{w_{i''}}{s}) \right|^2, \end{split} \label{qsumaltdef}
\end{equation}
then
\begin{equation}
 \mathcal{Q}_{\sup} [\{u_i\}_{i=1}^{\de }, \{v_i\}_{i=1}^{k},\{w_i\}_{i=1}^d] \approx \mathcal{Q} [\{u_i\}_{i=1}^{\de }, \{v_i\}_{i=1}^{k},\{w_i\}_{i=1}^d] \label{qcomp}
 \end{equation}
 for implicit constants that depend only on $\de ,d$, and $k$. This will become relevant in Section \ref{nesssec} while proving necessity of nondegeneracy.

It is useful to sort terms in the expansion \eqref{qsumdef} in a symmetric way using the notion of multiindices that count $(i_1,\ldots,i_s)$, $(i'_1,\ldots,i'_s)$, and $(i''_1,\ldots,i''_s)$.  The right-hand side of \eqref{qsumdef} is symmetric under permutations of $(i_1,\ldots,i_s)$, $(i'_1,\ldots,i'_s)$, and $(i''_1,\ldots,i''_s)$ because permutations of $u_1,\ldots,u_s$ preserve $\tup{u_i}{s} \cdot \nabla^s_t$ and permutations of $v_1,\ldots,v_s$ or $w_1,\ldots,w_s$ in \eqref{qsdef} change only the sign of $Q_s$. 
Thus for every $\alpha \in \Z_{\geq 0}^{\de}$, $\beta \in \Z_{\geq 0}^k$, and $\gamma \in \Z_{\geq 0}^d$ with $|\alpha| = |\beta| = |\gamma| = s$, there are exactly $(s!)^3 / (\alpha! \beta! \gamma!)$ equal terms on the right-hand side of \eqref{qsumdef} for which $\alpha$ counts $(i_1,\ldots,i_s)$, $\beta$ counts $(i'_1,\ldots,i'_s)$, and $\gamma$ counts $(i''_1,\ldots,i''_s)$.
If one then defines 
$(u \cdot \nabla_t)^{\alpha} Q_s (t, v_{\beta}, w_{\gamma})$ to equal
\[ (u_{i_1} \cdot \nabla_t) \cdots (u_{i_s} \cdot \nabla_t) Q_s(t,v_{i'_1},\ldots,v_{i'_s},w_{i''_1},\ldots,w_{i''_s}) \] for some  $i_1,\ldots,i_s$, $i'_1,\ldots,i'_s$ and $i''_1,\ldots,i''_s$ such that $\alpha$ counts $(i_1,\ldots,i_s)$, $\beta$ counts $(i'_1,\ldots,i'_s)$, and $\gamma$ counts $(i''_1,\ldots,i''_s)$, it follows that
\begin{equation} \begin{split} \left( \mathcal{Q} [\{u_i\}_{i=1}^{\de }, \{v_i\}_{i=1}^{k},\{w_i\}_{i=1}^d] \right)^2 & = \\  
1 + \sum_{s=1}^{\min\{d,k\}} \sum_{|\alpha| = |\beta| = |\gamma| = s} & \frac{(s!)^3}{\alpha! \beta! \gamma!} \left| (u \cdot \nabla_t)^{\alpha} Q_s(t, v_\beta, w_\gamma ) \right|^2. \end{split} \label{qsumdef2}
\end{equation}

Using this formulation, it is possible to give a slightly more intrinsic definition of $\nr(Q)$ that agrees exactly with \eqref{nr0def}. To that end, let $\{\tilde u_i\}_{i=1}^{\de }, \{\tilde v_i\}_{i=1}^{k}$, and $\{\tilde w_i\}_{i=1}^d$ be orthonormal bases of $\R^{\de }, \R^k$, and $\R^d$, respectively. Let
\begin{equation}
\begin{split}
N_0  (Q,\{\tilde u_i\}_{i=1}^{\de }, & \{\tilde v_i\}_{i=1}^{k}, \{\tilde w_i\}_{i=1}^d) \\ := \Big\{ (\alpha,\beta,\gamma) &  \in \Z_{\geq 0}^{\de } \times \Z_{\geq 0}^{k} \times \Z_{\geq 0}^d \ | \   |\alpha| = |\beta| = |\gamma| \leq \min\{d,k\}  \\ & \text{ and either } |\alpha| = 0 \text{ or }  (\tilde u \cdot \nabla_t)^\alpha Q_{|\alpha|}(t,\tilde v_\beta, \tilde w_\gamma) \neq 0  \Big\}
\end{split} \label{N0def}
\end{equation}
and let
\begin{equation} \nr(Q) := \! \! \! \mathop{\bigcap_{ \{\tilde u_i\}_{i=1}^{\de },  \{\tilde v_i\}_{i=1}^{k}, \{\tilde w_i\}_{i=1}^d}}_{\text{orthonormal}} \! \! \! \! \! \! \! \! \! \! \! \! \operatorname{conv. hull} \left( N_0  (Q,\{\tilde u_i\}_{i=1}^{\de }, \{\tilde v_i\}_{i=1}^{k}, \{\tilde w_i\}_{i=1}^d) \right). \label{nrdef}
\end{equation}
Note that the convex hull is taken in $[0,\infty)^{\de  + k + d}$. This is equivalent to \eqref{nr0def} simply because every orthonormal basis can be mapped to the standard basis by an orthogonal matrix, and likewise the action of any orthogonal matrix on the standard basis is to map it to an orthonormal basis. The relevance of $\nr(Q)$ is demonstrated by the following lemma.
\begin{lemma}
Given $Q$, $\mathcal{Q}$ and $\nr(Q)$ as defined above, the following are true: \label{newtonlemma}
\begin{enumerate}
\item For any $s \geq 0$, the point
\begin{equation} \Big(\overbrace{\frac{s}{\de },\ldots,\frac{s}{\de }}^{\de  \text{ copies}},\overbrace{\frac{s}{k},\ldots,\frac{s}{k}}^{k \text{ copies}},\overbrace{\frac{s}{d},\ldots,\frac{s}{d}}^{d \text{ copies}} \Big)\label{spoint} \end{equation} belongs to $\nr(Q)$ if and only if there exists $c > 0$ such that every $Q'$ sufficiently close to $Q$ in the standard topology generates a functional $\mathcal{Q}'$ (replacing $Q_s$ by $Q_s'$ on the right-hand side of \eqref{qsumdef}) satisfying
\begin{equation} \begin{split} \mathcal{Q}'[\{u_i\}_{i=1}^{\de }, & \{v_i\}_{i=1}^k,\{w_i\}_{i=1}^d] \\ & \geq c |\det \{u_i\}_{i=1}^{\de }|^{\frac{s}{\de }} |\det  \{v_i\}_{i=1}^{k}|^{\frac{s}{k}} |\det \{w_i\}_{i=1}^d|^{\frac{s}{d}} \end{split} \label{qnondegen} \end{equation}
for any bases $\{u_i\}_{i=1}^{\de }$, $\{v_i\}_{i=1}^{k}$, and $\{w_i\}_{i=1}^d$.

\item For any $s \geq 0$, the point \eqref{spoint} fails to belong to $\nr(Q)$ if and only if there exist orthonormal bases $\{u_i\}_{i=1}^{\de }$, $\{v_i\}_{i=1}^{k}$, and $\{w_i\}_{i=1}^d$ diagonalizing certain symmetric matrices $D_1, D_2,$ and $D_3$, respectively, such that the trace of $D_1$ is positive, the traces of $D_2$ and $D_3$ are zero, and
\begin{equation} \frac{\mathcal{Q}[\{e^{\tau D_1 } u_i\}_{i=1}^{\de }, \{e^{\tau D_2} v_i\}_{i=1}^k,\{e^{\tau D_3} w_i\}_{i=1}^d]}{|\det \{e^{\tau D_1 } u_i\}_{i=1}^{\de }|^{\frac{s}{\de }} |\det  \{e^{\tau D_2} v_i\}_{i=1}^{k}|^{\frac{s}{k}} |\det \{e^{\tau D_3} w_i\}_{i=1}^d|^{\frac{s}{d}}} \rightarrow 0 \label{hilbertmumford} \end{equation}
 as $\tau \rightarrow \infty$. Note specifically that $|\det \{e^{\tau D_1 } u_i\}_{i=1}^{\de }| \rightarrow \infty$ as $\tau \rightarrow \infty$ and $|\det  \{e^{\tau D_2} v_i\}_{i=1}^{k}| = |\det \{v_i\}_{i=1}^k| = 1 = |\det \{e^{\tau D_3} w_i\}_{i=1}^d| = |\det \{w_i\}_{i=1}^d|$ for all $\tau$.
\end{enumerate}
\end{lemma}
\begin{proof} 
By Propositions \ref{sumsqprop} and \ref{linprop} applied to the right-hand side of \eqref{qsumdef}, one may always replace $\{u_i\}_{i=1}^{\de }, \{v_i\}_{i=1}^{k}$, and $\{w_i\}_{i=1}^d$ by orthogonal (not necessarily normalized) bases without changing its value, and moreover the identity \eqref{bssize} guarantees that the values of $|\det \{u_i\}_{i=1}^{d_1}|, |\det \{v_i\}_{i=1}^{k}|$, and $|\det \{w_i\}_{i=1}^d|$ are also preserved during the orthogonalization process. If one takes $u \in \R^{\de }$ so that $u := (||u_1||,\ldots,||u_{\de }||)$ and likewise for $v$ and $w$ and lets $\{\tilde u_i\}_{i=1}^{\de }, \{\tilde v_i\}_{i=1}^{k}$, and $\{\tilde w_i\}_{i=1}^d$ be the orthonormal bases obtained by normalizing $\{u_i\}_{i=1}^{\de }, \{v_i\}_{i=1}^{k}$, and $\{w_i\}_{i=1}^d$, then by \eqref{qsumdef2},
\begin{equation} \begin{split} \left( \mathcal{Q} [\{u_i\}_{i=1}^{\de }, \{v_i\}_{i=1}^k,\{w_i\}_{i=1}^d] \right)^2 & = \\  
1 +  \mathop{\sum_{1 \leq |\alpha| = |\beta| = }}_{|\gamma| \leq \min\{d,k\}}  \frac{|\alpha|! |\beta|! |\gamma|!}{\alpha! \beta! \gamma!} &  \left| u^\alpha v^\beta w^\gamma (\tilde u \cdot \nabla_t)^{\alpha} Q_{|\alpha|}(t, \tilde v_\beta, \tilde w_\gamma ) \right|^2. \end{split} \label{qsumdef3}
\end{equation}
It thus suffices to study the right-hand side \eqref{qsumdef3} for arbitrary orthonormal bases $\{\tilde u_i\}_{i=1}^{\de }, \{\tilde v_i\}_{i=1}^{k}$, $\{\tilde w_i\}_{i=1}^d$ and arbitrary $u \in \R_{> 0}^{\de }, v \in \R^{k}_{> 0}, w \in \R_{> 0}^d$. 

Let $o_s$ denote the point \eqref{spoint} and suppose $o_s \in \nr(Q)$ for some $s \geq 0$. Then for any orthonormal $\{\tilde u_i\}_{i=1}^{\de }, \{\tilde v_i\}_{i=1}^{k}$, and $\{\tilde w_i\}_{i=1}^d$, there is a finite list of multiindices $(\alpha_1,\beta_1,\gamma_1), \ldots, (\alpha_N, \beta_N,\gamma_N)$, each satisfying $|\alpha_i| = |\beta_i| = |\gamma_i| \leq \min \{d,k\}$, and nonnegative $\theta_1,\ldots,\theta_N$ such that \begin{equation} \sum_{i=1}^N \theta_i = 1 \text{ and } \sum_{i=1}^N \theta_i (\alpha_i,\beta_i,\gamma_i) = o_s \label{weightedavg} \end{equation}
and
\begin{equation} \min_{i \, : \, |\alpha_i| \neq 0} |(\tilde u \cdot \nabla_t)^{\alpha_i} Q_{|\alpha_i|} (t,\tilde v_{\beta_i}, \tilde w_{\gamma_i})| > 0. \label{supportpt} \end{equation}
By continuity of $(\tilde u \cdot \nabla_t)^{\alpha_i} Q_{|\alpha_i|} (t,\tilde v_{\beta_i}, \tilde w_{\gamma_i})$ in the $\tilde u_i$, $\tilde v_i$, and $\tilde w_i$, the minimum \eqref{supportpt} must be nonzero over all orthonormal bases $\{\tilde u'_i\}_{i=1}^{\de }, \{\tilde v'_i\}_{i=1}^{k}$, and $\{\tilde w'_i\}_{i=1}^d$ sufficiently near to $\{\tilde u_i\}_{i=1}^{\de }, \{\tilde v_i\}_{i=1}^{k}$, and $\{\tilde w_i\}_{i=1}^d$. If one calls the collection $\{(\theta_i,\alpha_i,\beta_i,\gamma_i)\}_{i=1}^N$ satisfying \eqref{weightedavg} and \eqref{supportpt} for $\{\tilde u_i\}_{i=1}^{\de }, \{\tilde v_i\}_{i=1}^{k}$, and $\{\tilde w_i\}_{i=1}^d$ an ensemble, it follows by compactness of the orthogonal group that there always exists a finite collection of ensembles such that for any orthogonal bases $\{\tilde u_i'\}_{i=1}^{\de }, \{\tilde v_i'\}_{i=1}^{k}$, and $\{\tilde w_i'\}_{i=1}^d$, one of the ensembles $\{(\theta_i',\alpha_i',\beta_i',\gamma_i')\}_{i=1}^{N'}$ in the finite collection will satisfy \eqref{weightedavg} and \eqref{supportpt} for the given bases. In fact, the same compactness argument establishes that at least one of the ensembles in the collection will satisfy \eqref{supportpt} for each fixed choice of bases $\{\tilde u_i'\}_{i=1}^{\de }, \{\tilde v_i'\}_{i=1}^{k}$ even when $Q$ is replaced by a sufficiently small perturbation $Q'$: the left-hand side of \eqref{supportpt} is Lipschitz as a function of $Q$, and since the sum of $\min_{i \, : \, |\alpha_i| \neq 0} |(\tilde u \cdot \nabla_t)^{\alpha_i} Q_{|\alpha_i|} (t,\tilde v_{\beta_i}, \tilde w_{\gamma_i})|$ over our finite list of ensembles is bounded below by a positive quantity as $\{\tilde u_i\}_{i=1}^{\de }, \{\tilde v_i\}_{i=1}^{k}$, and $\{\tilde w_i\}_{i=1}^d$ range over all orthonormal bases, it remains everywhere strictly positive when $Q$ is replaced by any $Q'$ sufficiently near to it.

Moving now to non-normalized orthogonal bases, the reasoning above and the equality \eqref{qsumdef3} imply that there exists a positive constant $c$ such that all $Q'$ sufficiently near $Q$ and all orthogonal bases $\{u_i\}_{i=1}^{\de }, \{v_i\}_{i=1}^k,\{w_i\}_{i=1}^d$ admit an ensemble $\{(\theta_i,\alpha_i,\beta_i,\gamma_i)\}_{i=1}^N$ satisfying \eqref{weightedavg} and
\begin{equation*} 
\begin{split} \left( \mathcal{Q}' [\{u_i\}_{i=1}^{\de }, \{v_i\}_{i=1}^k,\{w_i\}_{i=1}^d] \right)^2 \geq 
c^2 \sum_{i=1}^N & \frac{(|\alpha_i|!)^3}{\alpha_i! \beta_i! \gamma_i!} \left| u^{\alpha_i} v^{\beta_i} w^{\gamma_i}  \right|^2. \end{split}
\end{equation*}
By the AM-GM inequality, then,
\begin{equation*} 
\begin{split} \mathcal{Q}' & [\{u_i\}_{i=1}^{\de }, \{v_i\}_{i=1}^{k},\{w_i\}_{i=1}^d]   \geq
c \prod_{i=1}^n |u^{\alpha_i} v^{\beta_i} w^{\gamma_i}|^{\theta_i} \\
& = c (||u_1|| \cdots ||u_{\de}||)^{\frac{s}{d_1}} (||v_1|| \cdots ||v_k||)^{\frac{s}{k}} (||w_1|| \cdots ||w_d||)^{\frac{s}{d}} \\
& = c |\det \{u_i\}_{i=1}^{\de }|^{\frac{s}{\de }} |\det  \{v_i\}_{i=1}^{k}|^{\frac{s}{k}} |\det \{w_i\}_{i=1}^d|^{\frac{s}{d}}.
\end{split}
\end{equation*}
This establishes the forward direction of the first main conclusion of the lemma.

If instead $o_s \not \in \nr(Q)$, there must be orthonormal bases $\{\tilde u_i\}_{i=1}^{\de }, \{\tilde v_i\}_{i=1}^k$, $\{\tilde w_i\}_{i=1}^d$ such that $o_s$ does not belong to $N_0(Q,\{\tilde u_i\}_{i=1}^{\de }, \{\tilde v_i\}_{i=1}^{k}, \{\tilde w_i\}_{i=1}^d)$ or its convex hull. By the Separating Hyperplane Theorem, there is an $x \in \R^{\de  + k + d}$ such that $x \cdot o_s > x \cdot (\alpha,\beta,\gamma)$ for all $(\alpha,\beta,\gamma) \in N_0(Q,\{\tilde u_i\}_{i=1}^{\de }, \{\tilde v_i\}_{i=1}^{k}, \{\tilde w_i\}_{i=1}^d)$. Because  $(0,0,0)$ belongs to $N_0(Q,\{\tilde u_i\}_{i=1}^{\de }, \{\tilde v_i\}_{i=1}^{k}, \{\tilde w_i\}_{i=1}^d)$, this means that the dot product $x \cdot o_s$ is strictly positive. Moreover, since $|\alpha| = |\beta| = |\gamma|$ for all triples $(\alpha,\beta,\gamma)$ in $N_0(Q,\{\tilde u_i\}_{i=1}^{\de }, \{\tilde v_i\}_{i=1}^{k}, \{\tilde w_i\}_{i=1}^d)$ and because the special point $o_s$ has the property that the sum of its first $\de $ entries equals the sum of the next $k$ and of the final $d$ entries, one may add any vector
\[ (\overbrace{a,\ldots,a}^{\de  \text{ copies}},\overbrace{b,\ldots,b}^{k \text{ copies}},\overbrace{c,\ldots,c}^{d \text{ copies}}) \] such that $a + b + c = 0$ 
to the vector $x$ without changing either $x \cdot o_s$ or $x \cdot (\alpha,\beta,\gamma)$ for any triple $(\alpha,\beta,\gamma) \in N_0(Q,\{\tilde u_i\}_{i=1}^{\de }, \{\tilde v_i\}_{i=1}^{k}, \{\tilde w_i\}_{i=1}^d)$
This allows one to assume without loss of generality that 
\[ x \cdot (\overbrace{0,\ldots,0}^{\de  \text{ copies}},\overbrace{1,\ldots,1}^{k \text{ copies}},\overbrace{0,\ldots,0}^{d \text{ copies}}) = x \cdot (\overbrace{0,\ldots,0}^{\de  \text{ copies}},\overbrace{0,\ldots,0}^{k \text{ copies}},\overbrace{1,\ldots,1}^{d \text{ copies}}) = 0. \]

Now let $D_1$ be the matrix with eigenvectors $\tilde u_1,\ldots,\tilde u_{\de }$ and associated eigenvalues $x^1,\ldots,x^{\de }$; similarly, let $D_2$ have eigenvectors $\tilde v_1,\ldots,\tilde v_k$ with eigenvalues $x^{\de +1},\ldots,x^{\de +k}$. Finally, let $D_3$ have eigenvectors $\tilde w_1,\ldots, \tilde w_d$ with eigenvalues $x^{\de + k + 1},\ldots,x^{\de +k+d}$. For any real $\tau$, the using the bases $\{e^{\tau D_1 } \tilde u_i\}_{i=1}^{\de }$, $\{e^{\tau D_2} \tilde v_i\}_{i=1}^{k}$, $\{e^{\tau D_3} \tilde w_i\}_{i=1}^{d}$ on the left-hand side of \eqref{qsumdef3} yields $u$, $v$, and $w$ equal to $(e^{\tau x^1},\ldots,e^{\tau x^{\de }})$, $(e^{\tau x^{\de  + 1}},\ldots,e^{\tau x^{\de  + k}})$, and $(e^{\tau x^{\de +k+1}},\ldots,e^{\tau x^{\de +k+d}})$, respectively, on the right-hand side. 

As $\tau \rightarrow \infty$, the asymptotic growth of \eqref{qsumdef3} is thus $o(e^{2 \tau x \cdot o_{s}})$ because every nonzero term of \eqref{qsumdef3} will be $O(e^{\tau x \cdot(\alpha,\beta,\gamma)})$ for some triple $(\alpha,\beta,\gamma)$ belonging to  $N_0(Q,\{\tilde u_i\}_{i=1}^{\de }, \{\tilde v_i\}_{i=1}^{k}, \{\tilde w_i\}_{i=1}^d)$. But
\[ |\det \{u_i\}_{i=1}^{\de }|^{\frac{s}{\de }} |\det \{v_i\}_{i=1}^k|^{\frac{s}{k}} |\det \{w_i\}_{i=1}^d|^{\frac{s}{d}} = e^{\tau x \cdot o_{s}}, \] so \eqref{hilbertmumford} must hold. By construction of $x$, $|\det  \{e^{\tau D_2} v_i\}_{i=1}^{k}| = |\det \{v_i\}_{i=1}^k| = 1$ and $|\det \{e^{\tau D_3} w_i\}_{i=1}^d| = |\det \{w_i\}_{i=1}^d| = 1$ for all $\tau$, forcing $|\det \{e^{\tau D_1 } u_i\}_{i=1}^{\de }| \rightarrow \infty$ because the numerator of \eqref{hilbertmumford} is always greater than or equal to $1$.

To conclude, observe that the properties \eqref{hilbertmumford} and \eqref{qnondegen} are mutually exclusive: if \eqref{hilbertmumford} is true, then \eqref{qnondegen} cannot hold even for $Q' = Q$. Thus \eqref{qnondegen} must be equivalent to the claim that $o_s \in \nr(Q)$. Likewise, if \eqref{qnondegen} does hold for some $c > 0$, then \eqref{hilbertmumford} cannot hold and consequently \eqref{hilbertmumford} is equivalent to the claim that $o_s \not \in \nr(Q)$.
\end{proof}

Readers familiar with Geometric Invariant Theory will recognize \eqref{hilbertmumford} as analogous to an instance of the Hilbert-Mumford Criterion. Theorem 2.1 of \cite{mfk1994} is an especially general expression of the criterion. Theorem 5.2 of \cite{birkes1971} is a somewhat more accessible reference from the standpoint of analysis, and Theorem 1.1 of \cite{bl2021} is even more so.

When applying Lemma \ref{newtonlemma} and \eqref{qnondegen} specifically to the trilinear functional \eqref{qdef}, a small technical issue arises from the need to specify a basis of $\ker D_x \phi$. 
\begin{corollary}
Suppose $Q$ is a nondegenerate trilinear functional defined by \eqref{qdef} at some point $(x_*,t_*)$ for some smooth $\R^k$-valued function $\phi(x,t)$. There exist positive constants $c,c'$ such that for any smooth $\R^k$-valued function $\tilde \phi(x,t)$ defined near $(x_*,t_*)$, if $||D_x \phi |_{(x_*,t_*)} - D_x \tilde \phi |_{(x,t)}|| < c'$ (where $||\cdot||$ is the Hilbert-Schmidt norm) and $|\partial^2_{t^i x^\ell} \phi |_{(x_*,t_*)} - \partial^2_{t^i x^\ell} \tilde \phi |_{(x,t)}| < c'$ for each $i \in \{1,\ldots,\de\}$ and $\ell \in \{1,\ldots,n\}$, then the functional $\tilde Q$ defined via \eqref{qdef} at the point $(x,t)$ using the map $\tilde \phi$ generates $\tilde{\mathcal{Q}}$ via \eqref{qsumdef} satisfies \label{stability}
\begin{equation} \begin{split} \tilde {\mathcal Q} [\{u_i\}_{i=1}^{\de }, & \{v_i\}_{i=1}^k,\{w_i\}_{i=1}^d] \\ & \geq c |\det \{u_i\}_{i=1}^{\de }|^{\frac{s}{\de }} |\det  \{v_i\}_{i=1}^{k}|^{\frac{s}{k}}  |\det \{w_i\}_{i=1}^d|^{\frac{s}{d}} \end{split} \label{qnondegen2} \end{equation}
for any bases $\{u_i\}_{i=1}^{\de }, \{v_i\}_{i=1}^k,\{w_i\}_{i=1}^d$.
\end{corollary}
\begin{proof}
The corollary is almost an immediate consequence of nondegeneracy of $Q$ and \eqref{qnondegen}, but there is one very minor outstanding issue: when the definition \eqref{qdef} is used, it must be possible to choose an orthonormal basis for the kernel of $D_x \tilde \phi$ which is close to whatever particular basis of the kernel of $D_x \phi$ was used to define $Q$ so that the coordinates of $Q$ and $\tilde Q$ are close to one another. This stability is guaranteed by the proposition immediately below.
\end{proof}
\begin{proposition}
Suppose $M$ is a real $k \times n$ matrix of full rank with $n > k$. Let $\{z_i\}_{i=1}^d$ be an orthonormal basis of the kernel of $M$. Then for all $M' \in \R^{k \times n}$ with $||M - M'|| < c$ for some positive $c$ depending on $M$, $M'$ is full rank and there is an orthonormal basis $\{z'_i\}_{i=1}^d$ of the kernel of $M'$ satisfying
\begin{equation} |z_i - z'_i| \leq C ||M - M'|| \label{lipvecs} \end{equation}
for some constant $C$ depending only on $M$. Here $||\cdot||$ indicates the Hilbert-Schmidt norm.
\end{proposition}
\begin{proof}
Because the proposition is invariant under rotations of both $\R^k$ and $\R^n$, by the Singular Value Decomposition, it may be assumed that $M$ has its only nonzero entries along the diagonal and that those entries $\sigma_{ii}$ are real and positive. It may also be assumed that $\{z_i\}_{i=1}^d$ coincide with the final $d = n-k$ coordinate directions in $\R^n$. If $M_0$ is the leftmost $k \times k$ minor of $M$, then $|M_0 v| \geq  |v| \min_i \sigma_{ii}$ for all $v \in \R^k$, and consequently if $M_0'$ is any $k \times k$ matrix such that $||M_0 - M_0'|| < \min_i \sigma_{ii}/2$, then $|M_0' v| \geq |M_0 v| - ||M_0' - M_0|| \cdot |v| \geq |v| \min_{i} \sigma_{ii}/2$, which implies that $M_0'$ is full rank. Therefore for any positive $c < \min_i \sigma_{ii}/2$, if $||M - M'|| < c$, then the leftmost $k \times k$ minor of $M'$ will have full rank. 

Now consider the smooth map
\[ \begin{split} \Phi& (M',\{z_i'\}_{i=1}^d) :=  \begin{pmatrix} M' z'_1 \\  z'_1 \cdot z'_1 \\ M' z'_2 \\ z'_2 \cdot z'_1 \\ z'_2 \cdot z'_2 \\ \vdots \\ M' z'_d \\ z'_d \cdot z'_1 \\ \vdots \\ z'_d \cdot z'_d \end{pmatrix} \end{split} \]
for any $M' \in \R^{k \times n}$ and any $z'_1,\ldots,z'_d$ in $\R^n$. Suppose also that ${z'}^j_i$ is fixed to equal zero whenever $j > k + i$. The Jacobian matrix of $\Phi$ with respect to the collection of variables ${z'}^1_1,\ldots,{z'}^{k+1}_1,\ldots,{z'}^1_d,\ldots,{z'}^n_d$ has special block structure: it has nonzero blocks along the diagonal with sizes $(k+1) \times (k+1)$ through $n \times n$ along with some additional nonzero entries below these blocks (so the Jacobian is block lower-triangular). Enumerating these as blocks $1,\ldots,d$, block $\ell$ has the form
\[ \begin{bmatrix}
M'_{11} & M'_{12} & \cdots & M'_{1(k+\ell)} \\
M'_{21} & M'_{22} & \ddots & \vdots  \\
\vdots & \ddots & \ddots & \vdots  \\
M'_{k 1} & \cdots & \cdots & M'_{k(k+\ell)}  \\
{z'}_1^1 & {z'}^2_1 & \cdots  & {z'}^{k+\ell}_1 \\
\vdots & \vdots & \ddots &  \vdots \\
{z'}^1_{\ell-1} & {z'}^2_{\ell-1} & \cdots & {z'}^{k+\ell}_{\ell-1} \\
2 {z'}^1_{\ell} & 2{z'}^2_{\ell} & \cdots &  2 {z'}^{k+\ell}_{\ell}
\end{bmatrix}. \]
Evaluating this block specifically when $M'  = M$ and $\{z'_i\}_{i=1}^d = \{z_i\}_{i=1}^d$, each block is diagonal with entries $\sigma_{11},\ldots,\sigma_{kk}$ followed by $\ell-1$ diagonal entries of $1$ and a last diagonal entry of $2$. This implies that the Jacobian determinant of $\Phi$ at $M,\{z_i\}_{i=1}^d$ is nonzero. By the Implicit Function Theorem, then, as $M'$ varies near $M$, there is a unique solution $\{z_i'\}_{i=1}^d$ of the system
\[ M' z_i' = 0, i \in \{1,\ldots,d\}, \ z_i' \cdot z_{i'}' = \delta_{ii'}, 1 \leq i \leq i' \leq d, \ {z'}^j_i = 0 \text{ for } j > k+i,\]
which varies smoothly with $M'$ and agrees with $\{z_i\}_{i=1}^d$ when $M' = M$. This implies \eqref{lipvecs}.
\end{proof}

A final observation needed regarding nondegeneracy is to establish an identity expressing the quantity \eqref{qsumdef} not in terms of $Q$ directly but in terms of a dualized object $\Theta : \R^{\de } \times \R^d \rightarrow \R^k$.
Given any basis $\{v_i\}_{i=1}^k$ of $\R^k$, there always exists a dual basis $\{v^*_i\}_{i=1}^k$ with the property that $v_i \cdot v^*_{i'} = \delta_{ii'}$ for every $i,i' \in \{1,\ldots,k\}$. A particularly important consequence of this identity is that $\det \{v_i^*\}_{i=1}^k = (\det \{v_i\}_{i=1}^k)^{-1}$, which can be seen by recognizing that the matrix with $v_i$'s along its rows is the inverse of the matrix with $v^*_{i'}$ along its columns. Another important property of the dual basis is that every $x \in \R^k$ satisfies $x = \sum_{i=1}^k (x \cdot v_i^*) v_i$ (this is immediate when $x = v_i$ for some $i$ and then extends to all of $\R^k$ by linearity and the fact that $\{v_i\}_{i=1}^k$ is a basis). As a consequence, it follows that the matrix with columns
\[ v_{i_1},\ldots, v_{i_\ell}, y_1,\ldots,y_{k-\ell} \]
can be written as a product of matrices $V M$, where the columns of $V$ are exactly given by $v_1,\ldots,v_k$ and $M$ has key structural properties. The first is that columns one through $\ell$ of $M$ have all zero entries with the exceptions that row $i_j$ of column $j$ has value $1$ for each $j \in \{1,\ldots,\ell\}$. For any remaining column with index $i' + \ell$, it will have entries
$v_1^* \cdot y_{i'},\ldots,v_k^* \cdot y_{i'}$. As a consequence, if $\{i'_1,\ldots,i'_{k-\ell}\}$ are distinct indices in $\{1,\ldots,k\}$ such that $\{i_1,\ldots,i_\ell\} \cup \{ i'_1,\ldots,i'_{k-\ell}\}$, then
 \[ |\det (v_{i_1},\ldots,v_{i_\ell}, y_1,\ldots,y_{k-\ell})| = |\det \{v_i\}_{i=1}^k| | \det (v^*_{i'_j} \cdot y_{j'})_{j,j'=1}^{k-\ell}|. \]
 By virtue of this identity, if $Q$ denotes the functional $Q(x,y,z) := y \cdot \Theta(x,z)$, then it follows that
\begin{equation}
\begin{split} \sum_{s=0}^{\min\{d,k\}}  \sum_{i=1}^{\de } \sum_{i'=1}^{k} & \sum_{i''=1}^{d} |(\tup{u_i}{s} \cdot \nabla_t^s) \det ( \tup{v_{i'}}{k-s}, \tup{\Theta(t,w_{i''})}{s}) |^2 = \\
  |\det \{v_i\}_{i=1}^k|^2 & \left[ 1 + \sum_{s=1}^{\min \{d,k\}} \sum_{i=1}^{\de } \sum_{i'=1}^k \sum_{i''=1}^d | (\tup{u_i}{s} \cdot \nabla_t^s) Q_s (t,\tup{v_{i'}^*}{s}, \tup{w_{i''}}{s}) |^2 \right]  \\
  & =  |\det \{v_i\}_{i=1}^k|^2 \left( \mathcal{Q} [\{u_i\}_{i=1}^{\de }, \{v_i^*\}_{i=1}^k,\{w_i\}_{i=1}^d] \right)^2
 \end{split} \label{inversionthing}
 \end{equation}
for $\mathcal{Q}$ exactly as defined in \eqref{qsumdef}. This identity will be of use in Section \ref{collectineq}.

\section{Geometric differential inequalities}
\label{gdineqsec}

This section revisits a family of derivative estimates for ``nice'' functions which first appeared in \cite{gressman2019}. In that paper, the relevant inequalities applied to finite-dimensional families of real-analytic functions. For the $C^\infty$ case of Theorem \ref{characterthm}, the main strategy will be that of polynomial approximation of the mapping $\gamma_t(x)$; for that reason, it will be necessary to reestablish key inequalities from \cite{gressman2019}*{Section 5} to demonstrate that the various constants involved depend in a manageable way on the degrees of the polynomials involved. The full generality of the approach in \cite{gressman2019} will not be needed here, so all functions will simply have open domains in $\R^d$.

Suppose that $U_1 \supset U_2 \supset \cdots$ are open sets in $\R^d$ and that for each integer $N \geq 1$,  $\{X^{(N)}_i\}_{i=1}^d$ is a family of smooth vector fields on $U_N$.  Any $\alpha := ((N_1,\ldots,N_\ell),(i_1,\ldots,i_\ell))$ such that $N_1 < N_2 < \cdots < N_\ell$ and $i_1,\ldots,i_\ell \in \{1,\ldots,d\}$, will be called a generalized multiindex and $X^{\alpha}$ will be defined to equal the differential operator
\begin{equation} X^\alpha := X^{{(N_\ell)}}_{i_\ell} \cdots X^{{(N_1)}}_{i_1} \label{valid} \end{equation}
which acts on smooth functions defined on $U_{N_\ell}$.
The the order of this operator is $\ell$ (which will also be denoted $|\alpha|$) and $N_{\ell}$ will be called the generation of the operator $X^\alpha$.  We also consider $\alpha := (\emptyset,\emptyset)$ to be a generalized multiindex of order $0$ and generation $0$ and define $X^{\alpha}$ to be the identity operator. Systems of such vector fields are exactly the main objects constructed in Theorem \ref{mainineq} below.

The main case of interest will be to apply these vector fields $\{X^{(N)}_i\}_{i=1}^d$ to polynomials on $\R^d$. However, it is desirable to have a slightly broader class of functions available (if for no other reason than to recover the sharp $L^p$--$L^q$ inequality for spherical averages via Theorem \ref{characterthm}\footnote{Throughout the previous work \cite{testingcond}, Proposition 3 is the only place where B\'{e}zout's Theorem is needed, and as noted there, the result applies perfectly well to Nash functions when degrees of polynomials are replaced by complexities. So Theorem 4 of \cite{testingcond} holds in this suitably modified way for Nash functions. A consequence of the proof in Section \ref{suffsec} below will then be that  $L^{p_b}$--$L^{q_b}$ inequalities hold when $\phi$ is Nash.}). The leap from polynomials to slightly more general functions turns out to be a modest one, as the proof of Theorem \ref{mainineq} below already requires (in a seemingly unavoidable way) some such structure. To that end, the class of functions considered will be the Nash functions. An analysis-friendly definition will be given in Section \ref{nash}; To understand Theorem \ref{mainineq}, one needs only to first know the most very basic features. Nash functions are the algebraic real-analytic functions and consequently include polynomials; each Nash function $f$ (aside from the zero function) has an associated nonnegative integer known as its complexity; complexity will be denoted $c(f)$, and for polynomials, the complexity is never greater than the degree.
The main result of this section is the following.
\begin{theorem}
Suppose $U_0 \subset \R^d$ is open and $f : U_0 \rightarrow \R^m$ is real analytic and has a rank $d$ Jacobian matrix $D_x f$ at every $x \in U_0$. Suppose also that each component function $f^j$ is a Nash function on $U_0$ of complexity at most $D$ and that $\wt$ is some nonnegative locally integrable function on $U_0$.  \label{mainineq} Finally, suppose that $E_0 \subset U_0$ is a compact set such that $\sup_{x \in E_0} |f^j(x)| \leq 1$ for each $j$.
 Then for every integer $N \geq 1$, there exists an open set $U_N \subset U_{N-1}$, a compact set $E_N \subset E_{N-1} \cap U_N$, smooth vector fields $\{X^{(N)}_{i}\}_{i=1,\ldots,d}$ defined on $U_N$, and positive constants $c_{N,d}$ depending only on $d,N$ such that the following are true:
\begin{enumerate}
\item For each $N \geq 1$, $\wt(E_N) \geq c_{N,d} m^{-Nd} \wt(E)$, where $\wt $ applied to a set denotes the measure of the set with respect to $\wt \, dx$ (i.e., Lebesgue measure with density $\wt$).
\item For each $N', N$ with $1 \leq N' < N$ and each $x \in U_N$,
\begin{equation} \left. X^{(N)}_i \right|_{x} = \sum_{i'=1}^d c_{i}^{i'} (x) \left. X^{(N')}_{i'} \right|_x \label{basissize0} \end{equation}
for some smooth coefficients $c_{i}^{i'}$ of magnitude at most $2$.
\item For each $N \geq 1$ and each $x \in E_N$, 
\begin{equation} \wt(x) |\det (X^{(N)}_1,\ldots,X^{(N)}_d)| \Big|_x \geq c_{N,d} m^{-Nd} D^{-(2d+2)^N} \wt(E_0). \label{nondegenvecs} \end{equation}
 Here $\det (X^{(N)}_1,\ldots,X^{(N)}_d)$ indicates the determinant of the $d \times d$ matrix whose columns are given by the representations of $X^{(N)}_1,\ldots,X^{(N)}_d$ in the standard coordinates of $\R^d$.
\item For each $j \in \{1,\ldots,m\}$ and each generalized multiindex $\alpha$ of generation at most $N$,
\begin{equation} \sup_{x \in U_N} |X^\alpha f^j(x)| \leq 1. \label{diffineq} \end{equation}
\end{enumerate}
\end{theorem}
With the exception of \eqref{nondegenvecs}, all conclusions of Theorem \ref{mainineq} will be immediate consequences of Lemma \ref{geocorr} below. This lemma has the pleasant property that it applies perfectly well to any smooth functions $f$, and so there is no need to appeal to any algebraic notions in Section \ref{vfc}. In the place of the constants in \eqref{nondegenvecs}, Lemma \ref{geocorr} will bound $|\det (X^{(N)}_1,\ldots,X^{(N)}_d)|$ from below in terms of a relevant integral expression. This expression is then estimated in Section \ref{nash} by lifting all the functions $X^{\alpha} f^j$ to some higher-dimensional space in such a manner that they become polynomials and B\'{e}zout's Theorem applies. Even if each $f^j$ is itself already a polynomial, such a lifting is necessary because the definition of the vector fields $X^{(N)}_i$ involves dividing by polynomials, which is simply not possible to handle in a straightforward way without enlarging the class of functions considered beyond polynomials.

\subsection{Vector field construction} 
\label{vfc}
As noted above, this section contains the proof of Lemma \ref{geocorr}, which differs from Theorem \ref{mainineq} primarily in the presentation of \eqref{nondegenvecs}, which is replaced by the somewhat more opaque inequality \eqref{volume}.
\begin{lemma}
Let $U_0 \subset \R^d$ be open and suppose that $f : U_0 \rightarrow \R^m$ is smooth and has Jacobian $D_x f$ which is everywhere rank $d$. Let $\wt$ be nonnegative and locally integrable on $U_0$ and let $E_0 \subset U_0$ be compact and satisfy $\sup_{x \in E_0} |f^j(x)| \leq 1$ for all $j=1,\ldots,m$. For each integer $N \geq 1$, there exists an open set $U_N \subset U_{N-1}$, a compact set $E_N \subset E_{N-1} \cap U_N$, smooth vector fields $\{X^{(N)}_{i}\}_{i=1}^d$ defined on $U_N$, and a positive constant $c_{N,d}$ depending only on $N$ and $d$ such that the following are true: \label{geocorr}
\begin{enumerate}
\item For each $N \geq 1$, \begin{equation} \wt(E_N) \geq c_{N,d} m^{-Nd} \wt(E_0). \label{setsize} \end{equation}
\item For each $N', N$ with $1 \leq N' < N$ and each $x \in U_N$,
\begin{equation} \left. X^{(N)}_i \right|_{x} = \sum_{i'=1}^d c_{i}^{i'}(x) \left. X^{(N')}_{i'} \right|_x \label{basissize} \end{equation}
for some smooth coefficients $c_{i}^{i'}$ of magnitude at most $2$.
\item For each $N \geq 1$, there exist generalized multiindices $\alpha_1,\ldots,\alpha_d$ of generation at most $N-1$ and $j_1,\ldots,j_d \in \{1,\ldots,m\}$ such that
\begin{align}
 \wt(x) \left. |\det (X_1^{(N)},\ldots,X_{d}^{(N)})| \right|_x &  \int_{E_{N-1}} \left| \det \frac{\partial ( X^{\alpha_1} f^{j_1}, \ldots, X^{\alpha_d} f^{j_d})}{\partial x} \right| dx \nonumber \\ & \geq \frac{c_{N,d}}{2^d} m^{-Nd} \wt(E_0)  \label{volume} \end{align}
for all $x \in E_N$. Likewise for these same $\alpha_{i'}$ and $j_{i'}$,
\begin{equation} X_i^{(N)} X^{\alpha_{i'}} f^{j_{i'}} = \frac{\delta_{ii'}}{2} \label{kdelta2} \end{equation}
at every point of $U_N$ and for every $i,i' \in \{1,\ldots,d\}$.
\item For each $N \geq 1$, each generalized multiindex $\alpha$ of generation at most $N$, and each $j \in \{1,\ldots,m\}$,
\begin{equation} \sup_{x \in U_N} |X^{\alpha} f^j (x) | \leq 1. \label{sizeest} \end{equation}
\end{enumerate}
\end{lemma}
\begin{proof}
The case $N=1$ captures most of the difficulty.  Given $\mathcal J := (j_1,\ldots,j_{d}) \in \{1,\ldots,m\}^{d}$, let $\jac{\mathcal J}(x)$ be defined to equal the Jacobian determinant
\[ \det \frac{\partial (f^{j_1},\ldots,f^{j_d})}{\partial x} (x). \]
Because $D_x f$ has full rank at every point $x \in U_0$, for any $x$, there is at least one $\mathcal J$ such that $\jac{\mathcal J}(x)$ is nonzero.   For each $\mathcal J \in \{1,\ldots,m\}^d$, define
\[ U_{\mathcal J} := \set{ x \in U_0}{ \max_{ \mathcal I \in \{1,\ldots,m\}^{d}} \left|\jac{\mathcal I} (x) \right| < 2 \left|\jac{\mathcal J} (x) \right| }. \]
Each $U_{\mathcal J}$ is open because $f$ has continuous partial derivatives, and the union $\bigcup_{\mathcal J} U_{\mathcal J}$ equals $U_0$ because every $x \in U_0$ must belong the set $U_{\mathcal J}$ for the particular $\mathcal J$ which attains the nonzero maximum $\max_{\mathcal I \in \{1,\ldots,m\}^{d}} |\jac{\mathcal I} (x) |$ at $x$.

Given the set $E_0$, fix a choice of $\mathcal J$ once and for all such that $\wt(E_0 \cap U_{\mathcal J}) \geq m^{-d} \wt(E_0)$. Such a $\mathcal J$ must exist because the union of all $U_{\mathcal J}$ contains $E_0$ and the number of such sets which are nonempty never exceeds $m^d$. Let $U_1 := U_{\mathcal J}$ for the fixed $\mathcal J$. For each $i \in \{1,\ldots,d\}$, let $X_1^{(1)},\ldots,X_d^{(1)}$ be vector fields on $U_1$ defined by
\begin{equation} X_i^{(1)} \varphi := (-1)^{i-1} \frac{1}{2 \jac{\mathcal J}(x)} \det \frac{\partial (\varphi, f^{j_1},\ldots,\widehat{f^{j_i}},\ldots,f^{j_d})}{\partial x}(x). \label{vfdef} \end{equation}
(Here $\widehat{f^{j_i}}$ denotes omission.)
These vector fields are smooth on $U_{\mathcal J}$ because $\jac{\mathcal J}$ is necessarily nonzero on $U_1$. By definition of $U_\mathcal J$, $|X_i^{(1)} f^j | \leq 1$
at every point $x \in U_{\mathcal J}$ for any $i \in \{1,\ldots,d\}$ and any $j \in \{1,\ldots, m\}$ because the magnitude of $X_i^{(1)} f^j$ is equal to a ratio $| \jac{\mathcal I} / (2 \jac{\mathcal J})|$ with $\mathcal I := \{j,j_1,\ldots,\widehat{j_i},\ldots,j_d\}$. This gives \eqref{sizeest}. It is also immediate that $X_i^{(1)} f^{j_{i'}} = \delta_{ii'}/2$
for each $i,i' \in \{1,\ldots,d\}$ at every point of $U_{\mathcal J}$, which gives \eqref{kdelta2} when $\alpha_1,\ldots,\alpha_d$ are trivial and $j_1,\ldots,j_d$ are the elements of $\mathcal J$. An important consequence of \eqref{kdelta2} is that
\begin{equation}  \left| \jac {\mathcal J} \det (X_1^{(1)},\ldots,X_d^{(1)})   \right| \Big|_x = 2^{-d} \text{ for all } x \in U_1, \label{size1} \end{equation}
which holds because the Jacobian matrix $\partial (f^{j_1},\ldots,f^{j_d})/\partial x$ times the matrix with columns $X_1^{(1)},\ldots,X_1^{(d)}$ is the matrix whose $(i',i)$ entry is $X_{i}^{(1)} f^{j_{i'}}$, so taking determinants gives exactly \eqref{size1} by definition of $\varphi^{\mathcal{J}}$.

If $\wt(E_0) = 0$, let $E_1$ be the empty set, in which case \eqref{volume} holds vacuously.  Otherwise, assume that $\wt(E_0) > 0$ (so that $\wt(E_0 \cap U_{\mathcal J}) > 0$ as well) and let $E' \subset E_0 \cap U_{\mathcal J}$ be the set of points $x \in U_1$ for which $\wt(x) > 0$ and
\begin{equation} \frac{ \left| \jac{\mathcal J}(x) \right|}{\wt(x)} <  \frac{2}{\wt(E_0 \cap U_{\mathcal J} )} \int_{E_0 \cap U_{\mathcal J}} | \jac{\mathcal J}|. \label{size2} \end{equation}
If $E''$ is the subset of $E_0 \cap U_{\mathcal J}$ on which $\wt(x) > 0$ and \eqref{size2} fails, then by Chebyshev's inequality,
\[  \left( \frac{2}{\wt(E_0 \cap U_{\mathcal J} )} \int_{E_0 \cap U_{\mathcal J}} | \jac{\mathcal J}| \right) \int_{E''} \wt(x) dx \leq \int_{E''} \wt(x) \frac{ \left| \jac{\mathcal J}(x) \right|}{\wt(x)} dx,\]
i.e.,
\begin{equation}
\frac{2 \wt(E'')}{\wt(E_0 \cap U_{\mathcal J})} \int_{E_0 \cap U_{\mathcal J}} |\jac{\mathcal J}|  
\leq \int_{E''} |\jac{\mathcal J}|. \label{wanttodivide}
\end{equation}
Trivially one has that $\int_{E_0 \cap U_{\mathcal J}} |\jac{\mathcal J}| \geq \int_{E''} |\jac{\mathcal J}|$, and $\int_{E_0 \cap U_{\mathcal J}} |\jac{\mathcal J}|$ must be strictly positive because $\jac{\mathcal J}$ is never zero on $E_0 \cap U_{\mathcal J}$ and $E_0 \cap U_{\mathcal J}$ has positive Lebesgue measure (since its measure with weight $\wt$ is also positive). This allows one to divide both sides of \eqref{wanttodivide} by $\int_{E_0 \cap U_{\mathcal J}} |\jac{\mathcal J}|$ to conclude that
$\wt(E'') \leq \wt(E_0 \cap U_{\mathcal J})/2$. Moreover, because $E_0 \cap U_{\mathcal J}$ is a disjoint union of $E'$, $E''$, and a set where $\wt \equiv 0$, one now knows that $\wt(E') > 0$. This forces $\int_{E'} | \jac{\mathcal J}| > 0$ as well, 
so that
\[ \int_{E_0 \cap U_{\mathcal J}} |\jac{\mathcal J}| > \int_{E''} |\jac{\mathcal J}|, \]
i.e., the inequality observed a moment ago must actually be strict.
So in fact $\wt(E'') < \wt(E_0 \cap U_{\mathcal J})/2$ and $\wt(E') > \wt(E_0 \cap U_{\mathcal J})/2$. By inner regularity, $E'$ must contain a compact subset $E_1$ such that $\wt(E_1) \geq \wt(E_0 \cap U_{\mathcal J})/2$, which is itself greater than $m^{-d} \wt(E_0)/2$, giving \eqref{setsize}.  Combining \eqref{size1} and \eqref{size2} gives
\begin{align*} m^{-d} \wt(E_0) & \leq \wt(E_0 \cap U_{\mathcal J})  <  \frac{2 \wt(x)}{|\varphi^{\mathcal J}(x)|} \int_{E_0 \cap U_{\mathcal J}} | \jac{\mathcal J}| \\
= 2^{d+1}&  \wt(x) \left| \det (X^{(1)}_1,\ldots,X^{(1)}_d) \right| \Big|_x \int_{E_0 \cap U_{\mathcal J}} \left| \det \frac{\partial (f^{j_1},\ldots,f^{j_d})}{\partial x} \right| dx \end{align*}
at all points $x \in E_1$, which implies \eqref{volume} (with $\alpha_1,\ldots,\alpha_d)$ taken to be trivial).
This completes the proof of the lemma when $N=1$ (with $c_{1,d} := 1/2$) because \eqref{basissize} is vacuous in this case.

Assuming that the lemma holds up to some value of $N$, the conclusions for $N+1$ (with the exception of \eqref{basissize}) follow from applying the $N=1$ case just established to the new map $f^{(N)}$ on $U_N$ which has coordinates $X^{\alpha} f^j$ for all possible $j \in \{1,\ldots,d\}$ and all $\alpha$ of generation at most $N$. The number of such functions is exactly $m (d+1)^N$ (because at each generation, one may choose to differentiate by any one of the $d$ vector fields of that generation or one may choose not to differentiate at all).  Given $U_N$ and $E_N$ from the previous step, there must exist $U_{N+1} \subset U_N$ and a compact set $E_{N+1} \subset E_N \cap U_{N+1}$ such that
$\wt(E_{N+1}) \geq \wt(E_N) / (2 m^d (d+1)^{N d})$ and
\begin{align*} \wt(x) \left| \det(X^{(N+1)}_1,\ldots,X^{(N+1)}_d) \right| \Big|_x & \int_{E_{N}} \left| \det \frac{\partial (X^{\alpha_1} f^{j_1},\ldots,X^{\alpha_d} f^{j_d})}{\partial x} \right| dx \\ & \geq \frac{\wt(E_N)}{2^{d+1} m^d (d+1)^{Nd}}
\end{align*}
for all $x \in E_{N+1}$.  The identity
\[ \frac{1}{2 m^d(d+1)^{(N-1)d}} \cdot \frac{1}{2 m^d(d+1)^{(N-2)d}} \cdots \frac{1}{2 m^d} = \frac{1}{2^N m^{dN} (d+1)^{d N(N-1)/2}} \]
gives \eqref{setsize} and \eqref{volume} with $c_{N,d} := 2^{-N} (d+1)^{-d N(N-1)/2}$. Both \eqref{kdelta2} and \eqref{sizeest} follow immediately by induction as well.

 The only remaining conclusion to establish is the relationship \eqref{basissize} between the vector fields $\{X^{(N)}_i\}_{i=1}^d$ and $\{ X^{(N')}_{i} \}_{i=1}^d$ for $N' < N$. Because there exist $\alpha_1,\ldots,\alpha_d$ of generation at most $N'-1$ and $j_1,\ldots,j_d \in \{1,\ldots,d\}$ such that
\[ X^{(N')}_{i} X^{\alpha_{i'}} f^{j_{i'}} = \frac{1}{2} \delta_{ii'} \]
on $U_{N'}$, applying both sides of the identity \eqref{basissize} to $2 X^{\alpha_{i'}} f^{j_{i'}}$ gives that
$c_{i}^{i'} = 2 X^{(N)}_{i} X^{\alpha_{i'}} f^{j_{i'}}$ on $U_N$ for each $i,i' \in \{1,\ldots,d\}$. By \eqref{sizeest}, $|c_{i}^{i'}| \leq 2$ on $U_{N}$.
\end{proof}

\subsection{Nash functions and polynomial lifting}
\label{nash} 

To prove Theorem \ref{mainineq}, the only difficulty not resolved by Lemma \ref{geocorr} is the estimation of the integral appearing on the left-hand side of \eqref{volume}. The size of the integral of a Jacobian determinant is, in  broad terms, controlled by the maximal number of nondegenerate solutions of an underlying system of equations, and in many similar problems, this issue is resolved in one way or another by an application of B\'{e}zout's Theorem. In this case, there is additional difficulty caused by the fact that even when the component functions $f^j$ are polynomial, the derivatives $X^{\alpha} f^j$ need not be, so B\'{e}zout's theorem is not directly applicable when $N \geq 2$. As it turns out, however, there is a somewhat broader class of functions than polynomials, namely Nash functions, which has the very desirable property that it remains closed under differentiation by the vector fields $X^{(N)}_i$. And because \textit{some} extension beyond the class of polynomials is needed anyway, there is very little added complexity in working directly with Nash functions at every step.

Informally, a Nash function on an open subset $U \subset \R^d$ can be understood as the restriction to a suitably nice submanifold of $U \times \R^{M}$ of a function which is a polynomial on the larger space. To that end,
suppose that $U \subset \R^d$ is open. A real analytic map $\Phi : U \rightarrow \R^M$ will be called an algebraic lifting map when there exists a polynomial map $p : \R^d \times \R^M \rightarrow \R^M$ such that $p(x,\Phi(x)) \equiv 0$ for all $x \in U$ and $\det D_y p(x,y) |_{y = \Phi(x)}$ is not identically zero on any open subset of $U$. The map $p$ will be called the associated annihilating map.  A function $f$ on $U$ will be called a Nash function when it lifts via some algebraic lifting map $\Phi$ to a polynomial, which means that there exists a polynomial $F$ on $\R^d \times \R^M$ such that $f(x) = F(x,\Phi(x))$ for all $x \in U$. Rational functions are Nash on domains where the denominator is nonvanishing, as are smooth algebraic functions ($\sqrt{1+x^2}$ is Nash on the entire real line but $\sqrt[3]{x}$ is not).

This differs somewhat from the traditional definition of Nash functions, which are more commonly defined to be those real analytic functions $f : U \rightarrow \R$ which satisfy $q(x,f(x)) = 0$ on $U$ for some polynomial $q(x,z)$ of $d+1$ variables for which $D_z q$ is not identically zero. The first definition above will be more convenient than the classical one for our purposes, but both are entirely equivalent.
\begin{proposition}
Let $U \subset \R^d$ be open and connected and let $f : U \rightarrow \R$ be a real analytic function. Then $f$ lifts to a polynomial $F$ for some algebraic lifting map $\Phi : U \rightarrow \R^M$ with associated annihilating map $p(x,y)$ if and only if there exists a real polynomial $q(x,z)$ of $d+1$ variables with $D_z q$ not identically zero such that $q(x,f(x)) \equiv 0$ on $U$. \label{itsnash}
\end{proposition}
\begin{proof}
Supposing that such a $q(x,z)$ exists, it may be assumed without loss of generality that $D_z q (x,f(x))$ is not identically zero on $U$, as if this were the case, the polynomial $D_z q$ could by induction replace $q$ as the polynomial for which $q(x,f(x)) \equiv 0$ (note that one needs to verify in this case that $D_z^2 q(x,z)$ is not identically zero; if it happened to be identically zero, then $D_z q(x,f(x))$ being zero on $U$ would force $D_z q(x,z)$ to be identically zero \textit{everywhere}, which is known not to be the case).  Then the algebraic lifting map $\Phi(x) := f(x)$ (with $M = 1$) has associated annihilating map $p(x,y) := q(x,y)$ and $f$ lifts to the polynomial $y$ via $\Phi$ on $U$.

Conversely, given $f, \Phi$, $p$, and $F$, let $I$ be the ideal of all polynomials $g(x,y,z) \in  \C[x,y,z]$ (for $x = (x^1,\ldots,x^d)$, $y = (y^1,\ldots,y^M)$, and $z$ being one-dimensional) such that $g(x,\Phi(x),f(x)) = 0$ for all $x \in U$. The ideal $I$ is prime because if $g(x,\Phi(x),f(x)) h(x,\Phi(x),f(x)) = 0$ for all $x \in U$, then real analyticity of $\Phi$ and $f$ guarantee that either $g \in I$ or $h \in I$. The variety $V(I) \subset \C^{d+M+1}$ of points at which every $g \in I$ vanishes is consequently irreducible. Its dimension is at most $d$ because the polynomials $p^1(x,y),\ldots,p^M(x,y)$ and $-z + F(x,y)$ each belong to $I$ and there exists an $x_0 \in U$ such that at the point $(x_0,\Phi(x_0),f(x_0)) \in V(I)$, the Jacobian matrix $D_{y,z} (p^1(x,y),\ldots,p^M(x,y),-z+F(x,y))$ has determinant $-\det D_y p(x,y) \neq 0$ and consequently has rank $M+1$ (see \cite{cos}*{9.6, Definition 7 and Theorem 8} for the relationship between Jacobians and the dimension of a variety). As a consequence, there must be a nontrivial polynomial $G(x,z) \in I \cap \C [ x,z]$ \cite{cos}*{9.5, Corollary 4}. If $G \in \C [x]$, then because every $G \in I$ vanishes at all points $(x,\Phi(x),f(x))$ for $x \in U$, this would mean that $G(x) = 0$ on $U$, which then means that $G$ is simply the trivial polynomial. Thus $G(x,z)$ must vanish at all points $(x,f(x))$ for $x \in U$ and must depend nontrivially on $z$. Both the real and imaginary parts of $G(x,z)$ also vanish at all points $(x,f(x))$ for $x \in U$, and at least one must depend nontrivially on $z$, so taking $q$ to be either the real or imaginary part of $G$ gives a real polynomial with nontrivial $z$ dependence such that $q(x,f(x)) = 0$ for all $x \in U$.
\end{proof}

Given an open set $U \subset \R^d$ and a Nash function $f$ not identically zero, its complexity, denoted $c(f)$, will be defined to be the minimum of the product $\deg F \deg p^1 \cdots \deg p^M$ over all polynomials $F$ and algebraic liftings $\Phi : U \rightarrow \R^M$ with associated annihilating map $p(x,y)$ such that $f(x) = F(x,\Phi(x))$ and $p(x,\Phi(x)) = 0$ for all $x \in U$. (For definiteness, let the complexity of the zero function simply equal zero.) It is easy to see that $c(f) \leq \deg f$ when $f$ is a nontrivial polynomial: one can use the trivial lifting $\Phi(x) = x$ with associated annihilating map $p(x,y) = -x + y$ for $M = d$ to lift $f(x)$ to $f(y)$, for example. Using the lifting constructed in the proof of Proposition \ref{itsnash}, one can also see that this definition of complexity is never larger than that of Ramanakoraisina \cite{ramanakoraisina89}, who proved a B\'{e}zout-type theorem for Nash functions. The following result shows that B\'{e}zout's Theorem also holds when using the current notion of complexity. 
\begin{theorem}[B\'{e}zout's Theorem for Nash functions]
Suppose that $f_1,\ldots,f_d$ are Nash functions on some open set $U \subset \R^d$. Then for any real $a_1,\ldots,a_d$, the number of nondegenerate solutions $x \in U$ of the system $f_1(x) = a_1,\ldots,f_d(x) = a_d$ is no greater than $c(f_1) \cdots c(f_d)$. \label{NashBezout}
\end{theorem}
The proof of Theorem \ref{NashBezout} is not particularly elaborate and is given below. The key computation in this regard is the following proposition.
\begin{proposition}
Suppose that $f(x) := (f_1(x),\ldots,f_d(x))$ is real analytic on $U$ and that each $f_i$ is a Nash function on $U$. Let $$\overline{F}(x,y_1,\ldots,y_d) := (F_1(x,y_1),\ldots,F_d(x,y_d),p_1(x,y_1),\ldots, p_d(x,y_d)),$$ where for each $j=1,\ldots,d$, $f_j(x) = F_j(x,\Phi_j(x))$ for the algebraic lifting $\Phi_j$ with associated annihilating map $p_j$. Then
\begin{equation} \begin{split} \left. \det D_{x,y_1,\ldots,y_d} \overline{F} \right|_{(x,\Phi_1(x),\ldots,\Phi_d(x))} 
= \det D_x f(x)  \prod_{j=1}^d \left. \det D_{y_j} p_j \right|_{(x,\Phi_j(x))} \end{split}
\label{jacobian} \end{equation}
for all $x \in U$. \label{jacprop}
\end{proposition}
\begin{proof}
To compute the Jacobian determinant $\det D_{x,y_1,\ldots,y_d} \overline{F}$, let $\overline{F}$ be regarded as a column vector and let the derivatives with respect to $x^1,\ldots,x^d$ and $y_1^1,\ldots,y_1^{M_1}$ through $y_d^{1},\ldots,y_d^{M_d}$ correspond to columns of the Jacobian.  Without loss of generality, each column $i \in \{1,\ldots,d\}$ may be replaced by 
\[ \left[ \frac{\partial}{\partial x^i} + \sum_{j=1}^{M_1} \frac{\partial (\Phi_1)^j}{\partial x^i} (x) \frac{\partial}{\partial y_1^j} + \cdots + \sum_{j=1}^{M_d} \frac{\partial (\Phi_d)^j}{\partial x^i} (x) \frac{\partial}{\partial y_d^j} \right] \overline{F}(x,y_1,\ldots,y_d) \]
because doing so is equivalent to applying a series of elementary column operations to the Jacobian matrix which add multiples of later columns to column $i$ (so none of these operations change the determinant). When the determinant is then evaluated at $(y_1,\ldots,y_d) = (\Phi_1(x),\ldots,\Phi_d(x))$, the chain rule guarantees that the first $d$ columns of the resulting Jacobian matrix must have the form
\[ \begin{bmatrix} \frac{\partial f_1}{\partial x^1}(x) & \cdots & \frac{\partial f_1}{\partial x^d}(x) \\
\vdots & \ddots & \vdots \\
\frac{\partial f_d}{\partial x^1}(x) & \cdots & \frac{\partial f_d}{\partial x^d}(x) \\ 0 & \cdots & 0 \\
\vdots & \cdots & \vdots \\
0 & \cdots & 0
\end{bmatrix}. \]
Expanding the determinant of the full Jacobian matrix in these first $d$ rows and using the fact that $p_i$ does not depend on $y_j$ when $i \neq j$ gives exactly \eqref{jacobian}.\end{proof}

\begin{proof}[Proof of Theorem \ref{NashBezout}.]
Suppose there exists some $a = (a_1,\ldots,a_d) \in \R^d$ and distinct $x_1,\ldots,x_N \in U$ with $N > c(f_1) \cdots c(f_d)$ which are nondegenerate solutions of the system $f_1(x_i) = a_1,\ldots,f_d(x_i) = a_d$ for all $i \in \{1,\ldots,N\}$; for convenience, this will be abbreviated $f(x_i) = a$. Every such $x_i$ must also have the property that $\overline{x}_i := (x_i,\Phi_1(x_i),\ldots,\Phi_d(x_i))$ is a solution of the system $\overline{F}(\overline{x}_i) = (a,0,\ldots,0) \in \R^{d} \times \R^{M_1 + \cdots + M_d}$. Moreover, by \eqref{jacobian}, the Jacobian determinant of $\overline{F}$ at $\overline{x}_i$ will equal exactly
\[ \det D_x f(x_i)  \prod_{j=1}^d \det D_{y_j} p_j |_{(x_i,\Phi_j(x_i))}. \]
Because the Jacobian determinant $\det D_x f$ is nonzero at each $x_i$ by assumption, the Inverse Function Theorem implies the existence of real analytic maps $\varphi_i$ defined on a neighborhood of $a \in \R^d$ such that $\varphi_i(a) = x_i$ and $f(\varphi_i(a')) = a'$ for all $a'$ sufficiently near $a$. By further restricting the domain of the $\varphi_i$, one may further assume that $\det D_x f |_{\varphi_i(a')}$ is nonzero for all such $a'$ as well and that $\varphi_i(a') \neq \varphi_j(a')$ when $i \neq j$. For each $i = 1,\ldots,N$, the quantity
\[ B_i(a') := \prod_{j=1}^d \det D_{y_j} p_j |_{ ( \varphi_i(a'), \Phi_j(\varphi_i(a')))} \]
is a real analytic function of $a'$ which is not identically zero. The product $B_1(a') \cdots B_N(a')$ is thus also a real analytic function of $a'$ defined on a neighborhood of $a$ and not identically zero. It is therefore possible to find an $a'$ belonging to any given neighborhood of $a$ such that
\[ \det D_x f(\varphi_i(a'))  \prod_{j=1}^d \det D_{y_j} p_j |_{(\varphi_i(a'),\Phi_j(\varphi_i(a')))} \neq 0 \]
for each $i \in \{1,\ldots,N\}$. It follows that $\overline{x}'_{i} := (\varphi_i(a'),\Phi_1(\varphi_i(a')),\ldots,\Phi_d(\varphi_i(a')))$ will be a nondegenerate solution of $\overline{F}(\overline{x}'_{i}) = (a',0,\ldots,0)$ for each $i$. Thus the system $\overline{F}(\overline{x}) = (a',0,\ldots,0)$ must consequently have at least $N$ distinct nondegenerate solutions $\overline{x}'_i$. Now B\'{e}zout's Theorem \cite{fulton1984}*{Chapter 8, Section 4} guarantees that the number of complex isolated solutions of the system $\overline{F}(\overline{x}) = (a',0,\ldots,0)$ is at most the product of degrees of the polynomials making up the system. Real nondegenerate solutions remain nondegenerate and therefore isolated when regarded as belonging to the complex solution set, so it follows that $N \leq \prod_{j=1}^d \deg F_j \deg p_j^1 \cdots \deg p_j^{M_j}$. Choosing each $F_j$ and the associated $\Phi_j$ and $p_j$ so that the product of degrees is as small as possible gives exactly that $N \leq c(f_1) \cdots c(f_d)$.
\end{proof}

Nash functions enjoy the nice property that they are closed under a rather long list of natural operations. For example, if $f$ and $g$ are Nash on some open set $U$, then both $fg$ and $f+g$ are Nash and in both cases, the complexities of $fg$ and $f+g$ are both bounded by $c(f) c(g)$. (The proof is elementary in both cases when one recognizes that any algebraic lifting maps $\Phi_i : U \rightarrow \R^{M_i}$ for $i=1,\ldots,2$ induce an algebraic lifting map $\overline{\Phi} : U \rightarrow \R^{M_1+M_2}$ by $\overline{\Phi}(x) = (\Phi_1(x),\Phi_2(x))$ with an associated annihilating map $\overline{p}$ formed by simply concatenating the maps $p_1$ and $p_2$ after interpreting each as constant in the variables of $\R^{M_1} \times \R^{M_2}$ it does not explicitly depend on). Other less obvious but true facts are that $1/f$ and $\sqrt{f}$ are Nash on the set where $f > 0$ and have complexity at most $2c(f)$. 
The key property of Nash functions as they pertain to Theorem \ref{mainineq} is as follows.
\begin{lemma}
If $f_1,\ldots,f_{d+1}$ are Nash functions on some open set $U \subset \R^{d}$, then the function \label{ratiolemma}
\begin{equation} \det \frac{ \partial (f_1,\ldots,f_d)}{\partial x}  \left( \det \frac{\partial (f_2,\ldots,f_{d+1})}{\partial x} \right)^{-1} \label{ratiofn} \end{equation}
is Nash on the set $U' := \set{x \in U}{\det \partial (f_2,\ldots,f_{d+1})/\partial x \neq 0}$. Its complexity is at most $(c(f_1) \cdots c(f_{d+1}))^2$.
\end{lemma} 
\begin{proof}
As before, for each $i=1,\ldots,d+1$, let $\Phi_i(x)$ be an algebraic lifting map with annihilating polynomial $p_i(x,y_i)$ and let $F_i(x,y_i)$ be such that $f_i(x) = F_i(x,\Phi_i(x))$ on $U$.
Let $\varphi(x)$ denote the function \eqref{ratiofn} and consider the polynomials
\begin{align*}
H_1(x,y_1,\ldots,y_{d+1}) & := \det \frac{\partial (F_1,\ldots,F_d,p_1,\ldots,p_d)}{\partial (x,y_1,\ldots,y_d)} \det D_{y_{d+1}} p_{d+1} , \\
H_2(x,y_1,\ldots,y_{d+1}) & := \det \frac{\partial (F_2,\ldots,F_{d+1},p_2,\ldots,p_{d+1})}{\partial (x,y_1,\ldots,y_d)} \det D_{y_{1}} p_{1}, 
\end{align*}
where each $p_i$ and $F_i$ is understood to depend on $x$ and $y_i$ only (i.e., to be independent of $y_j$ when $j \neq i$).
Let
\begin{align*}
\overline{\Phi}(x) & := (\Phi_1(x),\ldots,\Phi_{d+1}(x),\varphi(x)), \\ \overline{p}(x,y_1,\ldots,y_{d+1},z) & := (p_1(x,y_1),\ldots,p_{d+1}(x,y_{d+1}),\\ & \qquad H_1(x,y_1,\ldots,y_{d+1}) - z H_2(x,y_1,\ldots,y_{d+1})).
\end{align*} 
By \eqref{jacobian}, it must be the case that
\[ H_1 (x,\Phi_1(x),\ldots,\Phi_{d+1}(x)) = \det \frac{\partial (f_1,\ldots,f_d)}{\partial x}(x)  \prod_{j=1}^{d+1} \det D_{y_{j}} p_{j} |_{(x,\Phi_j(x))} \]
and
\[ H_2(x,\Phi_1(x),\ldots,\Phi_{d+1}(x)) = \det \frac{\partial (f_2,\ldots,f_{d+1})}{\partial x}(x) \prod_{j=1}^{d+1} \det D_{y_{j}} p_{j} |_{(x,\Phi_j(x))}, \]
which means that $H_1 - z H_2$ vanishes identically when evaluated at $(x,\overline{\Phi}(x))$ for $x \in U'$. Likewise, all other components of $\overline{p}$ are identically zero as a function of $(x,\overline{\Phi}(x))$.
The Jacobian determinant of $\overline{p}$ with respect to $y_1,\ldots,y_{d+1},z$ is exactly
\[ -H_2(x,y_1,\ldots,y_{d+1}) \prod_{j=1}^{d+1} \det D_{y_j} p_j|_{(x,y_j)}, \]
which is not identically zero on $U'$ when evaluated at the points $(x,\overline{\Phi}(x))$. This means that $\overline{\Phi}$ is an algebraic lifting map on $U'$ with annihilating map $\overline{p}$. Via this map, $\varphi$ lifts to $z$, so the complexity of $\varphi$ is therefore at most
\[ \left( \prod_{j=1}^{d+1} \prod_{i=1}^{M_i} \deg p_j^i \right) \cdot \deg (H_1 - z H_2). \]
Because each of $H_1$ and $H_2$ are products of derivatives of the functions $F_j$ and $p_j^i$ for $i \in \{1,\ldots,M_{j}\}$ and $j \in \{1,\ldots,d+1\}$, the degree of both $H_1$ and $H_2$ will be strictly less than $\prod_{j=1}^{d+1} \deg F_j \prod_{i=1}^{M_j} \deg p_j^i$, which means that the product of degrees of the polynomials in $\overline{p}$ never exceeds $\prod_{j=1}^{d+1} \deg F_j \prod_{i=1}^{M_j} (\deg p_j^i)^2$. Optimizing gives that the complexity of $\varphi$ cannot exceed $(c(f_1) \cdots c(f_{d+1}))^2$ as promised.
\end{proof}
The proof of Theorem \ref{mainineq} from Lemma \ref{geocorr} is now rather immediate.
\begin{proof}[Proof of Theorem \ref{mainineq}.]
By \eqref{volume}, it suffices to prove the inequality
\begin{equation} \int_{E_{N-1}} \left| \det \frac{\partial (X^{\alpha_1} f^{j_1},\ldots,X^{\alpha_d} f^{j_d})}{\partial x} \right| dx \leq 2^d D^{(2d+2)^N} \label{thecov} \end{equation}
for each $N$ when $\alpha_1,\ldots,\alpha_d$ are of generation at most $N-1$ and when $f = (f^1,\ldots,f^m)$ is a Nash mapping on some open $U \subset \R^d$ for which the complexity of $c(f^j) \leq D$ for all $j \in \{1,\ldots,m\}$. In the proof of Lemma \ref{geocorr}, the mappings $f^{(N)}$ are defined so that the components of $f^{(N)}$ are either components of $f^{(N-1)}$ or have the form $X^{(N)}_i g$ for some function $g$ being one of the components of $f^{(N-1)}$. By virtue of the definition \eqref{vfdef} of the vector fields $X^{(N)}_i$ and Lemma \ref{ratiolemma}, $X^{(N)}_i \varphi$ is a Nash function when $\varphi$ is Nash and each $X^{\alpha_i} f^{j_i}$ is Nash whenever $\alpha_i$ and $j_i$, $i=1,\ldots,d$ are as they appear in \eqref{volume}. By \eqref{ratiofn} induction on $N$,
\[ c(X^{(N)}_i X^{\beta} f^j) \leq \left( c( X^\beta f^j) \prod_{i'=1}^d c( X^{\alpha_{i'}} f^{j_{i'}}) \right)^2 \]
whenever $\beta$ is of generation at most $N-1$. By induction, each complexity on the right-hand side of this expression is at most $D^{(2d+2)^{N-1}}$, so $c(X^\alpha f^j) \leq (D^{(2d+2)^{N-1}})^{2(d+1)} = D^{(2d+2)^N}$ as desired.

Now consider the integral \eqref{thecov}. It is known that $|X^{\alpha_i} f_i| \leq 1$ at all points of $E_{N-1}$. By the change of variables formula, it follows that 
\[ \int_{E_{N-1}} \left| \det \frac{\partial (X^{\alpha_1} f^{j_1},\ldots,X^{\alpha_d} f^{j_d})}{\partial x} \right| dx \]
is bounded above by the Lebesgue measure of $[-1,1]^d$ times the maximal number of nondegenerate solutions $x \in U_{N-1}$ of the system $X^{\alpha_i}f^{j_i}(x) = c_i$, i.e.,
\[ \begin{split} \int_{E_{N-1}} & \left| \det \frac{\partial (X^{\alpha_1} f^{j_1},\ldots,X^{\alpha_d} f^{j_d})}{\partial x} \right| dx  \leq 2^d \sup_{(c_1,\ldots,c_d) \in \R} \# \Big\{ x \in U_{(N-1)} \ \Big| \\ & X^{\alpha_i} f^{j_1} (x) = c_i, i=1,\ldots,d  \text{ and } \left| \det \frac{\partial (X^{\alpha_1} f^{j_1},\ldots,X^{\alpha_d} f^{j_d})}{\partial x} (x) \right| \neq 0 \Big\}. \end{split} \]
By Theorem \ref{NashBezout} and the complexity bound, it follows that
\[ \int_{E_{N-1}} \left| \det \frac{\partial (X^{\alpha_1} f^{j_1},\ldots,X^{\alpha_d} f^{j_d})}{\partial x} \right| dx \leq 2^d \left( D^{(2d+2)^{N-1}} \right)^d. \]
So in particular, \eqref{volume} implies that
\[ \wt(x) \left| \det (X_1^{(N)},\ldots,X_d^{(N)}) \right| \Big|_x \geq c_{N,d} 4^{-d} D^{-(2d+2)^N} m^{-Nd} \wt(E_0) \]
for all $x \in E_N$, which gives \eqref{nondegenvecs}. 
\end{proof}
It is worth noting that with additional work, the dependence of this constant in \eqref{nondegenvecs} the parameter $D$ can be substantially improved in the case when all $f^j$ happen to be polynomials (because the current proof includes unnecessary duplication of equations in the lifting when all the original functions happen to be polynomials), but this will have no meaningful application at present and so is omitted.

\section{Sufficiency of the nondegeneracy condition}
\label{suffsec}

All of the necessary tools have now been assembled to prove Theorem \ref{characterthm} for polynomial mappings\footnote{In fact, the argument will show boundedness of \eqref{theop} when the associated $\phi$ is Nash. This is nice because the Implicit Function Theorem then allows one to establish boundedness of \eqref{theop} when the map $\gamma_t(x)$ is polynomial but does not have the form $(t,\phi(x,t))$.}.  Section \ref{collectineq} derives some key inequalities relating to nondegeneracy, and Section \ref{polysec} shows how these inequalities imply $L^{p_b}$--$L^{q_b}$ boundedness. Throughout this section, it will be assumed that $\phi(x,t)$ is a polynomial function on some open set $U \subset \R^{n} \times \R^{\de}$ with $D_x \phi(x,t)$ full rank on $U$, that $\gamma_t(x) := (x,\phi(x,t))$, and that $\pi(x,y)$ is given by \eqref{definingfn}. 
\subsection{Geometric sublevel set inequalities}
\label{collectineq}

\begin{lemma}
Suppose that the mappings $\gamma$ and $\pi$ (as they appear in \eqref{theop} and Section \ref{importsec}) are polynomial. Let $U \subset \R^{\de}$ be open and let $F$ be a nonnegative Nash function on $U$ and let $\wt$ be a nonnegative measurable function on $U$.
Let $E \subset U$ be any compact set of positive Lebesgue measure, let $\{\omega_i\}_{i=1}^n$ be a basis of $\R^n$ satisfying $|\det \{\omega_i\}_{i=1}^n| = 1$. If \label{ineqlemma}
\begin{equation} \sup_{t \in E} (F(t))^2 ||d_x \pi |_{(x, \gamma_t(x))} ||_\omega^2 \leq 1 \label{supassump} \end{equation}
(recall \eqref{typicalsum}) for some fixed $x \in \R^n$, then there exists a point $t_E \in E$ and bases $\{u_i\}_{i=1}^{\de}$ of $\R^{\de}$ and $\{\omega'_i\}_{i=1}^n$ of $\R^n$ such that $\wt(t_E) |\det \{u_i\}_{i=1}^{\de}| \gtrsim \wt(E)$, $|\det \{\omega'_i\}_{i=1}^n| = 1$, $\omega'_{k+1},\ldots,\omega'_n$ span $\ker D_x \pi_{(x,\gamma_{t_E}(x))}$, and
\begin{equation} \sum_{s=0}^{\min \{d,k\}} \sum_{i=1}^{\de} \sum_{i'=1}^k \sum_{i''=k+1}^n \left| F(t_E) (\tup{u_i}{s} \cdot \nabla_t^s) d_x \pi |_{(x,\gamma_{t}(x))} (\tup{\omega'_{i'}}{k-s},\tup{\omega'_{i''}}{s}) \right|^2 \lesssim 1 \label{basicinequality} \end{equation}
at $t = t_E$.
The implicit constants above do not depend on $E$, $\wt$, $x$, or $\{\omega_i\}_{i=1}^n$ and grow at most like a finite power of the degrees of $\pi$ and $\gamma$ and the complexity of $F$. (Note that $\tup{u_i}{0} \cdot \nabla_t^0$ is to be understood as the identity operator with no corresponding sums over the variable $i$.)
\end{lemma}
\begin{proof} 
Fix $N := \min \{d,k\}$. By \eqref{supassump} and the definition of $||d_x \pi||_{\omega}$ (combined with the fact that $d_x \pi$ is an alternating $k$-linear functional), it must be that
\begin{equation} \sup_{t \in E} F(t) \left|  d_x \pi |_{(x,\gamma_t(x))} (\omega_{i_1},\ldots,\omega_{i_k}) \right| \leq 1 \label{ftype1} \end{equation}
for any indices $1 \leq i_1 < \cdots < i_k \leq n$. Moreover, compactness of $E$ implies that there exists some nonzero $c$ such that
\begin{equation} \sup_{t \in E} |c t^i| \leq 1 \label{ftype2} \end{equation}
where, as usual, $t^1,\ldots,t^{\de}$ are the standard coordinate functions on $\R^{\de}$. Taking the functions appearing on the left-hand sides of \eqref{ftype1} and \eqref{ftype2} as components of a vector-valued mapping $f : \R^{\de} \rightarrow \R^{m}$ for suitable $m \leq n^k + \de$ gives by Theorem \ref{mainineq} that there exists an open set $U_N \subset \R^{\de}$ and families of vector fields $\{T_i^{(j)}\}_{i=1}^{\de }$ for $j \in \{1,\ldots,N\}$ such that
\begin{equation} \sup_{t \in U_N}  |  T^\beta F(t) d_x \pi |_{(x,\gamma_t(x))} (\omega_{i_1},\ldots,\omega_{i_k})| \leq 1 \label{derivestbegin} \end{equation}
for any $\beta$ of generation at most $N$ and any $1 \leq i_1 < \cdots < i_k \leq n$. Theorem \ref{mainineq} additionally guarantees the existence of a compact set $E_N \subset U_N \cap E$ such that $\wt(E_N) \gtrsim \wt(E)$ and $\wt(t) |\det \{T_i^{(j)}\}_{i=1}^{d_1}| \gtrsim |E|$ at all points $t \in E_N$, where in both cases the implicit constants depend only on dimensions, degrees, and $N$; as functions of the degrees of $\pi$ and $\gamma$ and the complexity of $F$, these implicit constants grow at most like a fixed power of those quantities. Squaring \eqref{derivestbegin} and summing gives that
\begin{equation} \sup_{t \in U_N} \mathop{\sum_{|\beta| \leq N}}_{\gen(\beta) \leq N} \sum_{i =1}^n | T^\beta F(t) d_x \pi |_{(x,\gamma_t(x))} (\tup{\omega_i}{k})|^2 \lesssim 1 \label{derivestapply0} \end{equation}
with implicit constant depending only on $d$, $\de $, $k$, and $N$.
For every $N' < N$ and every $i = 1,\ldots,\de $, there are coefficients $c_{i}^{i'}$ on $U_N$ with magnitude at most $2$ at each point of $U_N$ such that
\begin{equation} T^{(N)}_i = \sum_{i'=1}^{\de } c_{i}^{i'} T^{(N')}_{i'}. \label{compatible} \end{equation}
Let $t_E$ be any point in $E_N$ (which is nonempty because it has positive measure). 

Next, consider the application of Proposition \ref{sbfprop} to the sequence of subspaces $V_1 := \R^n \supset V_2 := \ker D_x \pi |_{(x,y)}$ and basis $\{\omega_i\}_{i=1}^n$. If $\{\omega'_i\}_{i=1}^n$ is the promised basis of $\R^n$ such that the span of $\omega'_{k+1},\ldots,\omega'_n$ is exactly $V_2$, then \eqref{sbf} and Proposition \ref{linprop} imply that 
\begin{equation}  \sum_{i=1}^n | T^\beta F(t) d_x \pi |_{(x,\gamma_t(x))} (\tup{\omega_i}{k})|^2 = \sum_{i=1}^n \left|T^\beta F(t) d_x \pi |_{(x,\gamma_t(x))} (\tup{\omega_i'}{k}) \right|^2 \label{fullsum} \end{equation}
when $t = t_E$ for each $|\beta| \leq N$ and further imply that $\omega'_i = \sum_j O_{ij} \omega_j$ for some orthogonal $O$. Since $|\det \{\omega_i\}_{i=1}^n| = 1$, this forces $|\det \{\omega'_i\}_{i=1}^n| = 1$ as well.

For any $\beta$ on the left-hand side of \eqref{fullsum}, suppose that the order of $\beta$ equals $s \in \{0,\ldots,\min\{d,k\}\}$. The right-hand side is not made larger if one restricts the expansion $\tup{\omega'_i}{k}$ in such a way that the indices $i_1,\ldots,i_{k-s}$ are constrained to belong to $\{1,\ldots,k\}$ and the indices $i_{k-s+1},\ldots,i_k$ are constrained to belong to $\{k+1,\ldots,n\}$. In other words,
\begin{equation} \begin{split} \sum_{i=1}^n & \left|T^\beta F(t) d_x \pi |_{(x,\gamma_t(x))} (\tup{\omega_i}{k}) \right|^2 \\ & \geq \sum_{i'=1}^k \sum_{i''=k+1}^n \left|T^\beta F(t) d_x \pi |_{(x,\gamma_t(x))} (\tup{\omega_{i'}'}{k-s},\tup{\omega'_{i''}}{s}) \right|^2 \end{split} \label{sumbound} \end{equation}
because the terms in the second sum are merely a subset of the terms on the right-hand side of \eqref{fullsum}.
 By definition of $d_x \pi$, 
\[ \begin{split} T^\beta F(t) d_x \pi &  |_{(x,\gamma_t(x))} (\omega'_{i_1},\ldots,\omega'_{i_k})  \\ & = T^\beta F(t) \det ( D_x \pi |_{(x,\gamma_t(x))} \omega'_{i_1} ,\ldots D_x \pi |_{(x,\gamma_t(x))} \omega'_{i_k}). \end{split} \]
Because each $\omega'_{k+1},\ldots,\omega'_n$ belongs to the kernel of $D_x \pi |_{(x,\gamma_{t_E}(x))}$, the vector $D_x \pi |_{(x,\gamma_{t}(x))} \omega'_{i}$ vanishes at $t=t_E$ when $i \geq k+1$. Thus each term in the expansion of $d_x \pi |_{(x,\gamma_t(x))} (\tup{\omega_{i'}'}{k-s},\tup{\omega'_{i''}}{s})$ vanishes to at least order $s$ at $t=t_E$. This means that when the right-hand side of \eqref{sumbound} is evaluated at this particular $t_E$, the differential operator $T^\beta$ may be replaced (without changing the value of the sum) by any order $s$ differential operator whose highest-order part agrees with $T^{\beta}$. In particular, if $\beta = ((N_1 ,\ldots,N_s),(i_1,\ldots,i_s))$ and if $u^{(j)}_i \in \R^{\de}$ is a constant vector equaling $T_i^{(j)}$ at this distinguished $t=t_E$, then defining $(u \cdot \nabla_t)^\beta = (u^{(N_s)}_{j_s} \cdot \nabla_t) \cdots (u^{(N_1 )}_{j_1} \cdot \nabla_t)$ and replacing $T^\beta$ by $(u \cdot \nabla_t)^{\beta}$ on the right-hand side of \eqref{sumbound} leaves the value unchanged.
By \eqref{compatible}, if one fixes $u_1,\ldots,u_{\de }$ to simply equal $u^{(N)}_1,\ldots,u^{(N)}_{\de }$, then 
\[ (u_{i_1} \cdot \nabla_t) \cdots (u_{i_s} \cdot \nabla_t) \]
is expressible as a linear combination of terms $T^{\beta}$ for generalized multiindices $\beta$ with $|\beta| = s$. The number of such terms needed is at most $\de ^s$, and the size of coefficient for each term is at most $2^{s}$.
Thus
\begin{align*}
  \sum_{i'=1}^k \sum_{i''=k+1}^n & \left|  (u_{i_1} \cdot \nabla_t) \cdots (u_{i_s} \cdot \nabla_t) F(t) d_x \pi |_{(x,\gamma_t(x))} (\tup{\omega_{i'}'}{k-s},\tup{\omega'_{i''}}{s})  \right|^2 
  \\
 & \leq \sum_{i'=1}^k \sum_{i''=k+1}^n \left| \mathop{\sum_{|\beta| = s}}_{\gen(\beta) = s} 2^s | T^\beta d_x F(t) \pi |_{(x,\gamma_t(x))} (\tup{\omega_{i'}'}{k-s},\tup{\omega'_{i''}}{s}) | \right|^2
\end{align*}
at $t = t_E$ and consequently by Cauchy-Schwarz, one has the inequality
\begin{align*}
(4 \de)^s & \mathop{\sum_{|\beta| = s}}_{\gen(\beta) \leq N} \sum_{i'=1}^k \sum_{i''=k+1}^n \left|T^\beta F(t) d_x \pi |_{(x,\gamma_t(x))} (\tup{\omega_{i'}'}{k-s},\tup{\omega'_{i''}}{s}) \right|^2
 \\ & \geq  \sum_{i=1}^{\de} \sum_{i'=1}^k \sum_{i''=k+1}^n \left| F(t) (\tup{u_i}{s} \cdot \nabla_t^s) d_x \pi |_{(x,\gamma_t(x))} (\tup{\omega_{i'}'}{k-s},\tup{\omega'_{i''}}{s})  \right|^2
\end{align*}
at $t = t_E$; note in particular that $F(t)$ may pass outside the derivatives $\tup{u_i}{s} \cdot \nabla_t^s$ because all terms arising from the product rule which involve derivatives of $F$ must vanish because $d_x \pi |_{(x,\gamma_t(x))} (\tup{\omega_{i'}'}{k-s},\tup{\omega'_{i''}}{s})$ vanishes to order $s$ at $t = t_E$.
Summing over $s$ and recalling \eqref{derivestapply0} and \eqref{sumbound} gives \eqref{basicinequality} with a constant that depends only on $n, k$, and $\de$. The condition $\wt(t_*) |\det \{u_i\}_{i=1}^{\de}| \gtrsim \wt(E)$ with a constant growing at most like a power of the degree of $\gamma$ or $\pi$ or complexity of $F$ is a consequence of the analogous property of the vector fields $\{T^{(N)}_i\}_{i=1}^{\de}$ at the point $t_E$, which is guaranteed by inequality \eqref{nondegenvecs} of Theorem \ref{mainineq}.
\end{proof}

\begin{proposition}
Under the same hypotheses as Lemma \ref{ineqlemma}, let $t \in \R^{d_1}$, $\{u_i\}_{i=1}^{\de}$, and $\{\omega'_i\}_{i=1}^n$ be as described there. Fix $y := \gamma_{t_E}(x)$ and let $Q$ be defined as in \eqref{qdef} for a given choice $\{z_i\}_{i=1}^d$ of orthonormal basis of $\ker D_x \pi |_{(x,y)}$.  \label{comparetoQ}
Let $\{v_i\}_{i=1}^k$ be the basis of $\R^k$ satisfying 
\begin{equation} v_i := D_x \pi|_{(x,y)} \omega'_i \qquad \forall i \in \{1,\ldots,k\}, \label{vdef1} \end{equation}
and let $\{w_i\}_{i=1}^d$ be the basis of $\R^d$ such that
\begin{equation} \omega'_{k+i} = \sum_{i'=1}^d w_{i}^{i'} z_{i'} \qquad \forall i \in \{1,\ldots,d\} \label{wdef2} \end{equation}
(i.e., the coefficients of the vectors $w_i$ are given by the change of basis matrix).
Then
\begin{equation}
\begin{split}
\sum_{s=0}^{\min \{d,k\}} & \sum_{i=1}^{\de} \sum_{i'=1}^k \sum_{i''=k+1}^n \left| (\tup{u_i}{s} \cdot \nabla_t^s) d_x \pi |_{(x,\gamma_t(x)} (\tup{\omega'_{i'}}{k-s},\tup{\omega'_{i''}}{s})  \right|^2 \\ & = |\det \{v_i\}_{i=1}|^2 ( \mathcal{Q} [ \{u_i\}_{i=1}^{\de }, \{v_i^*\}_{i=1}^k, \{w_i\}_{i=1}^d ])^2
\end{split} \label{compare}
\end{equation}
at $t = t_E$,
 where $\mathcal Q$ is derived from $Q$ as in \eqref{qsumdef} and $\{v_i^*\}_{i=1}^k$ is the dual basis of $\{v_i\}_{i=1}^k$ as in \eqref{inversionthing}.
\end{proposition}
\begin{proof}
As observed earlier, $D_x \pi |_{(x,\gamma_t(x))} \omega'_i$ vanishes to first order (or more) at the chosen point $t=t_E$ when $i \geq k+1$, which means that when $i'_1,\ldots,i'_{k-s} \leq k$ and $i''_1,\ldots,i''_s \geq k+1$,
\[ \begin{split}
& (u_{i_1} \cdot \nabla_t) \cdots (u_{i_s} \cdot \nabla_t) \det (D_x \pi \omega'_{i'_1},\ldots,D_x \pi \omega'_{i'_{k-s}} D_x \pi \omega'_{i''_1},\ldots,D_x \pi \omega'_{i''_s}) \\ & 
= (u_{i_1} \cdot \nabla_t) \cdots (u_{i_s} \cdot \nabla_t) \det (v_{i'_1},\ldots,v_{i'_{k-s}}, D_x \pi \omega'_{i''_1},\ldots,D_x \pi \omega'_{i''_s})
\end{split} \]
for any $\{u_i\}_{i=1}^{\de }$ at the chosen $t$.
Defining
\[ \Theta (u,w) :=  (u \cdot \nabla_t) D_x \pi |_{(x,\gamma_t(x))} \left[ \sum_{i=1}^d w^i z_i \right] \]
for any $(u,w) \in \R^{\de} \times \R^{d}$
implies that
\[ \begin{split} (u_{i_1} \cdot \nabla_t) & \cdots (u_{i_s} \cdot \nabla_t) \det (v_{i'_1},\ldots,v_{i'_{k-s}}, D_x \pi \omega'_{i''_1},\ldots,D_x \pi \omega'_{i''_s}) \\ & = (u_{i_1} \cdot \nabla_{t'})  \cdots (u_{i_s} \cdot \nabla_{t'})  \det (v_{i'_1},\ldots,v_{i'_{k-s}}, \Theta(t',w_{i''_1}),\ldots,\Theta(t',w_{i''_s})) \end{split}
\]
for any $t' \in \R^{\de}$ by virtue of the product rule (since, on both sides, the only terms which can possibly be nonzero are those terms in which each one of the final $s$ entries of the determinant is differentiated with respect to exactly one of the derivatives with respect to $t$ or $t'$). Thus
\[ \begin{split}
\sum_{s=0}^{\min \{d,k\}} & \sum_{i=1}^{\de} \sum_{i'=1}^k \sum_{i''=k+1}^n \left| (\tup{u_i}{s} \cdot \nabla_t^s) d_x \pi |_{(x,\gamma_t(x))} (\tup{\omega'_{i'}}{k-s},\tup{\omega'_{i''}}{s})  \right|^2 \\
& = \sum_{s=0}^{\min\{d,k\}}  \sum_{i=1}^{\de } \sum_{i'=1}^{k} \sum_{i''=1}^d \left| (\tup{u_i}{s} \cdot \nabla_{t}^s) \det (\tup{v_{i'}}{k-s},\tup{\Theta(t, w_{i''})}{s}) \right|^2.
\end{split}
 \]
The identity \eqref{compare} is now a direct consequence of \eqref{inversionthing} because $v \cdot \Theta(u,w) = Q(u,v,w)$ by virtue of the fact that $(u \cdot \nabla_t) D_x \pi |_{(x,\gamma_t(x))}$ is exactly equal to $(u \cdot \nabla_t) D_x \phi(x,t)$ at $t=t_E$ by virtue of the definition \eqref{definingfn}.  
\end{proof}

\subsection{Proof of sufficiency in the polynomial case}
\label{polysec}
Given a point $(x_*,t_*) \in \R^{n} \times \R^{\de}$ and a polynomial mapping $\phi(x,t)$ defined near $(x_*,t_*)$, suppose that $Q$ (given by \eqref{qdef}) is nondegenerate at $(x_*,t_*)$. Let $y_* := \gamma_{t_*}(x_*)$. Corollary \ref{stability} guarantees the existence of a finite constant $c$ depending only on $Q$ as defined at $(x_*,t_*)$ and a neighborhood of $(x_*,t_*)$ depending only on the $C^3$ norm of $\phi$ near $(x_*,t_*)$ such that when $\tilde \phi$ is any mapping sufficiently close to $\phi$ on that neighborhood (also measured in the $C^3$ norm), then \eqref{qnondegen2} holds for the same fixed value of $c$ for all $\tilde Q$ computed from this $\tilde \phi$ at any point of the neighborhood. Let $\eta(x,y)$ be any continuous cutoff function such that $\eta(x,\gamma_t(x))$ is compactly supported within this given neighborhood.

To prove that the operator \eqref{theop}, defined using this $\phi$ and $\eta$, maps $L^{p_b}$ to $L^{q_b}$, it suffices by Theorem \ref{fromtest} and \eqref{intdef} to establish that there is some constant $C$ such that
\begin{equation} \int_{\R^{\de}} \frac{(\eta(x,\gamma_t(x)))^{p_b'} dt}{||d_x \pi (x,\gamma_t(x))||_\omega^{p_b'-1}} \leq C \label{biggoal} \end{equation}
for all $x \in \R^n$ and all choices of basis $\{\omega_i\}_{i=1}^n$ which are normalized such that $|\det (\omega_1,\ldots,\omega_n)|=1$. We will show that, in fact, this estimate holds for all $\tilde \phi$ sufficiently near to $\phi$ with constant $C$ that grows at most like some power of the degree of the associated polynomials for $\phi$ and $\pi$.

Let $x \in \R^n$ and the basis $\{\omega_i\}_{i=1}^n$ be fixed, and assume without loss of generality that the integral on the left-hand side of \eqref{biggoal} is nonzero. Let $E \subset \R^{\de}$ be any compact set of positive measure contained in the support of $\eta(x,\gamma_t(x))$ for this fixed $x$. Let $F$ be the Nash function $||d_x \pi |_{x,\gamma_t(x)}||_{\omega}^{-1}$ (whose complexity grows no faster than some power of the degrees of $\pi$ and $\phi$ and which never vanishes on the support of the integral by the assumption we make that $D_x \phi$ is full rank at $(x_*,t_*)$) and let $\wt(t) := ||d_x \pi |_{x,\gamma_t(x)}||_{\omega}^{- n\de / (dk)}$. Because one trivially has
\[ \sup_{t \in E} (F(t))^2 ||d_x \pi |_{(x,\gamma_t(x))}||_\omega^2 \leq 1, \]
Lemma \ref{ineqlemma} applies; let $t_E$ be the distinguished value of $t$ promised by the lemma, and let $Q$ be the trilinear functional \eqref{qdef} computed at $(x,t_E)$. Because the support of $\eta$ has been made sufficiently small, the quantity ${\mathcal{Q}}$ satisfies \eqref{qnondegen2} with a constant $c$ (already discussed above) that is independent of $E$, $x$, and $t_E$. Combining \eqref{basicinequality} and \eqref{compare} gives that
\[ c (F(t_E))^2 |\det \{v_i\}|_{i=1}^d|^2 |\det \{u_i\}_{i=1}^{\de }|^{\frac{2s}{\de }} |\det \{v^*_i\}_{i=1}^k|^{\frac{2s}{k}} |\det \{w_i\}_{i=1}^d|^{\frac{2s}{d}} \lesssim 1 \]
for $s := dk/n$
with an implicit constant that does not depend on $x$ or $E$ and that grows at most like a power of the degree of $\pi$ (or equivalently, $\phi$). Because the bases $\{v_i\}_{i=1}^k$ and $\{v_i^*\}_{i=1}^k$ are dual, this means that
\begin{equation} c (F(t_E))^2 |\det \{u_i\}_{i=1}^{\de }|^{\frac{2s}{\de }} |\det \{v_i\}_{i=1}^k|^{2-\frac{2s}{k}} |\det \{w_i\}_{i=1}^d|^{\frac{2s}{d}} \lesssim 1. \label{mybound} \end{equation}
Lemma \ref{ineqlemma} guarantees that $\{u_i\}_{i=1}^{\de }$ satisfies $\wt(t_E) |\det \{u_i\}_{i=1}^{\de }| \gtrsim \wt(E)$.
The weight $\wt$ and function $F$ have the property that $\wt = (F)^{\de/s}$, so it follows that 
\begin{equation} \sqrt{c} (\wt(E))^{\frac{s}{\de }} |\det \{v_i\}_{i=1}^k|^{1-\frac{s}{k}} |\det \{w_i\}_{i=1}^d|^{\frac{s}{d}} \lesssim 1 \label{mybound2} \end{equation}
with implicit constant depending only on $d,\de,k$ and the degree of $\pi$ (as always, with growth bounded like a power of this degree).

Lemma \ref{ineqlemma} also guarantees that
$|\det \{\omega'_i\}_{i=1}^n| = 1$. For each $i \in \{1,\ldots,k\}$, let $\omega''_i$ be the unique vector in the orthogonal complement of $\ker D_x \pi |_{(x,\gamma_{\tilde t}(x))}$ such that $\omega''_i - \omega'_i \in \ker D_x \pi|_{(x,\gamma_{\tilde t}(x))}$, and when $i > k$, let $\omega''_i := \omega'_i$. Because $\omega''_i = \omega'_i + \sum_{i'=k+1}^n \theta_{ii'} \omega'_i$ for suitable coefficients $\theta_{ii'}$ and each $i \in \{1,\ldots,k\}$, $|\det \{\omega''_i\}_{i=1}^n| = 1$ as well. 
Now consider the map $M : \R^n \rightarrow \R^n$ such that  $M \omega''_i = (D_x \pi|_{(x,\gamma_{\tilde t}(x))} \omega''_i, z_1 \cdot \omega''_i,\ldots,z_d \cdot \omega''_i)$ for each $i$, where $z_1,\ldots,z_{d}$ is the orthonormal basis of $\ker D_x \pi |_{(x,\gamma_{t_E}(x))}$ from which $Q$ is defined. If $i \in \{1,\ldots,k\}$, then $M \omega''_i = (v_i,0,\ldots,0)$, and if $i \geq k$, then $M \omega''_i = (0,w_{i-k}^1,\ldots,w_{i-k}^d)$ for $w_{i}^{i'}$ as in Proposition \ref{comparetoQ}. It must be the case, then, that
\[ |\det \{v_i\}_{i=1}^k| |\det \{w_i\}_{i=1}^d| = |\det M| |\det \{\omega''_i\}_{i=1}^n| = |\det M|. \] 

Now consider the matrix $MM^T$. Because $\{z_1,\ldots,z_d\}$ is an orthonormal basis of $\ker D_x \pi|_{(x,y_j)}$, it follows that $MM^T$ has block structure 
\[ \begin{bmatrix} (D_x \pi(x,y)) (D_x \pi(x,y))^T & 0 \\ 0 & I \end{bmatrix} \]
(with $y := \gamma_{t_E}(x)$)
which means that $|\det M| = \sqrt{ \det (D_x \pi(x,y)) (D_x \pi(x,y))^T}$, (which, by Proposition 1 of \cite{testingcond}, equals  $||d_x \pi(x,y)||$). 
In particular, this means that $|\det M|$ is a continuous function of $(x,y)$ which is nonvanishing at $(x_*,y_*)$ and depends only on first derivatives of $\phi$. It may therefore be assumed that the support of the cutoff function $\eta$ has been sufficiently restricted so that $|\det M|$ is bounded uniformly below there by a constant that depends only on $\phi$ at $(x_*,t_*)$ and is stable when $\phi$ is replaced by any $\tilde \phi$ which is close to it in the $C^3$ sense. Thus \eqref{mybound2} implies that
\begin{equation} \int_{E} \frac{dt}{||d_x \pi (x,\gamma_t(x))||_\omega^{p_b'-1}} \leq C_0 \label{sublevgoal} \end{equation}
for some constant which is independent of $x, \{\omega_i\}_{i=1}^n,$ and $E$. The analogous estimate continues to hold when $\phi$ is replaced by any $\tilde \phi$ sufficiently close to it in the $C^3$ sense with a new constant $\tilde C_0$ which grows at most like some power of the degree of the perturbed map $\tilde \phi$.  Since $E$ may be any arbitrary compact set in the support of $\eta(x,\gamma_t(x))$ and since $\eta$ is bounded, it follows that \eqref{biggoal} must indeed hold uniformly in $x$ and $\{\omega_i\}_{i=1}^{n}$ and must continue to hold for all $\tilde \phi$ near $\phi$ with constant $C$ growing at most like some power of the degree of $\tilde \phi$.  Thus, not only will the operator \eqref{theop} map $L^{p_b} \rightarrow L^{q_b}$, but there must also be a form of stability: if $\phi$ is replaced by $\tilde \phi$ sufficiently close to it in the $C^3$ sense, then the new operator $\tilde T$ must also map $L^{p_b}$ to $L^{q_b}$ and the operator norm should be bounded by some power of the degree of $\tilde \phi$. This stability result will be a critical component in the passage from polynomial to smooth maps that is undertaken in Section \ref{npasec}.

\section{Necessity of the nondegeneracy condition}
\label{nesssec}

This section establishes the necessity of nondegeneracy for any operator $T$ which maps $L^p(\R^{\en})$ to $L^q(\R^n)$ for some pair $(1/p,1/q)$ close to the best possible pair $(1/p_b,1/q_b)$. Section \ref{tomodelsec} establishes a reduction to a model case in which $\phi(x,t)$ is bilinear in $x$ and $t$, and Section \ref{modelsec} then establishes necessity of nondegeneracy in this model case. The proof is built around the construction of appropriate Knapp-type examples for the model operator combined with a process (Proposition \ref{convertknapp}) to appropriately modify these examples so that they apply to non-model cases as well.

For convenience, it will be assumed throughout this section that the operator $T$ is being studied in a neighborhood of the origin $(0,0) \in \R^n \times \R^{\de}$.

\subsection{Reduction to a model case}
\label{tomodelsec}

Let $\phi(x,t)$ be a smooth $\R^k$-valued function on a neighborhood of the origin $(0,0) \in \R^n \times \R^{\de}$ and suppose that $D_x \phi$ is rank $k$ at $(0,0)$ and that $\phi(0,0) = 0$. Let $T$ be the operator \eqref{theop} associated to $\phi$, under the assumption that $\eta : \R^n \times \R^{\en} \rightarrow \R$ is a continuous function which is nonvanishing at $(0,0)$ and has the property that the support of $\eta(x,\gamma_t(x))$ is compact and contained in the given domain of $\phi$.  Rotating coordinates of $\R^n$ as necessary, it may be assumed without loss of generality that $\partial \phi / \partial x^i = 0$ at the origin for $i \in \{1,\ldots,d\}$. Every vector $x \in \R^n$ will be regarded as a pair $(x_0,x_1) \in \R^d \times \R^k$, and for simplicity, the quantity $\phi((x_0,x_1),t)$ will be written $\phi(x_0,x_1,t)$.
Let  $M := D_{x_1} \phi |_{(0,0)}$,
\begin{equation} \Theta (u,w) := \sum_{i=1}^{d} \sum_{i''=1}^{\de} \left. \frac{\partial^2 \phi}{\partial t^i \partial x_0^{i''}} \right|_{(0,0,0)} u^{i} w^{i''}, \label{thetadef} \end{equation}
and
\begin{equation} \tilde \phi(x_0,x_1,t) := M x_1 + \Theta(t,x_0). \label{modeldef} \end{equation}
Just as was the case for $\phi(x,t)$ and $\gamma_t(x)$, let $\tilde \gamma_t(x_0,x_1) := (t, \tilde \phi(x_0,x_1,t))$.
\begin{proposition}
Suppose there exists a sequence of compact sets $G_\delta \subset \R^{\en}$ and $F_\delta \subset \R^n$ for all sufficiently small positive $\delta$ which satisfy the following properties: \label{convertknapp} 
\begin{enumerate}
\item Each set $G_\delta$ and $F_\delta$ is contained in a ball of radius at most $C \delta^{-\epsilon}$ centered at the origin for some $\epsilon > 0$ and $C < \infty$.
\item There exists some $\epsilon' > 0$ such that 
\begin{equation} \limsup_{\delta \rightarrow 0^+} \delta^{\epsilon'} \frac{\int \chi_{F_\delta} (x_0,x_1) \chi_{G_\delta}(\tilde \gamma_{t}(x_0,x_1)) dx_0 dx_1 dt}{|F_\delta|^{1/q_b'} |G_\delta|^{1/p_b}} > 0. \label{scalestoobig} \end{equation}
\item The slices of the set $G_\delta$ with respect to the first $\de$ coordinates are compact and convex, i.e., 
\[ G_\delta(t) := \set{ y_1 \in \R^{k}}{ (t,y_1) \in G_\delta } \]
is compact and convex for all $t \in \R^{\de}$.
\item All slices $G_\delta(t)$ which are nonempty satisfy $|G_\delta(t)| \geq c \delta^{M}$ for some positive $M$ and $c$.
\item For all sufficiently small $\delta$, $|F_\delta| \geq c' \delta^{N}$ and $|G_\delta| \geq c' \delta^N$ for some positive $N$ and $c'$.
\end{enumerate}
Then for all pairs $(1/p,1/q)$ sufficiently near to $(1/p_b,1/q_b)$, $T$ is unbounded as a map from $L^p$ to $L^q$.
\end{proposition}
\begin{proof} 
By simply reparametrizing the $\delta$-dependence of the sets $F_\delta$ and $G_\delta$, substituting $\delta^{\theta}$ in the place of every $\delta$ for some small positive value of $\theta$, it may be assumed without loss of generality that $\epsilon \in (0, 1/(k+4))$ and that nonempty slices $G_\delta(t)$ satisfy $|G_\delta(t)| \geq c \delta^{1-(k+4) \epsilon}$.

Let $\phi_1 (x_0,x_1,t) := \phi(x_0,x_1 + R(x_0),t) - \phi(0,0,t)$ for some $\R^k$-valued homogeneous quadratic function $R$ to be determined momentarily. Because $\partial_{x_i^0} \phi$ vanishes at the origin for $i \leq d$, it follows that
\[ \partial_{x_0^i} |_{(0,0,0)} \phi_1 (x_0,x_1,t)  = 0 \text{ and } \partial_{t^i} \phi_1(x_0,x_1,t) = 0 \]
for each $i \in \{1,\ldots,d\}$ in the first equation and $i \in \{1,\ldots,\de\}$ in the second. Likewise
\[ \partial_{x_1^i} \phi_1 = \partial_{x_1^i}  \phi, \ \ \partial_{t^i x_0^{i'}}^2  \phi_1 = \partial_{t^i x_0^{i'}}^2  \phi, \ \text{ and } \ \partial^2_{t^{i} t^{i'}} \phi_1 = 0\]
when each quantity is evaluated at $(0,0,0)$ and 
\[ \partial^2_{x_0^i x_0^{i'}}  |_{(0,0,0)}  \phi_1(x_0,x_1,t) = \partial^2_{x_0^i x_0^{i'}}|_{(0,0,0)} \phi + D_{x_1} \phi|_{(0,0,0)} \partial^2_{x_0^i x_0^{i'}}|_{(0,0,0)} R(x_0). \]
Since $D_{x_1} \phi |_{(0,0,0)}$ is invertible, there is a choice of $R$ which makes all second derivatives $\partial^2_{x_0^i x_0^{i'}} \phi_1$ vanish at the origin.
By virtue of Taylor's theorem with remainder, then, one has the identity
\[ \phi_1(x_0,x_1,t) = M x_1 + \Theta(t,x_0) + E(x_0,x_1,t) \]
where $E$ is expressible as a quadratic function of $x_1$ plus a bilinear function of $x_1$ and $t$ and finitely many terms which equal cubic monomials in $(x_0,x_1,t)$ times smooth functions.  Consequently
\[ \phi_1(\delta x_0,\delta^2 x_1,\delta t) = \delta^2 (M x_1 + \Theta(t,x_0)) + O(\delta^{3-3 \epsilon}) \]
uniformly on the sets $K_\delta := \set{ (x_0,x_1,t)}{|x_0| + |x_1| + |t| < C \delta^{-\epsilon}}$ as $\delta \rightarrow 0$. For every small $\delta$, let $\phi_\delta(x_0,x_1,t) := \delta^{-2} \phi_1(\delta x_0,\delta^2 x_1,\delta t)$ and $\gamma_{\delta,t}(x_0,x_1) := (t,\phi_\delta(x_0,x_1,t))$.
Then
\begin{equation} \phi_\delta(x_0,x_1,t) = \tilde \phi(x_0,x_1,t) + O(\delta^{1-3\epsilon}) \label{ellipseclose} \end{equation}
uniformly on $K_{\delta}$ as $\delta \rightarrow 0$.

Now suppose $F_\delta, G_\delta$, and $\epsilon' > 0$ satisfy \eqref{scalestoobig} and that slices of $G_\delta$ have the required properties. Let $G'_\delta(t)$ be the set of all points $y_1'$ such that $|y'_1 - y_1| \leq \delta^{1- 4 \epsilon}$ for some $y_1 \in G_\delta(t)$ and let $G'_\delta$ be the set whose slices are $G'_\delta(t)$ for each $t$.  By the John Ellipsoid Theorem, each nonempty $G_\delta(t)$ admits an ellipsoid $E \subset \R^k$ such that
\[ E \subset G_\delta(t) \subset k E \]
(where $k E$ is the ellipsoid with the same center as $E$ with all axes stretched by a factor of $k$).
From this it follows that $|E| \leq |G_\delta(t)| \leq k^k |E|$. As the measure of $G_\delta(t)$ is at least $c \delta^{1 - \epsilon (k+4)}$ and as no axis of $E$ can be longer than $C \delta^{-\epsilon}$, it follows that there is some fixed constant $C'$ depending only on $C$, $c$, and $k$ such that every axis of $E$ is at least $C '\delta^{1-5 \epsilon}$ for all sufficiently small $\delta$. In particular, for sufficiently small $\delta$, the ellipsoid $2k E$ will necessarily contain the vector sum of $G_\delta(t)$ and the ball of radius $\delta^{1-4 \epsilon}$. As a consequence, this means that $|G'_\delta| \leq (2k)^k |G_\delta|$ for all sufficiently small $\delta$. Moreover, for sufficiently small $\delta$ if $(x_0,x_1) \in F_\delta$ and $\tilde \gamma_t(x_0,x_1) \in G_\delta$, then $\tilde \phi(x_0,x_1, t)$ belongs to the slice $G_\delta(t)$, so by \eqref{ellipseclose}, $\phi_\delta(x_0,x_1,t) \in G'_\delta(t)$, or $\gamma_{\delta,t}(x_0,x_1) \in G'_\delta$. Thus 
\[ \int \chi_{F_\delta}(x_0,x_1) \chi_{G'_\delta} ( \gamma_{\delta,t}(x_0,x_1)) dx dt \geq \int \chi_{F_\delta}(x_0,x_1) \chi_{G_\delta}(\tilde \gamma_t(x_0,x_1)) dx dt. \]
By \eqref{scalestoobig}, it is necessarily the case that 
\begin{equation} \limsup_{\delta \rightarrow 0^+} \delta^{\epsilon'} \frac{\int \chi_{F_\delta} (x_0,x_1) \chi_{G_\delta'}(\gamma_{\delta,t}(x_0,x_1)) dx_0 dx_1 dt}{|F_\delta|^{1/q_b'} |G_\delta'|^{1/p_b}} > 0. \label{stoo2} \end{equation}
A series of changes of variables (replacing $x_0$ by $\delta^{-1} x_0$, $x_1$ by $\delta^{-2} x_1$, $t$ by $\delta^{-1} t$, and finally $x_1$ by $x_1 - R(x_0)$) and recalling that $\phi_1 (x_0,x_1,t) = \phi(x_0,x_1 + R(x_0),t) - \phi(0,0,t)$ and $\phi_\delta = \delta^{-2} \phi_1(\delta x_0,\delta^2 x_1,\delta t)$ 
yields the identity
\[ \begin{split}
\int  \chi_{F_\delta} (x_0,x_1) \chi_{G_\delta'}(\gamma_{\delta,t}(x_0,x_1))&  dx_0 dx_1 dt = \\
  \delta^{-2k - d - \de} \int & \chi_{F_\delta}(\delta^{-1} x_0,\delta^{-2} (x_1 - R(x_0))) \\ & \cdot \chi_{G'_\delta} (\delta^{-1} t, \delta^{-2} (\phi(x_0,x_1,t) - \phi(0,0,t))) dx dt.
\end{split} \]
Now let $F'_\delta$ and $G''_\delta$ be defined so that
\[ \chi_{F'_\delta}(x_0,x_1) := \chi_{F_\delta}(\delta^{-1} x_0,\delta^{-2} (x_1 - R(x_0))) \]
and
\[ \chi_{G''_\delta}(y_0,y_1) := \chi_{G'_\delta} (\delta^{-1} y_0, \delta^{-2} (y_1 - \phi(0,0,y_0))). \]
It follows that $|F'_\delta| = \delta^{d + 2k} |F_\delta|$ and $|G''_\delta| = \delta^{\de + 2k} |G'_\delta|$, and for the particular exponents $p_b$ and $q_b'$, 
\[ |F'_\delta|^{1/q_b'} |G''_\delta|^{1/p_b} = \delta^{2k+d+\de} |F_\delta|^{1/q_b'} |G_\delta'|^{1/p_b}. \]
It follows by \eqref{stoo2} that
\begin{equation*} \limsup_{\delta \rightarrow 0^+} \delta^{\epsilon'} \frac{\int \chi_{F_\delta'} (x_0,x_1) \chi_{G_\delta''}(\gamma_{t}(x_0,x_1)) dx_0 dx_1 dt}{|F_\delta'|^{1/q_b'} |G_\delta''|^{1/p_b}} > 0.  \end{equation*}
Because  $F_\delta'$ and $G_\delta''$ are contained in any given small balls around the origin in $\R^n$ and $\R^{\en}$ (respectively) for all sufficiently small $\delta$, continuity of $\eta$ and the nonvanishing of $\eta(0,0)$ give that
\[ \begin{split}
\eta(0,0) \int & \chi_{F_\delta'} (x_0,x_1) \chi_{G_\delta''}(\gamma_{t}(x_0,x_1)) dx_0 dx_1 dt \\ & \approx \int \chi_{F'_\delta}(x) \chi_{G''_\delta}(\gamma_t(x)) \eta(x,\gamma_t(x)) dx dt \end{split} \]
for all sufficiently small $\delta$ (with implicit constants for both the lower and upper bounds which tend to $1$ as $\delta \rightarrow 0^+$). Finally, because $\epsilon'$ is positive and the measures of $F'_\delta$ and $G''_\delta$ are bounded between some fixed positive and negative powers of $\delta$, it follows that
\begin{equation*} \limsup_{\delta \rightarrow 0^+}  \frac{\left| \int \chi_{F_\delta'} (x) \chi_{G_\delta''}(\gamma_{t}(x)) \eta(x,\gamma_t(x)) dx dt \right|}{|F_\delta'|^{1/q'} |G_\delta''|^{1/p}}  = \infty  \end{equation*}
for all $1/p$ and $1/q$ sufficiently close to $(1/p_b,1/q_b)$. This means that $T$ is not of restricted weak type $(p,q)$ and hence not bounded from $L^p$ to $L^q$, either.
\end{proof}

\subsection{Analysis of the model case}
\label{modelsec}
Recall the definition of the model case \eqref{modeldef}. By the linear change of variables $x_1 \rightarrow M^{-1} x_1$ it may be assumed without loss of generality that $M$ is the identity (because the effect on $L^q(\R^n)$ norms is merely to multiply by a nonzero constant) and consequently it may be assumed that
\[ \tilde \phi(x_0,x_1,t) := x_1 + \Theta(t,x_0) \text{ and } \tilde \gamma_t(x_0,x_1) := (t, x_1 + \Theta(t,x_0)) \]
for $\Theta$ given by \eqref{thetadef}. Suppose that $\{\omega_i\}_{i=1}^n$ is any given basis of $\R^n$ and that
\[ B_\omega := \set{ \sum_{i=1}^n \theta_i \omega_i }{ \sum_{i=1}^n \theta_i^2 \leq 1}. \]
Let $L$ be any linear map from $\R^n$ to $\R^k$ (where as usual, $k < n$), and let $\omega$ be the $n \times n$ matrix whose $i$-th column equals $\omega_i$ as expressed in standard coordinates. The set $L B_\omega$ is then simply equal to $L \omega \mathbb{B}_n$ when $\mathbb{B}_n$ is the $n$-dimensional Euclidean unit ball. By the Singular Value Decomposition, there exist orthogonal matrices $O_1$ and $O_2$ such that $L \omega = O_1 D O_2$ for some $k \times n$ matrix $D$ whose only nonzero elements are on the diagonal (and are nonnegative). Since $\mathbb{B}_n = O_2 \mathbb{B}_n$, it follows that $L B_\omega = O_1 D \mathbb{B}_n$, which means that $L B_\omega$ is an ellipsoid centered at the origin in $\R^k$ and its semiaxes are simply the diagonal elements of $D$. The volume of $L B_\omega$ is exactly $|\mathbb{B}_k| \sqrt{\det (L \omega) (L \omega)^T}$, and by Proposition 1 of \cite{testingcond}, this quantity also equals
\begin{equation} |\mathbb{B}_k| \sqrt{ \frac{1}{k!} \sum_{i=1}^{n} \left| \det (\tup{L \omega_{i}}{k}) \right|^2}. \label{multidef} \end{equation}
If $L_t$ is taken to be the Jacobian matrix $D_{x} \tilde \phi(x_0,x_1,t)$ (i.e., the Jacobian of $\tilde \phi$ with respect to the $x$ variables) and if $\{\omega_i\}_{i=1}^n$ is chosen to have the special form that $\omega_i := (w_i,0) \in \R^d \times \R^k$ when $i \in\{1,\ldots,d\}$ and $\omega_{i+d} := (0,v_i) \in \R^d \times \R^k$ for $i \in \{1,\ldots,k\}$, then for any fixed $t$, the expression \eqref{multidef} becomes
\[ |\mathbb{B}_k| \sqrt{ \sum_{s=0}^{\min\{d,k\}} \frac{1}{s!(k-s)!} \sum_{i'=1}^{k} \sum_{i''=1}^d \left| \det ( \tup{v_{i'}}{k-s},\tup{\Theta(t,w_{i''})}{s}) \right|^2} \]
because $L \omega_{i''} = \Theta(t,w_{i''})$ for $i'' = 1,\ldots,d$ and $L \omega_{i'+d} = v_{i'}$ for $i'=1,\ldots,k$ (where $\Theta$ is as in \eqref{inversionthing}).

Heuristically, in order to invoke Proposition \ref{convertknapp} to establish unboundedness of $T$, the goal is to construct Knapp-type example sets $F_\delta$ and $G_\delta$ in such a way that the $F_\delta$ sets are ellipsoids $B_\omega$ (for some basis $\omega$ chosen to depend on $\delta$) and the $G_\delta$ sets have the property that their nonempty slices $G_\delta(t)$ are sets of the form $L_t B_\omega$ for all $t$ in some ellipsoid of its own. The problem is that, although there is an exact formula for the measure of $L_t B_\omega$, the dependence on $t$ is opaque. The solution will be to establish that the size of $L_t B_\omega$ is often as large as it can reasonably be and then to restrict to the subset of those $t$ such that good lower and upper bounds for $|L_t B_\omega|$ hold. (As a consequence, the set of $t$ for which $G_\delta(t)$ is nonempty will not be an ellipsoid {per se} but merely some substantial fraction of a suitable ellipsoid.)
At this point, we need an auxiliary result concerning homogeneous polynomials on the unit ball:
\begin{proposition}
For each natural number $s$ and each $\epsilon \in (0,1)$, there is a positive constant $c$ such that
\begin{equation} \left| \set{ t \in {\mathbb{B}}_{\de} }{ |P(t)| < c \sup_{t' \in \mathbb{B}_{\de}} |P(t')|} \right| \leq \epsilon |\mathbb{B}_{\de}| \label{elipsedistro} \end{equation}
for all real homogeneous polynomials $P$ of degree $s$ on $\R^{\de}$, where $\mathbb{B}_{\de}$ is the Euclidean unit ball.
\end{proposition}
\begin{proof} Since \eqref{elipsedistro} is vacuously true when $P$ is identically zero, if \eqref{elipsedistro} failed to hold for any $c > 0$, then, after renormalizing, there would be a sequence of polynomials $\{P_{N}\}_{N=1}^\infty$ such that $\sup_{t \in \mathbb{B}_{\de}} |P_N(t')| = 1$ for each $N$ but have $|P_N(t')| \leq 1/N$ for all $t' \in \mathbb{B}_{\de}$ in a set of measure at least $\epsilon |\mathbb{B}_{\de}|$ for each $N$. Because the space of homogeneous polynomials of degree $s$ is finite-dimensional, it is always possible to pass to a subsequence which converges uniformly on compact sets to some limit polynomial $P_*$. The limit polynomial cannot be identically zero because it must attain the value $1$ on the unit ball. Moreover, for any $\delta > 0$, there must be some $N$ for which $|P_N(t') - P_*(t')| \leq \delta/2$ for all $t' \in {\mathbb B}_{\de}$ and some set of measure at least $\epsilon |\mathbb{B}_{\de}|$ on which $|P_N| \leq \delta/2$, which means that the sublevel set $\set{t \in \mathbb{B}_{\de}}{ |P_*(t)| \leq \delta}$ must have measure at least $\epsilon |\mathbb{B}_{\de}|$ for each $\delta > 0$. By the Lebesgue Dominated Convergence Theorem applied to the indicator functions $\chi_{|P_*(t)| \leq \delta}$, sending $\delta$ to $0$ implies that the set of points $t \in {\mathbb B}_{\de}$ at which $P_*(t) = 0$ must also have Lebesgue measure at least $\epsilon |\mathbb{B}_{\de}|$, which is impossible because $P_*$ is not identically zero.
\end{proof}
By a linear change of variables in \eqref{elipsedistro}, the analogous result holds when $\mathbb{B}_{\de}$ is replaced with any centered ellipsoid $E \subset \R^{\de}$. Now each expression
\[ \det ( \tup{v_{i'}}{k-s},\tup{\Theta(t,w_{i''})}{s}) \]
is a homogeneous polynomial of degree $s$ in the $t$ variables, so for any fixed $\epsilon \in (0,1)$, it may be assumed that there is some $c > 0$ for which 
\[ |\det ( \tup{v_{i'}}{k-s},\tup{\Theta(t,w_{i''})}{s})|^2 \geq c \sup_{t' \in E} |\det ( \tup{v_{i'}}{k-s},\tup{\Theta(t',w_{i''})}{s})|^2 \]
for all $t \in E$ aside from some exceptional set of measure at most $\epsilon |E|$. By fixing $\epsilon$ sufficiently small, it follows that on any centered ellipsoid $E$, there is a subset $E'$ of measure at least $|E|/2$ such that
\[ \begin{split}
\sum_{s=0}^{\min\{d,k\}} \frac{1}{s!(k-s)!} \sum_{i'=1}^{k} &  \sum_{i''=1}^d \inf_{t \in E'} \left| \det ( \tup{v_{i'}}{k-s},\tup{\Theta(t,w_{i''})}{s}) \right|^2 \\
 & \geq c \sum_{s=0}^{\min\{d,k\}} \sum_{i'=1}^{k} \sum_{i''=1}^d \sup_{t \in E} \left| \det ( \tup{v_{i'}}{k-s},\tup{\Theta(t,w_{i''})}{s}) \right|^2 \end{split} \]
 for some constant $c$ that depends only on $\de, d$, and $k$.
 As a consequence of this observation combined with \eqref{inversionthing} and \eqref{qcomp}, when $E := B_u$ for some basis $\{u_i\}_{i=1}^{\de}$ of $\R^{\en}$, there always exists a set $B_u' \subset B_u$ with $|B_u'| \geq |B_u|/2$ on which
 \[ \inf_{t \in B_u'} |L_t B_\omega| \geq c_{\de,d,k} |\det\{v_i\}_{i=1}^{k}| \mathcal{Q} [ \{u_i\}_{i=1}^{\de}, \{v_i^*\}_{i=1}^k,\{w_i\}_{i=1}^d], \]
where $\mathcal{Q}$ is exactly \eqref{qsumdef} associated to $\phi$ at the origin with the kernel basis being simply the standard basis in the $x_0$ variables. (Note that we may assume by inner regularity of Lebesgue measure that $B'_u$ is compact.) An upper bound for $|L_t B_\omega|$ in terms of this same expression is trivial thanks to \eqref{qcomp}, and so one has
\[ |L_t B_\omega| \approx |\det\{v_i\}_{i=1}^{k}| \mathcal{Q} [ \{u_i\}_{i=1}^{\de}, \{v_i^*\}_{i=1}^k,\{w_i\}_{i=1}^d] \]
for all $t \in B_u'$, with implicit constants depending only on $\de,d,$ and $k$. 

Under the assumption that $Q$ is degenerate, it follows by \eqref{hilbertmumford} that there exist orthonormal choices of $\{u_i\}_{i=1}^{\de}$, $\{v_i^*\}_{i=1}^k$, and $\{w_i\}_{i=1}^d$ along with matrices $D_1 ,D_2,D_3$ diagonal in these bases such that
\begin{equation} \frac{\mathcal{Q} [\{e^{\tau D_1 } u_i\}_{i=1}^{\de}, \{e^{\tau D_2} v_i^*\}_{i=1}^k,\{e^{\tau D_3} w_i\}_{i=1}^d]}{|\det \{e^{\tau D_1 } u_i\}_{i=1}^{\de }|^{\frac{dk}{n \de }}} = O(e^{- \epsilon \tau}) \label{hm3} \end{equation}
for some $\epsilon > 0$ as $\tau \rightarrow \infty$. It is further known that $|\det \{e^{\tau D_2} v_i^*\}_{i=1}^k| = 1$ and $|\det \{e^{\tau D_3} w_i\}_{i=1}^d| = 1$ for all $\tau$.

Proceed as follows: for each $\tau > 0$, let $\omega^{(\tau)}$ be the basis of $\R^n$ comprised of vectors of the form $(0, e^{-\tau D_2} v_i)$ for $i=1,\ldots,k$ and $(e^{\tau D_3} w_{i'},0)$ for $i' =1,\ldots,d$ (the negative sign in $e^{-\tau D_2}$ is appropriate because these are dual to the basis $\{e^{\tau D_2} v_i^*\}_{i=1}^k$). Let $\delta :=  e^{-\tau}$ and fix $F_\delta := B_{\omega^{(\tau)}}$ and let $G_\delta$ be the set
\[ \set{ (t, y) \in B_{e^{\tau D_1} u}' \times \R^k}{ y = x_1 + \Theta(t,x_0) \text{ for some } (x_0,x_1) \in F_\delta}. \]
It follows that
\[ \int \chi_{F_\delta}(x) \chi_{G_\delta}(\tilde \gamma_t(x)) dx dt = |F_\delta| |B_{e^{\tau D_1} u}'| \approx |\det e^{\tau D_1}| \]
because $\chi_{F_\delta}(x) \chi_{G_\delta}(\tilde \gamma_t(x))$ equals one for every $(x_0,x_1) \in F_\delta$ and every $t \in B'_{e^{\tau D_1} u}$. Also used here are the facts that $|F_\delta|$ is constant as a function of $\delta$ and that $|B'_{e^{\tau D_1} u}|$ is comparable to $|B_{e^{\tau D_1} u}|$. The implicit constant depends only on $\de,d,$ and $k$. On the other hand, 
\[ |G_\delta| \approx |\det e^{\tau D_1}| \mathcal{Q} [\{e^{\tau D_1 } u_i\}_{i=1}^{\de}, \{e^{\tau D_2} v_i^*\}_{i=1}^k,\{e^{\tau D_3} w_i\}_{i=1}^d] \]
because slices $G_\delta(t)$ are exactly equal to $L_t B_{\omega^{(\tau)}}$ when $t \in B'_{e^{\tau D_1} u}$ and are empty otherwise (and because $|\det \{ e^{-\tau D_2}v_i \}_{i=1}^k| = 1$). Once again,
the implicit constants depend only on the basic parameters $\de,d,$ and $k$. These observations combine with \eqref{hm3} to give that
\begin{equation}
\begin{split}
& \frac{\int \chi_{F_\delta}(x) \chi_{G_\delta}(\tilde \gamma_t(x)) dx dt}{|F_\delta|^{1/q_b'} |G_\delta|^{1/p_b}} \\ & \approx
\left[ \frac{|\det e^{\tau D_1}|^{p_b-1}}{\mathcal{Q}[\{e^{\tau D_1 } u_i\}_{i=1}^{\de}, \{e^{\tau D_2} v_i^*\}_{i=1}^k,\{e^{\tau D_3} w_i\}_{i=1}^d]} \right]^{1/p_b} \geq c e^{\tau \epsilon / p_b} = c \delta^{-\epsilon/p_b}
\end{split} \label{itsstb}
\end{equation}
for all sufficiently small $\delta$.  The compact sets $F_{\delta}$ and $G_{\delta}$ therefore satisfy all the required properties of Proposition \ref{convertknapp} because the lengths of vectors in the bases $\{e^{\tau D_1 } u_i\}_{i=1}^{\de}, \{e^{\tau D_2} v_i^*\}_{i=1}^k$, and $\{e^{\tau D_3} w_i\}_{i=1}^d$ grow at most exponentially in $\tau$ (and therefore like a power of $\delta^{-1}$), the inequality \eqref{scalestoobig} holds (by virtue of \eqref{itsstb}), the sets $G_{\delta}$ have slices which are ellipsoids whose measure is no smaller than exponential decay in $\tau$, and the measures of $G_\delta$ and $F_\delta$ are likewise no smaller than exponentially decaying in $\tau$. Thus the operator \eqref{theop} must be unbounded from $L^p$ to $L^q$ for all pairs $(1/p,1/q)$ near $(1/p_b,1/q_b)$. As the model operator based on $\tilde \gamma_t(x)$ depends only on the quantity $Q$ defined for \eqref{theop} at the origin, the neighborhood of pairs near $(1/p_b,1/q_b)$ for which $T$ is unbounded may be assumed to depend only on $Q$ at the origin.

\section{Nonpolynomial averages}
\label{npasec}

This final section deals with the passage from polynomial $\phi$ to the $C^\infty$ category. Section \ref{genapprox} establishes the abstract tools necessary for the transition, designed to be of continued use for the future study of degenerate objects. Section \ref{specapprox} then applies these tools to the main case of interest, namely, nondegeneracy of the trilinear form $Q$ from \eqref{qdef}.

\subsection{General approximation results}
\label{genapprox}

As has been the case throughout, suppose $\phi(x,t)$ is a smooth $\R^k$-valued function of its parameters $(x,t)$ for all such points belonging to some open set $U \subset \R^n \times \R^{\de}$. Assume that the Jacobian matrix $D_x \phi$ is full rank at $(x_*,t_*)$ and let $\gamma_t(x) := (t,\phi(x,t)) \in \R^{\en}$ for each $(x,t) \in U$. The first proposition below establishes local $L^p$-boundedness of \eqref{theop}.
\begin{proposition}
If $\eta$ is a continuous function on $\R^n \times \R^{\en}$ such that $\eta(x,\gamma_{t}(x))$ is compactly supported sufficiently close to the point $(x_*,t_*)$, then the operator \eqref{theop} maps $L^p(\R^{\en})$ to $L^p(\R^n)$ for all $1 \leq p \leq \infty$. \label{lpbound}
\end{proposition}
\begin{proof}
By interpolation, it suffices to consider the cases $p=1$ and $p=\infty$.  Consider first $p = 1$. Without loss of generality, one may assume that the map $(t,x^1,\ldots,x^k) \mapsto (t,\phi(x,t))$ is locally injective with nondegenerate Jacobian for each fixed $x^{k+1},\ldots,x^n$ provided that the points $x$ and $t$ are sufficiently close to $x_*$ and $t_*$, respectively. Letting $K \in \R^d$ be a compact set such that $\eta(x,\gamma_t(x)) = 0$ when $(x^{k+1},\ldots,x^n) \not \in K$, it follows that
\begin{align*}
\int & \left| \int g(\gamma_t(x)) \eta(x,\gamma_t(x)) dt \right| dx \leq \int |g(t,\phi(x,t))| |\eta(x,\gamma_t(x)) | dt dx \\
& \leq \int_K \left( \int |g(t,\phi(x,t))| |\eta(x,\gamma_t(x)) | dt dx^1 \cdots dx^{k} \right) dx^{k+1} \cdots dx^{n} \\
& \leq C |K| ||g||_{L^1} ||\eta||_\infty 
\end{align*}
with $C$ being the supremum of $1/|\det \partial \phi(x,t) / \partial (x^1,\ldots,x^k)|$ over those $(x,t)$ for which $\eta(x,\gamma_t(x)) \neq 0$. Making the support of $\eta(x,\gamma_t(x))$ sufficiently close to $(x_*,t_*)$ ensures that $C$ is finite.

When $p = \infty$,  $|g(y)| \leq ||g||_\infty$ for all $y$ not belonging to some null set; by inequality just proved for $p=1$, almost every $x$ has the property that $\gamma_t(x)$ belongs to this null set only for a null set of parameters $t$. So redefining $g$ to be zero on this null set preserves the value of $T g(x)$ at almost every point $x$, and consequently it suffices to assume that $|g(y)| \leq ||g||_\infty$ everywhere. Then
\[ |T g(x)| \leq ||g||_\infty \int |\eta(x,\gamma_t(x))| dt, \]
and the integral on the right-hand side will be uniformly bounded for all $x$ because $\eta(x,\gamma_t(x))$ is bounded and can only be nonzero for some compact set of $t \in \R^{\de}$.
\end{proof}

The next proposition establishes a very rudimentary sort of $L^2$-Sobolev inequality for operators $T$ satisfying the H\"{o}rmander condition. The estimates provided here will be just enough to give crucial decay when summing terms of the Littlewood-Paley decomposition of $T$.
\begin{proposition} 
Suppose the double fibration $\pi_1(x,t) := \gamma_t(x)$ and $\pi_2(x,t) := x$ of the open set $U \subset \R^n \times \R^{\de}$ satisfies the H\"{o}rmander condition at the point $(x_*,t_*)$. Then for all smooth $\eta$ with $\eta(x,\gamma_t(x))$ supported sufficiently near $(x_*,t_*)$ and for all $N$ sufficiently large, the operator \eqref{theop} has the property that $(T^* T)^N$ is given by integration against a kernel $K_N(y,y')$ such that
\begin{equation}
\int \left| K_N(y,y'+h) - K_N(y,y') \right| dy' \leq C |h|^\delta \label{l1delta}.
\end{equation}
uniformly in $y,h \in \R^{\en}$ for some positive $\delta$. As a consequence, if $\psi$ is any continuous function of compact support on $\R^{\en}$ whose integral vanishes and $R$ is convolution with $\psi$ (i.e., $Rf := f * \psi$), then for any $N$ sufficiently large that \eqref{l1delta} holds,
\begin{equation}
|| T R f||_2 \leq ||\psi||_1^{1-\frac{1}{2N}} \left[ C \int |h|^\delta |\psi(h)| dh \right]^\frac{1}{2N} ||f||_2
\label{cancell1}
\end{equation}
for all $f \in L^2(\R^{\en})$, where $C$ and $\delta$ are the same as in \eqref{l1delta}. \label{cancelprop}
\end{proposition}
\begin{proof}
If $n = \en$, the existence of $K_N$ satisfying \eqref{l1delta} is a direct consequence (after suitable changes of variables) of Lemma 20.1 and Proposition 7.2 in the seminal work of Christ, Nagel, Stein, and Wainger \cite{cnsw1999}, so the only cases in question arise when $n \neq \en$. Thankfully, these cases are rather trivial consequences of the equal dimension case and can be proved by simply tacking on extra directions onto whichever side is ``deficient'' using trivial behavior in these new directions. 

When $n < \en$, the argument goes as follows. Consider the double fibration defined on an open subset of the point $(x_*,0,t_*) \in \R^{n} \times \R^{\en-n} \times \R^{\de}$ which is given by
\[ \overline{\pi}_1(x,s,t) := \gamma_t(x) \text{ and } \overline{\pi}_2(x,s,t) := (x,s). \]
The H\"{o}rmander condition must be satisfied at $(x_*,0,t_*)$ for these extended mappings simply because all vector fields annihilated by $d \pi_1$ and $d \pi_2$ are also annihilated by $d \overline{\pi}_1$ and $d \overline{\pi}_2$, respectively, when the former are transported to the new space by taking them to be constant in the new parameter $s$. By assumption, taking successive Lie brackets of these original vector fields will eventually yield a family of smooth vector fields which span all of the original tangent directions. The remaining directions (namely, the new directions $\partial/\partial s^i$) belong to the kernel of $d \overline{\pi}_1$, so every direction is spanned. It follows that when $\overline{T}$ is the operator defined by duality in terms of the formula
\[ \int  f(x,s) \overline{T} g(x,s)  dx ds = \int f(\overline{\pi}_2(x,s,t)) g(\overline{\pi}_1(x,s,t)) \eta(x,\gamma_t(x)) \varphi(s) dx dt ds \]
for smooth $\eta$ such that $\eta(x,\gamma_t(x))$ is supported near $(x_*,t_*)$ and smooth $\varphi$ such that $\varphi$ is supported sufficiently near zero, all sufficiently large $N$ have that $(\overline{T}^* \overline{T})^N$ is given by integration against a kernel which is uniformly in $L^1_\delta$ as in \eqref{l1delta}. 

Let $\varphi$ be a nonnegative smooth function of compact support sufficiently close to the origin in $\R^{\en-n}$ having the property that $\int |\varphi(s)|^2 ds = 1$. By the duality formula, $\overline{T} g(x,s) = \varphi(s) T g(x),$
where $T$ is the operator \eqref{theop} with cutoff function $\eta$. Computing the dual of both sides gives
\[ \overline{T}^* \overline{f}(y) = T^* \left( \int \overline{f}(\cdot,s) \varphi(s) ds \right) \]
for any suitable $\overline{f}$ defined on $\R^{n} \times \R^{\en-n}$,
by which it follows that $\overline{T}^* \overline{T} = T^* T$. Thus the kernel of $(T^*T)^N$ must be uniformly in $L^1_\delta$ for large $N$ because the same holds for $(\overline{T}^* \overline{T})^N$.  

The case $\en < n$ is similar but slightly more involved. In this case let $\overline{\pi}_1(x,s,t) := (\gamma_t(x),s)$ and $\overline{\pi}_2(x) := x$ and define $\overline{T}$ so that
\[ \int  f(x) \overline{T} g(x)  dx = \int f(\overline{\pi}_2(x,s,t)) g(\overline{\pi}_1(x,s,t)) \eta(x,\gamma_t(x)) \varphi(s) dx dt ds. \]
Just as before, the H\"{o}rmander property still holds in this extended case, meaning that the kernel of $(\overline{T}^* \overline{T})^N$ is uniformly in $L^1_\delta$ provided that $\eta$ and $\varphi$ are suitably localized. Similar to the situation above,
\[ \overline{T} \overline{g}(x) = T \left( \int \overline{g}(\cdot,s) \varphi(s) ds \right) \]
for any function $\overline{g}$ on $\R^{\en} \times \R^{n - \en}$.
Let $P \overline{g} (y,s) := \int \overline{g}(y,s') \varphi(s') ds' \varphi(s)$. This projection $P$ is self-adjoint and satisfies $\overline{T} P = \overline{T}$. It follows that $P (\overline{T}^* \overline{T}) P = \overline{T}^* \overline{T}$ and $P(\overline{T}^* \overline{T})^N P =  (\overline{T}^* \overline{T})^N$ for any $N \geq 1$, so the kernel $\overline{K}_N$ of $(\overline{T}^* \overline{T})^N$ must have the property that
\[ \begin{split} \overline{K}_N((y,s),(y',s')) & = \varphi(s) \varphi(s') \int \overline{K}_N((y,s''),(y',s'''')) \varphi(s'') \varphi(s''') ds'' ds''' \\ & =: K_N(y,y') \varphi(s) \varphi(s') \end{split}\]
for some $K_N$ which is independent of $s$ and $s'$.
The quantity $K_N$ must in fact be the kernel of $(T^* T)^N$ because $\overline{T}^* \overline{T} ( g(\cdot) \varphi(\cdot)) = \varphi T^* T g$, where $g (\cdot) \varphi(\cdot)$ indicates the function $g(y) \varphi(s)$ on $\R^{\en} \times \R^{n - \en}$. By induction on $N$,
$(\overline{T}^* \overline{T})^N ( g(\cdot) \varphi(\cdot)) = \varphi(\cdot) (T^* T)^N g$ and therefore
\[ \begin{split}
\varphi(s) (T^*T)^N g (y) & = \overline{T}^* \overline{T} ( g(\cdot) \varphi(\cdot)) (y) \\ & = \varphi(s) \int K_N(y,y') \varphi(s') g(y') \varphi(s') ds' dy' \\
& = \varphi(s) \int K_N(y,y') g(y') dy, \end{split}\]
which clearly forces $K_N$ to be the kernel of $(T^*T)^N$.
Because $\overline{K}_N((y,s),\cdot)$ is uniformly in $L^1_\delta$ for all pairs $(y,s)$, $K_N(y,y')$ must also be uniformly in $L^1_\delta$ for all $y$:
\begin{align*}
\int & |K_N(y,y'+h) - K_N(y,y')| dy' \\ & = \int |K_N(y,y'+h) - K_N(y,y')| |\varphi(s)|^2 |\varphi(s')|^2 dy' ds ds' \\
& = \int  |\overline{K}_N((y,s),(y'+h,s')) - \overline{K}_N((y,s),(y'+h,s')| |\varphi(s)| |\varphi(s')| dy' ds ds' \\
& \leq  ||\varphi||_\infty \int  |\overline{K}_N((y,s),(y'+h,s')) - \overline{K}_N((y,s),(y'+h,s')| |\varphi(s)| dy' ds ds' \\
& \leq ||\varphi||_1 ||\varphi||_\infty \sup_{s} \int |\overline{K}_N((y,s),(y'+h,s')) - \overline{K}_N((y,s),(y'+h,s')| dy' ds' \\ & \leq C |h|^\delta||\varphi||_1 ||\varphi||_\infty,
\end{align*}
This completes the proof of \eqref{l1delta}.

Assuming now that \eqref{l1delta} holds, Let $\psi$ be any continuous function on $\R^{\en}$ which is compactly supported and has integral zero, and let $R g := \psi * g$. The desired inequality \eqref{cancell1} is a consequence of an elementary argument combined with the observation that
\begin{equation} || R^* T^* T R ||_{2 \rightarrow 2}  \leq ||R^* R||_{2 \rightarrow 2}^{1 - \frac{1}{N}} ||R^* (T^* T)^{N} R||_{2 \rightarrow 2}^{\frac{1}{N}}. \label{spectral} \end{equation}
for all $N \geq 1$. The proof of this inequality is based on the simpler observation that for each integer $N \geq 0$, 
\begin{equation} ||R^* (T^* T)^{N + 1} R ||_{2 \rightarrow 2}  \leq \left( ||R^* (T^* T)^{N} R||_{2 \rightarrow 2}  ||R^* (T^* T)^{N+2} R||_{2 \rightarrow 2} \right)^{\frac{1}{2}}. \label{nmconv} \end{equation} When $N$ is even, \eqref{nmconv} holds because
\begin{align*}
||R^* (T^* T)^{N + 1} R ||_{2 \rightarrow 2} & \leq ||R^* (T^*T)^{\frac{N}{2}}||_{2 \rightarrow 2} ||(T^* T)^{\frac{N}{2}+1}R ||_{2 \rightarrow 2} \\ & =  ||(T^*T)^{\frac{N}{2}}R ||_{2 \rightarrow 2} ||(T^* T)^{\frac{N}{2}+1}R ||_{2 \rightarrow 2} \\
& = \left( ||R^* (T^* T)^{N} R||_{2 \rightarrow 2} ||R^* (T^* T)^{N+2} R||_{2 \rightarrow 2} \right)^{\frac{1}{2}},
\end{align*}
and when $N$ is odd, one has instead has \eqref{nmconv} because
\begin{align*}
||R^* (T^* T)^{N + 1} R ||_{2 \rightarrow 2} & \leq ||R^* (T^*T)^{\frac{N-1}{2}}T^* ||_{2 \rightarrow 2} ||T (T^* T)^{\frac{N-1}{2}+1}R ||_{2 \rightarrow 2} \\ & =  ||T (T^*T)^{\frac{N-1}{2}}R ||_{2 \rightarrow 2} ||T (T^* T)^{\frac{N-1}{2}+1}R ||_{2 \rightarrow 2} \\
& = \left( ||R^* (T^* T)^{N} R||_{2 \rightarrow 2} ||R^* (T^* T)^{N+2} R||_{2 \rightarrow 2} \right)^{\frac{1}{2}}.
\end{align*}
Because \eqref{spectral} will be trivially true if either $||R^* T^* T R||_{2 \rightarrow 2} = 0$ or if $||R^* R||_{2 \rightarrow 2} = 0$ (meaning that $R = 0$), without loss of generality, it may be assumed that neither of these norms vanish. Then \eqref{nmconv} implies that
 $||R^* (T^* T)^{N+2} R||_{2 \rightarrow 2} \neq 0$ when both $||R^* (T^* T)^{N} R||_{2 \rightarrow 2}$ and $||R^* (T^* T)^{N+1} R||_{2 \rightarrow 2}$ are nonzero. By induction, then, it may be assumed that $||R^* (T^* T)^{N} R||_{2 \rightarrow 2}$ is nonzero for each nonnegative integer $N$. In this case, the inequality \eqref{nmconv} implies that
\[ \frac{||R^* (T^* T)^{N+1} R||_{2 \rightarrow 2} }{||R^* (T^* T)^{N} R||_{2 \rightarrow 2}} \leq \frac{||R^* (T^* T)^{N+2} R||_{2 \rightarrow 2}}{||R^* (T^* T)^{N+1} R||_{2 \rightarrow 2}} \]
for each $N \geq 0$, which implies by induction on $N$ that
\[  \frac{||R^* T^* T R||_{2 \rightarrow 2} }{||R^* R||_{2 \rightarrow 2}} \leq \frac{||R^* (T^* T)^{N+1} R||_{2 \rightarrow 2}}{||R^* (T^* T)^{N} R||_{2 \rightarrow 2}} \]
for all $N \geq 0$, which then implies that 
\[ \left( \frac{||R^* T^* T R||_{2 \rightarrow 2} }{||R^* R||_{2 \rightarrow 2}} \right)^N \leq \frac{||R^* (T^* T)^{N} R||_{2 \rightarrow 2}}{||R^* R||_{2 \rightarrow 2}} \]
for all $N \geq 1$. Rearranging terms gives \eqref{spectral}. By the usual Hilbert space theory of $L^2$, it follows that
\begin{equation} ||TR||_{2 \rightarrow 2} \leq ||R||_{2 \rightarrow 2}^{1 - \frac{1}{N}} ||R^* (T^* T)^N R||_{2 \rightarrow 2}^{\frac{1}{2N}} \label{spectralroot} \end{equation}
for each $N \geq 1$.
Now for $N$ sufficiently large, the kernel smoothness condition \eqref{l1delta} implies that
\begin{align*}
 (T^* T)^N R g (y) & = \int K_N(y,y') \psi(y'-y'') g(y'') dy' dy'' \\
 & = \int K_N(y,y''+h) \psi(h) g(y'') dh dy'' \\
 & = \int \left[ K_N(y,y''+h) - K_N(y,y'') \right] \psi(h) g(y'') dh dy'',
 \end{align*}
 where the term $-K_N(y,y'')$ can be added without changing the integral because the integral of $\psi(h)$ over $h$ is zero. Thus when $K_N$ satisfies \eqref{l1delta}, it follows that
 \[ \begin{split} ||R^*(T^* T)^N R g||_\infty & \leq C ||g||_\infty ||R^*||_{\infty \rightarrow \infty} \int |h|^\delta |\psi(h)| dh \\ & \leq C ||g||_\infty \int |\psi(h)| dh  \int |h|^\delta |\psi(h)| dh. \end{split} \]
Because $R^* (T^* T)^N R$ is self-adjoint, duality and interpolation combine to give that $||R^* (T^* T)^N R||_{2 \rightarrow 2} \leq C ||\psi||_1 ||\, |\cdot|^\delta \psi||_1$. This inequality combined with \eqref{spectralroot} and the trivial inequality $||R||_{2 \rightarrow 2} \leq ||\psi||_1$ give \eqref{cancell1}.
\end{proof}

The main result in this section is Theorem \ref{approxthm} below. Informally, the theorem guarantees that if one can approximate the mapping $\gamma_t(x)$ within distance $2^{-j}$ by polynomial mappings $\gamma^j_t(x)$ for all positive integers $j$ in such a way that the polynomial Radon-like transforms are bounded from $L^{p_0}$ to $L^{q_0}$ with only slow growth of the norm, then the original $T$ must be bounded for pairs $(1/p,1/q)$ arbitrarily close to $(1/p_0,1/q_0)$.
\begin{theorem}
Suppose $\gamma_t(x)$ is as described above and that both of the following hypotheses hold: \label{approxthm}
\begin{enumerate}
\item The double fibration $\pi_1(x,t) := \gamma_t(x)$ and $\pi_2(x,t) := x$ satisfies the H\"{o}rmander condition at the point $(x_*,t_*)$.
\item There exists a nonnegative continuous function $\eta_0$ with $\eta_0(x,\gamma_t(x))$ compactly supported in the domain $U$ of $\phi$ which is nonvanishing at $(x_*,t_*)$ and mappings $\gamma_t^j(x) := (t,\phi^j(x,t))$ for each $j \geq 1$ defined on an open set $U_j$ containing the support of $\eta_0(x,\gamma_t(x))$ such that $|\gamma^j_t(x) - \gamma_t(x)| \leq 2^{-j}$ on $U_j$. Suppose also that the operators
\begin{equation} T^j f(x) := \int f(\gamma^j_t(x)) \eta_0(x,\gamma_t(x)) dt \label{tjdef} \end{equation}
have the property that for every $\epsilon > 0$, there is a constant $C_\epsilon$ such that
\begin{equation} || T^j f||_{L^{q_0}(\R^n)} \leq C_\epsilon 2^{\epsilon j} ||f||_{L^{p_0}(\R^{\en})} \label{epsbound} \end{equation}
for all continuous functions of compact support $f$ all $j$ sufficiently large (with threshold independent of $f$),  where $1 \leq p_0 < q_0 \leq \infty$.
\end{enumerate}
Then for all continuous $\eta$ such that $\eta(x,\gamma_t(x))$ is supported sufficiently close to $(x_*,t_*)$, the operator \eqref{theop}
maps $L^p$ to $L^q$ for all pairs $(1/p,1/q)$ in the interior of the triangle with vertices $(0,0)$, $(1,1)$, and $(1/p_0,1/q_0)$.
\end{theorem}

The proof of Theorem \ref{approxthm} requires only a most rudimentary sort of Littlewood-Paley decomposition.  Let $\varphi$ be a radial Schwartz function on $\R^{\en}$ which is nonnegative, nonincreasing as a function of the radius, has integral $1$, and vanishes outside the ball of radius $1/2$.  Let $\psi(x) :=  \varphi(x) - 2^{-\en} \varphi(x/2)$, and for each $j \geq 1$, let $\psi_j(x) = 2^{\en j} \psi(2^j x)$. Each $\psi_j$ is a $C^\infty$ function of mean zero on $\R^{\en}$ which is supported on the ball of radius $2^{-j}$. It is easy to verify by classical means that 
\begin{equation} f = f * \varphi + \sum_{j=1}^\infty f * \psi_j \label{convoeq} \end{equation}
for any $f \in L^p(\R^{\en})$ with convergence pointwise almost everywhere and in the $L^p$ norm provided $1 < p < \infty$. For convenience, let $R_j f := f * \psi_j$ for each $j \geq 1$ and let $R_0 f := f * \varphi$.

The proof of Theorem \ref{approxthm} also relies on the following two very basic propositions.
\begin{proposition}
Let $\eta$ be a nonnegative continuous function on $\R^n \times \R^{\de}$ with $\eta(x,\gamma_t(x))$ compactly supported within the domain of $\phi$ and let $T$ and $\tilde{T}$ be Radon-like transforms defined by
\begin{equation} \begin{split} T f(x) & := \int_{\R^{\de}} f(\gamma_t(x)) \eta(x,\gamma_t(x)) dt,\\ \tilde{T} f(x) & := \int_{\R^{\de}} f(\tilde{\gamma_t}(x))  \eta(x,\gamma_t(x)) dt, \end{split} \label{talldef} \end{equation}
for all continuous $f$ of compact support on $\R^{\en}$, where
 $\gamma_t$ and $\tilde{\gamma_t}(x) := (t, \tilde \phi(x,t))$ are defined on the support on open sets large enough to contain the support of $\eta(x,\gamma_t(x))$ (so that the integrals \eqref{talldef} are well-defined) and take values in $\R^{\en}$ such that
\[ |\gamma_t(x) - \tilde{\gamma_t}(x)| \leq 2^{-j} \]
for all $(x,t)$ in the support of $\eta(x,\gamma_t(x))$. Then for any $p,q \in [1,\infty]$,
\begin{equation} || T R_j||_{p \rightarrow q} \leq C_{\en} ||\tilde{T}||_{p \rightarrow q}. \label{pqpayoff} \end{equation}
\end{proposition}
\begin{proof}
As a side note before proving \eqref{pqpayoff}, observe that when $\tilde \eta(x,(t,y')) := \eta(x,\gamma_t(x))$, it follows that $\eta(x,\gamma_t(x)) = \tilde \eta(x,\tilde \gamma_t(x))$. Thus the form of the cutoff function inside $\tilde T$ is compatible with all previous assumptions  even though it is written as a function of $\gamma_t(x)$ rather than $\tilde \gamma_t(x)$. 

For any continuous $f$ of compact support, $T R_j f (x) - \tilde{T}R_j f(x)$ equals
\[ \int \left[ \int \left[ \psi_j(\gamma_t(x) - y) - \psi_j(\tilde{\gamma_t}(x) - y) \right] f(y) dy \right] \eta(x,\gamma_t(x)) dt. \]
By the Mean Value Theorem,
\[\left| \psi_j(\gamma_t(x) - y) - \psi_j(\tilde{\gamma_t}(x) - y) \right| \leq |\gamma_t(x) - \tilde{\gamma_t}(x)| \sup_{y'} | \nabla \psi_j(y')| \leq C_{\en} 2^{\en j} \]
for all $(x,t)$.
The difference $\psi_j(\gamma_t(x) - y) - \psi_j(\tilde{\gamma_t}(x) - y)$ must also vanish when $|\tilde{\gamma_t}(x)-y| \geq 2^{-j+1}$ because $|\gamma_t(x) - y| \geq |\tilde \gamma_t(y) - y| - |\tilde \gamma_t(y) - \gamma_t(y)| \geq 2^{-j+1} - 2^{-j} \geq 2^{-j}$ and consequently both  $\psi_j(\gamma_t(x) - y)$ and $\psi_j(\tilde{\gamma_t}(x) - y)$ are already zero. Letting $\Psi_j$ be the function which is equal to $C_{\en} 2^{\en j}$ on the ball of radius $2^{-j+1}$ and zero elsewhere, it follows that
\[ |T (f * \psi_j)(x) - \tilde{T}( f * \psi_j)(x)| \leq \tilde{T} (|f| * \Psi_j)(x) \]
for all $x$. Thus
\[ || T R_j f||_q \leq || \tilde{T} (f * \psi_j)||_q + ||\tilde{T} (|f| * \Psi_j)||_q \leq C_{\en} ||\tilde T||_{p \rightarrow q} ||f||_p \]
because convolution with $\psi_j$ and $\Psi_j$ are uniformly bounded on $L^p$ with a constant independent of $p$ and $j$ and map the continuous functions of compact support to themselves. When $j = 0$, this proof must be modified by replacing $\psi_j$ with $\varphi$, but as no cancellation properties of $\psi_j$ were just used, the same argument implies \eqref{pqpayoff} in this case as well. 
\end{proof}

\begin{proposition}
Under the same hypotheses as Proposition \ref{cancelprop}, if $\eta$ is smooth and $\eta(x,\gamma_t(x))$ is supported sufficiently near $(x_*,t_*)$, then there exists \label{sobprop}
exists $\epsilon > 0$ such that 
\begin{equation} || T R_j f||_2 \leq C 2^{-\epsilon j} ||f||_2 \label{sobolevpayoff} \end{equation}
for some constant $C$ which is independent of $f \in L^2(\R^{n})$ and $j \geq 0$.
\end{proposition}
\begin{proof}
This is an immediate consequence of \eqref{cancell1} applied to $R_j$, because
\[ \int |h|^\delta |\psi_j(h)| dh = \int |2^{-j} h|^\delta |\psi(h)| dh = C 2^{-j \delta} \]
for some constant $C$ independent of $f$ and $j$. It follows that $\epsilon = \delta / 2N$ for any sufficiently large $N$ so that \eqref{cancell1} holds. Although \eqref{cancell1} does not apply in the case $j = 0$, the fact that both $T$ and $R_0$ are bounded on $L^2$ gives \eqref{sobolevpayoff} instead.
\end{proof}

\begin{proof}[Proof of Theorem \ref{approxthm}]
Let $\eta_0$ be a continuous nonnegative cutoff function as hypothesized in Theorem \ref{approxthm} which yields Radon-like transforms 
\[ T^j f(x) := \int f(\gamma_t^j(x)) \eta_0(x,\gamma_t(x)) dt \]
satisfying \eqref{epsbound}. Let $\eta_1$ be any smooth nonnegative cutoff function satisfying the hypotheses of Proposition \ref{sobprop} and nonvanishing at $(x_*,\gamma_{t_*}(x_*))$ so that \eqref{sobolevpayoff} holds for the unperturbed Radon-like transform $T$ given by
\[  T f(x) := \int f(\gamma_t(x)) \eta_1(x,\gamma_t(x)) dt. \] 
Multiplying $\eta_1$ by a small positive constant and a smooth bump function if necessary, it is always possible to assume that $\eta_1(x,\gamma_t(x)) \leq \eta_0(x,\gamma_t(x))$ for all $x$ and $t$.  Let
\[ \tilde T^j f(x) := \int f(\gamma_t^j(x)) \eta_1(x,\gamma_t(x)) dt. \]
By \eqref{epsbound} and the inequality
\[ \left| \int f(\gamma_t^j(x)) \eta_1(x,\gamma_t(x)) dt \right| \leq \int |f(\gamma_t^j(x))| \eta_0(x,\gamma_t(x)) dt, \]
 it must be the case for each $\epsilon_1 > 0$ that $||\tilde T^j ||_{p_0 \rightarrow q_0} \leq C_{\epsilon_1 } 2^{\epsilon_1  j}$ for all $j$ sufficiently large. Then \eqref{pqpayoff} and \eqref{epsbound} imply that
\[ ||T R_j||_{p_0 \rightarrow q_0} \leq C'_{\epsilon_1 } 2^{\epsilon_1  j} \]
for some constant $C'_{\epsilon_1}$ which is independent of $j$ (we can assume that this inequality holds for all $j \geq 0$ because, when $j$ is below the finite threshold at which we have bounds for $||\tilde T^j||_{p_0 \rightarrow q_0}$, it is still the case that $R_j$ maps $L^{p_0}$ to $L^{q_0}$ and that $T$ is bounded on $L^{q_0}$). Interpolation with \eqref{sobolevpayoff} implies that
\[ ||T R_j||_{\tilde p \rightarrow \tilde q} \leq C''_{\epsilon_1} 2^{-j \tilde \epsilon} \]
where
\[ \frac{1}{\tilde p} := \frac{\epsilon}{\epsilon + 2 \epsilon_1 } \frac{1}{p_0} + \frac{2 \epsilon_1 }{\epsilon + 2 \epsilon_1 } \frac{1}{2} , \qquad  \frac{1}{\tilde q} := \frac{\epsilon}{\epsilon + 2 \epsilon_1 } \frac{1}{q_0} + \frac{2 \epsilon_1 }{\epsilon + 2 \epsilon_1 } \frac{1}{2},  \]
and
\[ \tilde \epsilon := \frac{\epsilon}{\epsilon + 2 \epsilon_1 } (- \epsilon_1 ) + \frac{2 \epsilon_1 }{\epsilon + 2 \epsilon_1 } \epsilon = \frac{\epsilon \epsilon_1 }{\epsilon + 2 \epsilon_1 } > 0. \]
Summing over $j$ gives that $T$ must be bounded from $L^{\tilde p}$ to $L^{\tilde q}$. As $\epsilon_1  \rightarrow 0$, the pair $(\tilde p, \tilde q)$ tends to $(p_0,q_0)$, so interpolating between this $L^{\tilde p} \rightarrow L^{\tilde q}$ inequality and the trivial $L^p$-boundedness of $T$ established by Proposition \ref{lpbound} implies boundedness of $T$ whenever $(1/p,1/q)$ belongs to the promised open triangle in $[0,1]^2$. Finally, if $\eta(x,y)$ is any merely continuous cutoff function with support of $\eta(x,\gamma_t(x))$ close enough to $(x_*,t_*)$ that $\eta_1(x,\gamma_t(x))$ is bounded below there, then the inequality $|\eta(x,\gamma_t(x))| \leq C |\eta_1(x,\gamma_t(x))|$ for some finite $C$ implies (just as was observed earlier) that the operator \eqref{theop} defined using this cutoff is bounded for all pairs $(p,q)$ that were just established for the cutoff function $\eta_1$.
\end{proof}

\subsection{Proof of sufficiency in the $C^\infty$ case}
\label{specapprox}

The proof of Theorem \ref{characterthm} can now be completed by applying Theorem \ref{approxthm}.

Suppose $Q$ at the point $(x_*,t_*)$ is nondegenerate. Proposition \ref{hormander} (from Section \ref{nondegensec}) implies that the double fibration $\pi_1(x,t) := \gamma_t(x)$ and $\pi_2(x,t) := x$ satisfies the H\"{o}rmander condition at $(x_*,t_*)$ using at most first commutators, so the first hypothesis of Theorem \ref{approxthm} holds. 
By Corollary \ref{stability} (in the latter part of Section \ref{nondegensec} as well), there exists a ball centered at $(x_*,t_*)$ and a constant $c > 0$ such that if $||\phi - \phi_j||_{C^3(B)}$ is sufficiently small, then
\begin{equation} \begin{split} \mathcal{Q}_j & |_{(x,t)} [ \{u_i\}_{i=1}^{\de}, \{v_i\}_{i=1}^k, \{w_i\}_{i=1}^d] \\ & \geq c |\det \{u_i\}_{i=1}^{\de}|^{\frac{s}{\de}} |\det \{v_i\}_{i=1}^k|^{\frac{s}{k}} |\det \{w_i\}_{i=1}^{d}|^{\frac{s}{d}} \end{split} \label{uniformapx} \end{equation} 
for all $(x,t) \in B$, where $Q_j$ is defined via \eqref{qdef} using $\phi_j$ and $s = dk/n$. Here the constant $c$ is uniform in $x,t$, and $j$.

By the multivariate Jackson's Theorem, e.g., Theorem 2 of Bagby, Bos, and Levenberg \cite{bbl2002}, if $\phi$ is class $C^m$ on a neighborhood of a fixed Euclidean ball $B$ centered at $(x_*,t_*) \in U$, then for each integer $D$ and each $i \in \{1,\ldots,k\}$, there is a polynomial $p^i_D(x,t)$ of degree at most $D$ such that
\begin{equation} \sup_{(x,t) \in B} | \partial^\alpha_{x,t} \phi^i(x,t) - \partial^\alpha_{x,t} p^i_D(x,t)| \leq C ||\phi^i||_{C^m(B)} D^{-m+|\alpha|} \label{papprox} \end{equation}
for all $|\alpha| \leq \min\{m,D\}$, where $C$ is a constant which is independent of $\alpha$ and $\phi$ (depending only on $m$, the ball $B$, and the usual parameters $n,\de$).  

Fix any positive $\epsilon$ and let $m$ be any integer larger than $\epsilon^{-1}$. For each nonnegative integer $j$, choose $D$ as small as possible so that $C ||\phi^i||_{C^m(B)} D^{-m} < k^{-1} 2^{-j}$ for each $i$, i.e., let 
\[ \max_i (C ||\phi^i||_{C^m(B)} k 2^j)^{1/m} < D \leq \max_i (C ||\phi^i||_{C^m(B)} k 2^j)^{1/m} + 1. \]
 Let $\phi_j(x,t)$ be the polynomial function $(\phi^1_D(x,t),\ldots,\phi^k_D(x,t))$ and $\gamma^j_t(x) := (t,\phi_j(x,t))$. By the choice of $D$, $|\phi(x,t) - \phi^j(x,t)| < 2^{-j}$ on the ball $B$. If this $D$ is by itself does not guarantee sufficient smallness of $||\phi - \phi_j||_{C^3(B)}$ to yield \eqref{uniformapx}, then it is possible to increase $D$ by at most a fixed amount independent of $j$ so that \eqref{uniformapx} does indeed hold. By the choice of $m$, it follows that $D \leq C' 2^{\epsilon j}$ for some constant $C'$ which is independent of $j$.  It follows that
\begin{equation} \sup_{(x,t) \in B} |\gamma_t(x) - \gamma_t^j(x)| \leq 2^{-j} \label{closeness} \end{equation}
as required by Theorem \ref{approxthm}. By Theorem \ref{fromtest} and \eqref{sublevgoal}, when $\eta$ is such that $\eta(x,\gamma_t(x))$ is supported sufficiently close to $(x_*,t_*)$, it will follow that
\[ T_j f(x) := \int f(\gamma^j_t(x)) \eta(x,\gamma_t(x)) dt \]
maps $L^{p_b}$ to $L^{q_b}$ inequality with norm that grows at most like some fixed power of $D$. This can be interpolated with the trivial $L^p$-boundedness of $T$ given by Proposition \ref{lpbound} (where one notes that the $L^p$ norm of each $T_j$ will be uniform in $j$ because the constants were shown in Proposition \ref{lpbound} to depend only on the first derivatives of $\phi$ and size of $\eta$ on its support) to conclude that 
\[ ||T_j f||_{L^q(\R^n)} \leq C D^M ||f||_{L^p(\R^{\en})} \leq C (C' 2^{\epsilon j})^M ||f||_{L^p(\R^{\en})} \]
for any pair $(1/p,1/q)$ in the triangle with vertices $(0,0), (1,1)$, and $(1/p_b,1/q_b)$. Because the triangle is open, Theorem \ref{approxthm} may now be invoked for any $(1/p_0,1/q_0)$ in the triangle, and taking the union over all such pairs establishes boundedness of $T$ at all interior points of this same triangle. This completes the proof of Theorem \ref{characterthm} in the $C^\infty$ case\footnote{Note also that a slight modification of this argument can be used to attain endpoint $L^{p_b}$--$L^{q_b}$ boundedness when $\phi$ has $||\phi||_{C^m(B)}^{1/m}$ bounded as a function of $m$. In this case, one can simply take $m > \ln j$ and take a limit as $j \rightarrow \infty$ to let $T_j$ approximate $T$.}.

\bibliography{mybib}

\end{document}